\documentclass[a4paper]{amsart}
\usepackage[utf8]{inputenc}
\usepackage{amssymb}
\usepackage{amscd}
\usepackage[all]{xy}
\usepackage{nicefrac,mathtools,enumitem}
\usepackage{microtype}
\usepackage[pdftitle={Generalized fixed point algebras for coactions of locally compact quantum groups},
  pdfauthor={Alcides Buss},
  pdfsubject={Mathematics}
]{hyperref}
\usepackage[lite]{amsrefs}
\newcommand*{\MRref}[2]{ \href{http://www.ams.org/mathscinet-getitem?mr=#1}{MR \textbf{#1}}}
\newcommand*{\arxiv}[1]{\href{http://www.arxiv.org/abs/#1}{arXiv: #1}}

\renewcommand{\PrintDOI}[1]{\href{http://dx.doi.org/#1}{DOI #1}%
  \IfEmptyBibField{pages}{, (to appear in print)}{}}

\usepackage{xcolor}

\numberwithin{equation}{section}
\newtheorem{theorem}[equation]{Theorem}
\newtheorem{lemma}[equation]{Lemma}
\newtheorem{proposition}[equation]{Proposition}
\newtheorem{corollary}[equation]{Corollary}
\theoremstyle{definition}
\newtheorem{definition}[equation]{Definition}
\theoremstyle{remark}
\newtheorem{example}[equation]{Example}
\newtheorem{remark}[equation]{Remark}
\newtheorem{question}[equation]{Question}

\DeclareMathOperator{\spn}{span}
\DeclareMathOperator{\cspn}{\overline{span}}
\DeclareMathOperator{\ran}{ran}
\DeclareMathOperator{\Fix}{Fix}
\DeclareMathOperator{\Prim}{Prim}
\DeclareMathOperator*{\dom}{dom}

\newcommand*{\into}{\hookrightarrow}
\newcommand*{\cont}{\textup C}
\newcommand*{\contc}{\cont_c}

\newcommand*{\slc}{s}
\newcommand*{\<}{\langle}
\renewcommand{\>}{\rangle}
\newcommand*{\Real}{\mathbb{R}}
\newcommand*{\Nat}{\mathbb{N}}
\newcommand*{\C}{\mathbb{C}}
\newcommand*{\R}{\mathcal{R}}
\newcommand*{\M}{\mathcal{M}}
\newcommand*{\E}{\mathcal{E}}
\newcommand*{\F}{\mathcal{F}}
\newcommand*{\G}{\mathcal{G}}
\newcommand*{\T}{\mathcal{T}}
\newcommand*{\Hc}{\mathcal{H}}
\newcommand*{\B}{\mathcal{B}}

\newcommand*{\I}{\mathcal{I}}
\newcommand*{\Ls}{\mathcal L}
\newcommand*{\K}{\mathcal{K}}
\newcommand*{\La}{\Lambda}
\newcommand*{\la}{\lambda}
\newcommand*{\dualla}{\hat\lambda}
\newcommand*{\laop}[1]{\la_{#1}^{\!^\op}}
\newcommand*{\f}{\varphi}
\newcommand*{\rc}{\sim^{\!\!\!\!\!\!^{rc}}}
\newcommand*{\Dt}{\Delta}
\newcommand*{\co}{\gamma}
\newcommand*{\dtg}{\delta}
\newcommand*{\mo}{\delta}
\newcommand*{\sbe}{\subseteq}
\newcommand*{\dd}{\,\mathrm{d}}
\newcommand*{\dualG}{\widehat{\G}} 
\newcommand*{\dualg}{\widehat{G}} 
\newcommand*{\opG}{\G^{^\op}} 

\newcommand*{\cdualG}{\dualG^{\,^{\comm}}} 
\newcommand*{\dualJ}{\hat{J}} 

\newcommand*{\ket}[1]{\lvert#1\rangle}
\newcommand*{\bra}[1]{\langle#1\rvert}
\newcommand*{\braket}[2]{\langle#1\!\mid\!#2\rangle}
\newcommand*{\kket}[1]{{\lvert#1\mathclose{\rangle\!\rangle}}}
\newcommand*{\bbra}[1]{{\mathopen{\langle\!\langle}#1\rvert}}
\newcommand*{\bbraket}[2]{\mathopen{\langle\!\langle}#1\!\mid\!#2\mathclose{\rangle\!\rangle}}
\newcommand*{\KKET}[1]{{\bigl\lvert#1\mathclose{\bigl\rangle\!\bigr\rangle}}}

\newcommand*{\BBRAKET}[2]{\mathopen{\bigl\langle\!\bigl\langle}#1\bigl|#2\mathclose{\bigr\rangle\!\bigr\rangle}}

\newcommand*{\defeq}{\mathrel{\vcentcolon=}}

\newcommand*{\cstar}{\texorpdfstring{$C^*$\nobreakdash-\hspace{0pt}}{*-}}
\newcommand*{\Star}{\texorpdfstring{$^*$\nobreakdash-\hspace{0pt}}{*-}}

\newcommand*{\id}{{\mathrm{id}}}
\newcommand*{\su}{{\mathrm{su}}}
\newcommand*{\com}{{\mathrm{c}}}
\newcommand*{\sco}{{\mathrm{sc}}}
\newcommand*{\si}{{\mathrm{si}}}
\newcommand*{\ii}{{\mathrm{i}}}
\newcommand*{\op}{{\mathrm{op}}}
\newcommand*{\comm}{{\mathrm{c}}} 
\newcommand*{\red}{\mathrm{r}} 

\newcommand*{\rot}[1]{\underset{#1}{\otimes}}

\begin{document}
\title[Generalized fixed point algebras
]{Generalized fixed point algebras for coactions of locally compact quantum groups
}

\author{Alcides Buss}
\email{alcides@mtm.ufsc.br}
\address{Departamento de Matem\'atica\\
 Universidade Federal de Santa Catarina\\
 88.040-900 Florian\'opolis-SC\\
 Brazil}

\begin{abstract}
We extend the construction of generalized fixed point algebras to the setting of locally compact quantum groups -- in the sense of Kustermans and Vaes -- following the treatment of Marc Rieffel, Ruy Exel and Ralf Meyer in the group case. We mainly follow Meyer's approach analyzing the constructions in the realm of equivariant Hilbert modules.

We generalize the notion of continuous square-integrability, which is exactly what one needs in order to define generalized fixed point algebras. As in the group case, we prove that there is a correspondence between continuously square-integrable Hilbert modules over an equivariant \cstar{}algebra $B$ and Hilbert modules over the reduced crossed product of $B$ by the underlying quantum group. The generalized fixed point algebra always appears as the algebra of compact operators of the associated Hilbert module over the reduced crossed product.
\end{abstract}

\thanks{This article is based on the author doctoral dissertation under supervision of Siegfried Echterhoff and Ralf Meyer. It has been supported by CNPq and CAPES}
\maketitle

\section{Introduction}

Let $G$ be a locally compact group and let $X$ be a $G$-space, that is, a locally compact Hausdorff space with a continuous action of $G$. The action of $G$ on $X$ is called \emph{proper} if the map $G\times X\to X\times X$, $(t,x)\mapsto (t\cdot x,x)$ is proper in the sense that inverse images of compact subsets are again compact.

Properness is a concept that enables properties of actions of non-compact groups to resemble those of compact groups. Proper actions have many nice properties. One of the most important ones is the fact that the orbit space $G\backslash X$ is again a locally compact Hausdorff space.

A program to extend this notion to the setting of noncommutative dynamical systems, that is, groups acting on \cstar{}algebras, was initiated by Marc Rieffel in \cite{Rieffel:Proper}. His idea relies on one basic result, namely, the fact that for a proper $G$-space $X$, the commutative \cstar{}algebra $\cont_0(G\backslash X)$ associated to the orbit space is Morita equivalent to an ideal in the reduced crossed product $\cont_0(X)\rtimes_\red G$, where we let $G$ act on $\cont_0(X)$ in the usual way. This ideal is the entire crossed product if and only if the action is free.

The imprimitivity bimodule implementing the Morita equivalence between the algebra $\cont_0(G\backslash X)$ and the ideal in the crossed product turns out to be a suitable completion of the space $\cont_c(X)$ of compactly supported continuous functions on $X$. Based on this fact, Rieffel called a (not necessarily commutative) $G$-\cstar{}algebra, that is, a \cstar{}algebra $A$ with a (strongly) continuous action of $G$, \emph{proper} if there exists a dense $*$-subalgebra $A_0$ of $A$ with some suitable properties (modeled on $\cont_c(X)$ in the commutative case) such that $A_0$ can be completed to an imprimitivity bimodule between the \emph{generalized fixed point algebra} $\Fix(A_0)$ -- obtained from averaging elements of $A_0$ along the given action -- and a suitable ideal $\I(A_0)$ in the reduced crossed product algebra $A\rtimes_\red G$. Needless to say, in the commutative case one takes $A_0=\contc(X)$ to obtain $\Fix(A_0)=\cont_0(G\backslash X)$. Thus $\Fix(A_0)$ is a noncommutative version of the orbit space.
If, in addition, $\I(A_0)$ is the whole reduced crossed product, the action is called \emph{saturated}. Saturation is therefore a noncommutative manifestation of freeness. For other noncommutative notions of freeness (for actions of finite/compact groups), we refer to \cite{Phillips:Freeness_actions_finite_groups}.

Rieffel has further investigated his first definition of proper actions in a second work \cite{Rieffel:Integrable_proper}.
He came out with another possible notion for noncommutative proper actions, the so-called \emph{integrable actions}. These include the proper actions previously defined. In order to explain better this second definition, let us say that $G$ acts on $A$ via an action $\alpha$.
A positive element $a\in A$ is called \emph{integrable} if there exists $b$ in the multiplier algebra $\M(A)$ of $A$ such
that for any positive linear functional $\theta$ on $A$, the function $t\mapsto \theta(\alpha_t(a))$ is integrable in the ordinary sense,
and $\int_G\theta(\alpha_t(a))\dd{t}=\theta(b)$. In this case, it is natural to write $b=\int\alpha_t(a)\dd{t}$. However, we should point out that this integral does not converge in Bochner's sense, unless $G$ is compact or $a=0$, because the integrand has constant norm. The $G$-\cstar{}algebra $A$ is called \emph{integrable} if the space of integrable elements (that is, elements of $A$ that can be written as a linear combination of positive integrable elements) is dense in $A$.

Integrability is closely related to the notion of properness discussed previously. Indeed, Rieffel proved in \cite{Rieffel:Integrable_proper} that
if $A$ is proper, then it is also integrable. Furthermore, he also proved that in the commutative case $A=\cont_0(X)$, where $X$ is some locally compact $G$-space, $A$ is integrable if and only if $X$ is a proper $G$-space. Moreover, in this case $\cont_c(X)$ consists of integrable elements and the generalized fixed point algebra $\cont_0(G\backslash X)$ is generated by the averages $\int\alpha_t(a)\dd{t}$ with $a\in \cont_c(X)$.

However, it was not clear to Rieffel in \cite{Rieffel:Integrable_proper} whether, given an integrable $G$-\cstar{}algebra $A$, there is a dense subspace $A_0\sbe A$ yielding the properness of $A$ (as defined in \cite{Rieffel:Proper}) and hence the desired generalized fixed point algebra. He defined a ``big generalized fixed point algebra"  generated by averages that worked in the commutative case, but, in general, it was really too big to be Morita equivalent to an ideal in the reduced crossed product. As explained by Ruy Exel in \cite{Exel:Unconditional,Exel:SpectralTheory}, the problem appears already in the case of Abelian groups.	

Exel was more interested in another point, namely, to characterize the $G$-\cstar{}algebras appearing as dual
actions on cross-sectional \cstar{}algebras of \emph{Fell bundles} (also called \cstar{}algebraic bundles; see \cite{Doran-Fell:Representations_2}).

Given an abelian group $G$ with Pontrjagin dual $\dualg$, the main result of \cite{Exel:Unconditional} states that the cross-sectional \cstar{}algebra $C^*(\B)$ of a Fell bundle $\B=\{\B_x\}_{x\in \dualg}$ over $\dualg$ is proper in Rieffel's sense if we equip it with the dual action of $G$. Conversely, Exel proved in \cite{Exel:SpectralTheory} that a proper $G$-\cstar{}algebra can be realized as the cross-sectional \cstar{}algebra of some Fell bundle over $\dualg$. And the associated generalized fixed point algebra can be identified with the unit fiber of the Fell bundle.
To prove this result Exel defined in \cite{Exel:SpectralTheory} an interesting relation between integrable elements called \emph{relative continuity}.
Moreover, this relation turns out to be equivalent to the requirement that some natural operators belong to the crossed product algebra \cite[Theorem~7.5]{Exel:SpectralTheory}. Due to this fact, if relative continuity is present, then it is possible to construct a generalized fixed point algebra which is Morita equivalent to an ideal in the crossed product \cite[Section 9]{Exel:SpectralTheory}.

Thus relative continuity is closely related to the notion of proper action defined by Rieffel in \cite{Rieffel:Proper} and, in particular, this is a sufficient condition to find the generalized fixed point algebra that Rieffel was looking for in \cite{Rieffel:Integrable_proper}.

However, some things were not clear in \cite{Exel:SpectralTheory} (see Questions~9.4, 9.5 and 11.16) and essentially these were the same problems that Rieffel met in \cite{Rieffel:Integrable_proper}:
\begin{question}\label{q:RelativelyContinuousSubspaces}
{\bf (1)} Suppose that $A$ is an integrable $G$-\cstar{}algebra. Is there a dense, relatively continuous subspace of $A$?

{\bf (2)} Are the generalized fixed point algebras associated to two different (say, maximal) relatively continuous subspaces always the same?
\end{question}
The answers to these questions were given by Ralf Meyer in \cite{Meyer:Generalized_Fixed} where he also generalized the notion of relative continuity to non-Abelian groups.

Meyer introduced in \cite{Meyer:Equivariant} the notion of square-integrability in the setting
of group actions on Hilbert modules and proved that the square-integrable actions on (countably generated) Hilbert modules $B$-modules are exactly those satisfying an equivariant version of the Kasparov Stabilization Theorem. Roughly speaking, this means that all such Hilbert modules are direct summands of countably many copies of $L^2(G,B)$, where $B$ is some fixed $G$-\cstar{}algebra.

The main ingredient towards Meyer's results is the construction of the so-called bra-ket operators. Suppose that $\E$ is a Hilbert $B,G$-module, that is, a Hilbert $B$-modules with a continuous action $\gamma$ of $G$ compatible with the action $\beta$ of $G$ on $B$. Given an element $\xi\in \E$, Meyer defined the following maps in \cite{Meyer:Equivariant,Meyer:Generalized_Fixed}:
\begin{alignat*}{2}
\bbra{\xi} &\colon \E\to\cont_b(G,B),&\qquad
(\bbra{\xi}\eta)(t) &\defeq \braket{\gamma_t(\xi)}{\eta},\\
\kket{\xi} &\colon \cont_c(G,B)\to\E,&\qquad
\kket{\xi} f &\defeq \int_G \gamma_t(\xi)\cdot f(t)\dd{t}.
\end{alignat*}
An element $\xi\in\E$ is said \emph{square-integrable} if $\bbra{\xi}\eta\in L^2(G,B)$ for all $\eta\in\E$.  In this case, $\bbra{\xi}$ becomes an
adjointable operator $\E\to L^2(G,B)$, whose adjoint extends $\kket{\xi}$ to an adjointable operator $L^2(G,B)\to\E$; we denote
these extensions by $\bbra{\xi}$ and $\kket{\xi}$ as well. Conversely, if $\kket{\xi}$ extends to an \emph{adjointable} operator $L^2(G,B)\to\E$, then~$\xi$ is square-integrable. We say that $\E$ is square-integrable if the space $\E_\si$ of square-integrable elements is dense in $\E$. This notion is equivalent to Rieffel's integrability as previously defined: an action on a \cstar{}algebra $A$ is integrable if and only if $A$ is square-integrable as a $A,G$-Hilbert module.

The basic example of a square-integrable Hilbert $B,G$-module is $L^2(G,B)$ endowed with the diagonal action $\beta\otimes\la$, where we identify $L^2(G,B)\cong B\otimes L^2(G)$ and write $\la$ for the left regular representation of $G$. Moreover, one can prove that direct sums or $G$-invariant direct summands of square-integrable Hilbert $B,G$-modules are again square-integrable. In particular, $\Hc_B:=\bigoplus_{n\in \Nat} L^2(G,B)$ is square-integrable, and the stabilization theorem in  \cite{Meyer:Equivariant}  says that it is the ``universal example" in the sense that it contains (as direct summands) all the other countably generated square-integrable Hilbert $B,G$-modules.

Now we turn our attention to the second work of Meyer \cite{Meyer:Generalized_Fixed}. Given square-integrable elements $\xi,\eta\in \E$, we write $\bbraket{\xi}{\eta}\defeq \bbra{\xi}\circ\kket{\eta}$ and $\kket{\xi}\bbra{\eta}:=\kket{\xi}\circ\bbra{\eta}$. A short computation shows that the operators $\bbra{\xi}:\E\to L^2(G,B)$ and $\kket{\eta}:L^2(G,B)\to\E$ are $G$-equivariant. In particular, so are the operators $\bbraket{\xi}{\eta}\in \Ls\bigl(L^2(G,B)\bigr)$ and $\kket{\xi}\bbra{\eta}\in \Ls(\E)$, where for any two Hilbert $B$-modules $\E_1$ and $\E_2$, we denote by $\Ls(\E_1,\E_2)$ the space of all adjointable operators $\E_1\to\E_2$. We also write $\Ls^G(\E_1,\E_2)$ for the subspace of $G$-equivariant operators. Note that the space of $G$-equivariant operators $\Ls^G(\E)$ is (canonically isomorphic to) the multiplier fixed point algebra $\M_1\bigl(\K(\E)\bigr)$ and this should possibly contain a generalized fixed point algebra. This indicates that the operators $\kket{\xi}\bbra{\eta}$ may generate a candidate for the generalized fixed point algebra. On the other hand, the reduced crossed product algebra $B\rtimes_\red G$ has a canonical realization as a \cstar{}subalgebra of $\Ls^G\bigl(L^2(G,B)\bigr)$. Our basic principle is that a generalized fixed point algebra should be Morita equivalent
to some ideal in the reduced crossed product. This naturally leads us to the following definition (\cite[Definition~6.1]{Meyer:Generalized_Fixed}):

\begin{definition}\label{544}
A subset $\R\sbe\E$ consisting of square-integrable elements is called \emph{relatively continuous} if
$\bbraket{\R}{\R}:=\{\bbraket{\xi}{\eta}:\xi,\eta\in \R\}\sbe B\rtimes_\red G.$
\end{definition}
Given a relatively continuous subset $\R\sbe\E$, we define
$$\F(\E,\R):=\cspn\bigl(\kket{\R}\circ B\rtimes_\red G\bigr)\sbe\Ls^G\bigl(L^2(G,B),\E\bigr).$$
By definition of relative continuity, $\F(\E,\R)$ is a \emph{concrete} Hilbert $B\rtimes_\red G$-module in the
sense that it is a closed subspace of $\Ls^G\bigl(L^2(G,B),\E\bigr)$ satisfying
$$\F(\E,\R)\circ B\rtimes_\red G\sbe\F(\E,\R)\quad\mbox{and}\quad\F(\E,\R)^*\circ\F(\E,\R)\sbe B\rtimes_\red G.$$
A concrete Hilbert $B\rtimes_\red G$-module can be regarded as an abstract Hilbert $B\rtimes_\red G$-module in the obvious way. Conversely,
any abstract Hilbert $B\rtimes_\red G$-module $\F$ can be represented in an essentially unique way in $\Ls^G(L^2(G,B),\E_\F)$, where
$\E_\F$ is the balanced tensor product $\F\otimes_{B\rtimes_\red G} L^2(G,B)$ (\cite[Theorem~5.3]{Meyer:Generalized_Fixed}).

The algebra of compact operators on $\F(\E,\R)$ is canonically isomorphic to the closed linear span of $\F(\E,\R)\circ\F(\E,\R)^*\sbe\Ls^G(\E)$
which we denote by $\Fix(\E,\R)$ and call the generalized fixed point algebra associated to the pair $(\E,\R)$. It is therefore Morita equivalent
to the ideal $\I(\E,\R):=\cspn\bigl(\F(\E,\R)^*\circ\F(\E,\R)\bigr)\sbe B\rtimes_\red G$ and $\F(\E,\R)$ may be viewed as an imprimitivity Hilbert
bimodule implementing this Morita equivalence.

In general, there are many relatively continuous subspaces $\R\sbe\E$ yielding the same Hilbert $B\rtimes_\red G$-module $\F=\F(\E,\R)$ (and hence the same generalized fixed point algebra). However, we can control this by imposing some more natural conditions on $\R$. We say that $\R$ is
\emph{complete} if it is a $G$-invariant $B$-submodule of $\E$ (that is, $\gamma_t(\R)\sbe\R$ and $\R\cdot B\sbe\R$) which is closed with respect the $\si$-norm: $\|\xi\|_\si:=\|\xi\|+\|\kket{\xi}\|$. The \emph{completion} of $\R$ is the smallest complete subspace $\R_\com$ containing $\R$. If $\R$ is complete, then the Hilbert module $\F(\E,\R)$ is just the closure of $\kket{\R}$ and, as a consequence, the generalized fixed point algebra
$\Fix(\E,\R)$ and the ideal $\I(\E,\R)$ are just the closed linear spans of $\kket{\R}\bbra{\R}$ and $\bbraket{\R}{\R}$, respectively.
Moreover, we always have $\F(\E,\R)=\F(\E,\R_\com)$ for any relatively continuous subset $\R$ and hence we can replace $\R$ by its completion
to get the same results. If we restrict attention to complete subspaces, then $\R$ is uniquely determined by the Hilbert module $\F(\E,\R)$ by the following result (\cite[Theorem~6.1]{Meyer:Generalized_Fixed}):

\begin{theorem}\label{548} Let $\E$ be a Hilbert $B,G$-module. Then the map $\R\mapsto \F(\E,\R)$ is a bijection between complete,
relatively continuous subspaces $\R\sbe\E$ and concrete Hilbert $B\rtimes_\red G$-modules $\F\sbe\Ls^G\bigl(L^2(G,B),\E\bigr)$.
The inverse map is given by the assignment $\F\mapsto \R_\F:=\{\xi\in \E_\si:\kket{\xi}\in \F\}$. Moreover, $\R$ is dense in $\E$ if and
only if $\F(\E,\R)$ is \emph{essential}, meaning that $\cspn\F(\E,\R)(L^2(G,B))=\E$.
\end{theorem}

A \emph{continuously square-integrable Hilbert $B,G$-module} is a pair $(\E,\R)$ consisting of a Hilbert $B,G$-module $\E$ and a dense, complete, relatively continuous subspace $\R\sbe\E$. This class forms a category if we take \emph{$\R$-continuous} $G$-equivariant operators as morphisms, that is, $G$-equivariant operators that are compatible with the relatively continuous subspaces in the obvious way (\cite{Meyer:Generalized_Fixed}).

The construction $(\E,\R)\mapsto \F(\E,\R)$ is a functor from the category of continuously square-integrable Hilbert $B,G$-modules to
the category of Hilbert $B\rtimes_\red G$-modules with morphisms as usual. Theorem~\ref{548} and the fact that any abstract Hilbert module
can be realized as a concrete one imply that $(\E,\R)\mapsto \F(\E,\R)$ induces a bijection between the isomorphism classes. Moreover, this construction is natural and yields an equivalence between the respective categories (\cite[Theorem~6.2]{Meyer:Generalized_Fixed}). Using this correspondence, Meyer could give a negative answer to Question~\ref{q:RelativelyContinuousSubspaces} analyzing the subtle difference between square-integrable and continuously square-integrable representations on Hilbert spaces.

\subsection{The quantum case: our main results}
The main goal of this paper is to generalize the concepts and results above to the setting of locally compact quantum groups
in the sense of Kustermans and Vaes \cite{Kustermans-Vaes:LCQG}.

In \cite{Buss-Meyer:Square-integrable} we defined the notion of square-integrable coactions of a locally compact quantum group $\G$
on \cstar{}algebras and Hilbert modules generalizing the notion of integrable (or proper) actions of groups mentioned above. The basic ingredient here is the existence of a Haar weight on $\G$ which naturally leads us to the setting of locally compact quantum groups. One basic example is the comultiplication of $\G$ itself which is always integrable, for any locally compact quantum group. Moreover, given coactions $\co_A$ and $\co_B$ of $\G$ on \cstar{}algebras $A$ and $B$, respectively, and given a nondegenerate $\G$-equivariant $*$-homomorphism $\pi:A\to\M(B)$, if $\co_A$ is integrable, then so is $\co_B$. As a consequence, we get that any dual coaction is integrable. In particular, if $\G$ is regular, the dual coaction of $\G$ on the algebra of compact operators $\K:=\K\bigl(L^2(\G)\bigr)$ is integrable, where $L^2(\G)$ denotes the $L^2$-Hilbert space associated to
the Haar weight of $\G$. Moreover, even if $\G$ is not regular, $\K$ always has a canonical coaction of $\G$, and it is always integrable. More generally, we can always furnish the tensor product $A\otimes\K$ with a coaction of $\G$ (whenever $A$ has a coaction of $\G$) and this coaction is also always integrable. In particular, any coaction is Morita equivalent to an integrable coaction.

As in the group case, the basic example of a square-integrable Hilbert $B,\G$-module is $B\otimes L^2(\G)$ endowed with a canonical coaction of $\G$. The main result in \cite{Buss-Meyer:Square-integrable} is the quantum version of the equivariant Kasparov Stabilization Theorem (see \cite[Theorem~6.1]{Buss-Meyer:Square-integrable}).

Our main goal in this paper is to give the definition of relative continuity and generalized fixed point algebras in the setting of coactions of locally compact quantum groups on Hilbert modules. In fact, once we have the bra-ket operators, the definitions are exactly the same as in the group case. Given a relatively continuous subset $\R$ in a Hilbert $B,\G$-module $\E$, we define, as in the group case, a concrete Hilbert module $\F(\E,\R)$ over the reduced crossed product $B\rtimes_\red\cdualG$ (the reason for this notation will be clear later). Again, the algebra of compact operators on $\F(\E,\R)$ is (canonically isomorphic to) the generalized fixed point algebra $\Fix(\E,\R)$ and therefore, it is Morita equivalent to the ideal $\I(\E,\R):=\cspn\bigl(\F(\E,\R)^*\circ\F(\E,\R)\bigr)$ in $B\rtimes_\red\cdualG$. If $\I(\E,\R)$ is equal to $B\rtimes_\red\cdualG$, then we say that $\R$ is \emph{saturated}.

One of the first examples that we analyze is the coaction of $\G$ on itself via the comultiplication. We already mentioned that this coaction is always integrable, but here is where the first difference appears: we prove that there is a non-zero relatively continuous subset of $\G$ if and only if $\G$ is semi-regular. Moreover, there is a saturated relatively continuous subset of $\G$ if and only if $\G$ is regular.

If $\G$ is compact, then any subset $\R\sbe\E$ is relatively continuous and the generalized fixed point algebra $\Fix(\E)=\Fix(\E,\E)$ is the usual fixed point algebra which is therefore Morita equivalent to an ideal in the reduced crossed product.

The most important example is the Hilbert $B,\G$-module $B\otimes L^2(\G)$. We prove that we always can find a dense, relatively continuous subspace $\R_0\sbe B\otimes L^2(\G)$ such that $\F\bigl(B\otimes L^2(\G),\R_0\bigr)=B\rtimes_\red\cdualG$. In particular, this shows that reduced crossed products appear as generalized fixed point algebras. This is a basic observation in the group case.

We also analyze some completeness conditions of relatively continuous subsets. The possible non-co-amenability of the quantum group brings about some technical problems at this point. As in the group case, we can define complete subspaces, but it turns out that completeness alone is not
enough in general. We need an extra condition that we call \emph{\slc-completeness}. This is a sort of ``slice map property" and this
is where the script ``\slc{}" comes from. If $\G$ is co-amenable, then this condition reduces to completeness. Having these completeness conditions
we can then define a \emph{continuously square-integrable Hilbert $B,\G$-module} to be a pair $(\E,\R)$, where $\E$ is Hilbert $B,\G$-modules, and $\R$ is a dense, complete, relatively continuous subspace. If, in addition, $\R$ is \slc-complete then we say that
$(\E,\R)$ is an \slc-continuously square-integrable Hilbert $B,\G$-module.

One of our main results is a quantum version of Meyer's Theorem~\ref{548} above. If we replace completeness by \slc-completeness, then the result remains almost unchanged in the quantum setting (see Theorem~\ref{289} below). As in the group case, this implies that the construction $(\E,\R)\mapsto\F(\E,\R)$ is an equivalence between the categories of \slc-continuously square-integrable Hilbert $B,\G$-modules and Hilbert modules over the reduced crossed product $B\rtimes_\red\cdualG$. The inverse construction is given by the assignment $\F\mapsto (\E_\F,\R_\F)$, where $\E_\F:=\F\otimes_{B\rtimes_\red\cdualG}\bigl(B\otimes L^2(\G)\bigr)$ and $\R_\F$ is the \slc-completion of the algebraic tensor product $\F\odot_{B\rtimes_\red\cdualG}\R_0$.

Several applications and further developments of the theory of generalized fixed point algebras have been already explored in the group case since Rieffel's pioneering work on proper actions \cite{Rieffel:Proper}. To mention just a few, we refer to \cite{Huef-Raeburn-Williams:FunctorialityGFPA,Huef-Kaliszewski-Raeburn-Williams:Naturality_Rieffel,Huef-Kaliszewski-Raeburn-Williams:Naturality_Symmetric,Kaliszewski-Muhly-Quigg-Williams:Fell_bundles_and_imprimitivity_theoremsII,Echterhoff.Williams:StructureCrossedProducts,Echterhoff-Emerson:Structure_proper}. We expect that in the future some of these applications will also be available in the quantum world.

\section{Preliminaries and notational conventions}

Most of our notations will be as in \cite{Buss-Meyer:Square-integrable}. For reader's convenience, we review some of these here.
A (reduced) locally compact quantum group -- in the sense of Kustermans and Vaes \cite{Kustermans-Vaes:LCQG} -- will be generally denoted by $\G$.
This is, therefore, a \cstar{}algebra endowed with a comultiplication $\Delta\colon \G\to \M(\G\otimes\G)$ and left and right invariant faithful Haar weights $\f$ and $\psi$, respectively, and all these data is required to satisfy several technical conditions (see \cite{Kustermans-Vaes:LCQG} for details). The symbol $\otimes$ will always mean the minimal tensor product between \cstar{}algebras in this paper, and sometimes also denote (internal or external) tensor products between Hilbert modules. The Haar weights on $\G$ are supposed to be lower semi-continuous and they can be uniquely extended to strictly lower semi-continuous weights on the multiplier \cstar{}algebra $\M(\G)$. We use the same letters $\f$ and $\psi$ to denote these extensions. The domain of $\f$ will be denoted $\dom(\f)$. This is a strictly dense \Star{}subalgebra of $\M(\G)$ which is defined as the linear space of all positive elements $x\in \M(\G)_+$ with $\f(x)<\infty$; of course, we use the same kind of notation for $\psi$ or any other unbounded linear map. We shall fix a GNS-construction associated to $\f$: this is a triple $(H,\iota,\La)$, where $H$ is a Hilbert space, $\iota\colon\G\to \Ls(H)$ is a nondegenerate \Star{}homomorphism (with extension to $\M(\G)$ also denoted $\iota$) and an unbounded linear map $\La\colon\dom(\La)\defeq\{x\in \M(\G):\f(x^*x)<\infty\}\sbe \M(\G)\to H$ with dense range satisfying $\braket{\La(x)}{\La(y)}=\f(x^*y)$ for all $x,y\in \dom(\La)$ and $\La(xy)=\iota(x)\La(y)$ for all $x\in \M(\G)$ and $y\in \dom(\La)$. The map $\La$ is closed with respect to the strict topology on $\M(\G)$ and the norm on $H$. Moreover, since $\f$ is faithful, $\iota$ is a faithful representation of $\G$ and we use it to identify $\G$ with its image in $\Ls(H)$ via $\iota$. In other words, we view $\iota$ as an inclusion map $\G\into \Ls(H)$ and omit it from all formulas (so, for instance, we have $\La(xy)=x\La(y)$). The Hilbert space $H$ is also sometimes denoted $L^2(\G)$. Other objects associated to a quantum group, like the left and right regular corepresentations $W$ and $V$, the modular element $\delta$, scaling constant $\nu$, and so on, will be reviewed throughout the text as need.

Of most importance for us will be the slice maps associated to $\f$ and $\La$. Given an arbitrary \cstar{}algebra $A$, there is an unbounded linear $\id_A\otimes\f$ from a suitable strictly dense (hereditary) \Star{}subalgebra $\dom(\id_A\otimes\f)\sbe\M(A\otimes\G)$ to $\M(A)$; a positive element $x\in \M(A\otimes\G)_+$ belongs to $\dom(\id_A\otimes\f)$ iff there is $a\in \M(A)$ such that for all positive linear functionals $\theta\in A^*_+$, $(\theta\otimes\id_\G)(x)\in \dom(\f)$ and in this case $(\id_A\otimes\f)(x)=a$. There is also an unbounded linear map $\id_A\otimes\La$ from $\dom(\id_A\otimes\La)\defeq \{x\in \M(A\otimes\G):x^*x\in \dom(\id_A\otimes\f)\}$ to $\M(A\otimes H)\defeq \Ls(A,A\otimes H)$, where here $A$ is viewed as (right) Hilbert $A$-module and $A\otimes H$ denotes the usual Hilbert $A$-module defined as the (external) tensor product of $A$ with $H$.
The space $\M(A\otimes H)$ is also sometimes called the multiplier module of $A\otimes H$; this is a Hilbert $\M(A)$-module in the canonical way.
The space $\dom(\id_A\otimes\La)$ is a strictly dense left ideal in $\M(A\otimes\G)$ whose associated hereditary \Star{}subalgebra is $\dom(\id_A\otimes\f)$, that is, every element of $\dom(\id_A\otimes\f)$ is a linear combination of products $x^*y$ with $x,y\in \dom(\id_A\otimes\La)$. The map $\id_A\otimes \La$ is closed for the strict topology on $\M(A\otimes\G)$ is the strong topology on $\M(A\otimes H)$ and satisfies
$$(\id_A\otimes\La)(x)^*(\id_A\otimes\La)(y)=(\id_A\otimes\f)(x^*y)\quad\mbox{for all }x,y\in \dom(\id_A\otimes\La);$$
$$(\id_A\otimes\La)(xy)=x\cdot (\id_A\otimes\La)(y)\quad\mbox{for all }x\in \M(A\otimes\G),\, y\in \dom(\id_A\otimes\La).$$
The full construction and further properties of the maps $\id_A\otimes\f$ and $\id_A\otimes\La$ are given in \cite{Kustermans-Vaes:Weight}. Recall that $\La$ is the GNS-map for $\f$; the map $\id_A\otimes\La$ may be viewed as a sort of generalized KSGNS-map for $\id_A\otimes\f$. Using linking algebras, it is possible to extend these constructions to Hilbert modules; we will have more to say about that in the next section.

A coaction of $\G$ on a \cstar{}algebra $B$ is a nondegenerate \Star{}homomorphism $\co_B\colon B\to \M(B\otimes\G)$ satisfying $(\co_B\otimes \id)\circ\co_B=(\id\otimes\Delta)\circ\co_B$. We are very flexible with this definition in general in the sense that we do not assume, for instance, that $\co_B$ is injective or even that its range is contained in $\tilde\M(B\otimes\G)\defeq \{x\in \M(B\otimes\G): x(1\otimes \G), (1\otimes \G)x\sbe B\otimes\G\}$. We say that $\co_B$ is continuous if the closed linear span of $\co_B(B)(1\otimes\G)$ equals $B\otimes\G$, and in this case we say $B$ is a $\G$-\cstar{}algebra. Given a continuous coaction, we define the \emph{reduced crossed product}:
$$B\rtimes_\red\cdualG\defeq\cspn\big(\co_B(B)(1\otimes\cdualG)\big)\sbe \Ls(B\otimes H),$$
where we have simply identified $\co_B(B)$ as a subalgebra of $\Ls(B\otimes H)$ using the representation $\G\into \Ls(H)$. Here $\cdualG=\dualJ\dualG \dualJ$ denotes the \cstar{}commutant of $\G$, where $\dualJ$ is the modular conjugation of $\dualG$ (see \cite{Kustermans-Vaes:LCQG} for further details). Continuity of the coaction $\co_B$ ensures that $B\rtimes_\red\cdualG$ is a \cstar{}subalgebra of $\Ls(B\otimes H)$.

Let $B$ be a \cstar{}algebra with a $G$-coaction $\co_B$, and let $\E$ be a (right) Hilbert $B$-module. A coaction on $\E$ is a linear map $\co_\E\colon \E\to \M(\E\otimes\G)\defeq \Ls(B\otimes\G,\E\otimes\G)$ satisfying:
\begin{enumerate}
\item $\co_\E(\xi\cdot b)=\co_\E(\xi)\co_B(b)$ for all $\xi\in \E$, $b\in B$;
\item $\co_\E(\xi)^*\co_\E(\eta)=\co_B(\braket{\xi}{\eta}_B)$ for all $\xi,\eta\in \E$;
\item $\co_\E$ is nondegenerate, meaning that $\cspn\co_\E(\E)(B\otimes\G)=\E\otimes\G$; and
\item $(\co_\E\otimes\id)\circ\co_\E=(\id\otimes\Delta)\circ\co_\E$ (this equation makes sense by nondegeneracy);
\end{enumerate}
If the underlying coaction $\co_B$ of $\G$ on $B$ is continuous, we also say that $\E$ is a Hilbert $B,\G$-module. Notice that in this case we have $\cspn\bigl(\co_\E(\E)(1\otimes\G)\bigr)=\E\otimes\G$. If, in addition, $\cspn\bigl((1_\E\otimes\G)\co_\E(\E)\bigr)=\E\otimes\G$, we say that a coaction $\co_\E$ is \emph{continuous} (this is not automatic, even if $\co_B$ is continuous). The theory of coactions on Hilbert modules has been developed in \cite{Baaj-Skandalis:Hopf_KK} and this is our main reference on the subject. For a coaction on $\E$, there is a canonical induced coaction $\co_{\K(\E)}$ on the \cstar{}algebra $\K(\E)$ of compact operators on $\E$ satisfying $\co_{\K(\E)}(\ket{\xi}\bra{\eta})=\co_\E(\xi)\co_\E(\eta)^*$, where $\ket{\xi}\bra{\eta}\in\K(\E)$ denotes the compact operator defined by $\ket{\xi}\bra{\eta}(\zeta)\defeq \xi\braket{\eta}{\zeta}_B$. If $\co_\E$ is continuous, so is $\co_{\K(E)}$. For a Hilbert $B,\G$-module $(\E,\co_\E)$, we define:
$$\E\rtimes_\red\cdualG:=\cspn\bigl(\co_\E(\E)(1_B\otimes\cdualG)\bigr)\sbe \Ls(B\otimes H,\E\otimes H).$$
Here we are using the embedding $\G\into \Ls(H)$ to view $\M(\E\otimes\G)$ as a subspace of $\M(\E\otimes \K(H))\cong\Ls(B\otimes H,\E\otimes H)$.
Observe that $\E\rtimes_\red\cdualG\sbe\Ls^\G(B\otimes H,\E\otimes H)$ is a \emph{concrete} Hilbert $B\rtimes_\red\cdualG$-module (as defined in \cite[Section~5]{Meyer:Generalized_Fixed}). Moreover, the map $\xi\otimes x\mapsto \co_\E(\xi)x$ yields a canonical isomorphism $\E\rot{\co_B}(B\rtimes_\red\cdualG)\cong\E\rtimes_\red\cdualG$. If $\co_\E$ is continuous, then $\E\rtimes_\red\cdualG=\cspn\bigl((1_{\K(\E)}\otimes\cdualG)\co_\E(\E)\bigr)$ and we have a canonical isomorphism $\K(\E\rtimes_\red\cdualG)\cong\K(\E)\rtimes_\red\cdualG$.

Let $\E$ be a Hilbert $B$-module with a coaction $\co_\E$ of $\G$. Given $\omega\in \G^*$ and $\xi\in \M(\E)$ we define
\begin{equation}\label{483}
\omega*\xi:=(\id_\E\otimes\omega)\bigl(\co_\E(\xi)\bigr).
\end{equation}
This gives $\M(\E)$ the structure of
a Banach left $\G^*$-module. Here we use the canonical Banach algebra structure on $\G^*$: $\omega\cdot\theta\defeq (\omega\otimes\theta)\circ\Delta$. In particular, $\M(\E)$ is also a Banach left $L^1(\G)$-module (this is a suitable Banach subalgebra of $\G^*$ isomorphic to the predual of von Neumann algebra $\G''\sbe\Ls(H)$; see Section~\ref{409} below and \cite{Kustermans-Vaes:LCQG} for the precise definition). But even if $\xi\in \E$ and $\omega\in L^1(\G)$, it is not true, in general, that $\omega*\xi\in \E$. However, if $\E$ is a Hilbert $B,\G$-module, that is, if $\co_B$ is continuous, this is true and in this case the left action~\textup{\eqref{483}} turns $\E$ into a nondegenerate Banach left $L^1(\G)$-module, that is, $\cspn(L^1(\G)*\E)=\E$. This is related to the notion of weak continuous actions defined in \cite{Baaj-Skandalis-Vaes:Non-semi-regular}.

\section{Review of square-integrable coactions}

In this section we review the main results concerning square-integrability for coactions of locally compact quantum group as studied in \cite{Buss-Meyer:Square-integrable}. Throughout we fix a locally compact quantum group $\G$ and denote its left Haar weight by $\f$.
As in the previous section, we fix a GNS-construction for $\f$ of the form $(H,\iota,\La)$, where $\iota$ denotes
the inclusion map $\G\hookrightarrow\Ls(H)$.

\begin{definition}[Definition~5.7 in \cite{Buss-Meyer:Square-integrable}] Let $\E$ be a Hilbert $B$-module with a coaction $\co_{\E}$ of $\G$. We say that $\xi\in\M(\E)$ is \emph{square-integrable}\index{square-integrable!element} if $\co_{\E}(\xi)^*(\eta\otimes 1)\in\M(B\otimes\G)$ belongs to the domain of  $\id_B\otimes\Lambda$ for all $\eta\in\E$. We write $\M(\E)_\si$ for the set of all square-integrable elements of $\M(\E)$, and $\E_\si$ for the set of square-integrable elements of $\E$. We say that $\E$ (or the coaction $\co_\E$)
is \emph{square-integrable}\index{square-integrable!coaction} if $\E_\si$ is dense in $\E$.
\end{definition}

If $A$ is a \cstar{}algebra with a coaction $\co_A$ of $\G$, we may view $A$ as a Hilbert $A$-module in the usual way, and therefore
speak of square-integrable elements and coactions in this case. It turns out that in the case of \cstar{}algebras one can give a slightly different description of square-integrable coactions in terms of integrable elements: an element $a\in A^+$ is called integrable if $\co_A(a)\in \dom(\id_A\otimes\f)$. An arbitrary element (not necessarily positive) $a\in A$ is said to be integrable if it is a linear combination of positive integrable elements. We write $A_\ii^+$ for the set of positive integrable elements and $A_\ii=\spn A_\ii^+$ for the space of all integrable elements. This is a hereditary \Star{}subalgebra of $A$. Moreover, $a\in A$ is square-integrable if and only if $aa^*\in A^+_\ii$. The coaction is called integrable if $A_\ii$ is dense in $A$ (or equivalent, if $A_\ii^+$ is dense in $A^+$). This is equivalent to square-integrability of $\co_A$ if $A$ is considered as a Hilbert $A$-module as above. More generally, if $\co_{\E}$ is a coaction of $\G$ on a Hilbert $B$-module $\E$, and
$\co_{\K(\E)}$ denotes the induced coaction of $\G$ on  $\K(\E)$, then an element $\xi\in \M(\E)$ is square-integrable if and only if $\ket{\xi}\bra{\xi}\in \K(\E)$ is integrable (see Proposition~5.20 in \cite{Buss-Meyer:Square-integrable}).

Square-integrable Hilbert $B,\G$-modules are characterized by the existence of sufficiently many adjointable $\G$-equivariant operators $\E\to B\otimes H$. We are going to explain how we can construct such operators in what follows.

Let $\E$ be a Hilbert $B$-module with a coaction $\co_\E$ of $\G$ and suppose that $\xi\in \M(\E)_\si$. Then the equation
\begin{equation}\label{eq:DefBraOperator}
\bbra{\xi}\eta:=(\id_B\otimes\La)\bigl(\co_\E(\xi)^*(\eta\otimes 1)\bigr)
\end{equation}
defines an adjointable operator $\bbra{\xi}:\E\to B\otimes H$ (see \cite[Lemma~5.17]{Buss-Meyer:Square-integrable}). For all $b\in B$ and $x\in \dom(\f)$, we have $\co_\E(\xi)(b\otimes s)\in \dom(\id_\E\otimes\f)$ and the adjoint operator $\kket{\xi}:=\bbra{\xi}^*$ is given by the formula
$$\kket{\xi}\bigl(b\otimes\La(x)\bigr)=(\id_\E\otimes\f)\bigl(\co_\E(\xi)(b\otimes x)\bigr)$$ for all $b\in B$ and $x\in \dom(\f)$. Here we are using the (unbounded) slice map $\id_\E\otimes\f\colon \dom(\id_\E\otimes\f)\sbe \M(\E\otimes \G)\to \M(\E)$ induced by $\f$. One way to define this is pass to the linking algebra $L(\E)=\left(\begin{array}{cc}\K(\E) & \E \\ \E^* & B\end{array}\right)\cong \K(\E\oplus B)$ of $\E$, consider the (already defined) slice $\id_{L(\E)}\otimes\f\colon \dom(\id_{L(\E)}\otimes\f)\sbe \M(L(\E)\otimes\G)\to \M(\E)$ and taking the ``upper right corner" to get the slice $\id_\E\otimes\f$ in such way that $\dom(\id_{L(\E)}\otimes\f)=\left(\begin{array}{cc}\dom(\id_{\K(\E)}\otimes\f)&\dom(\id_\E\otimes\f)\\ \dom(\id_{\E^*}\otimes\f) & \dom(\id_B\otimes\f)\end{array}\right)$ and $\id_{L(\E)}\otimes \f=\left(\begin{array}{cc}\id_{\K(\E)}\otimes\f & \id_\E\otimes\f\\ \id_{\E^*}\otimes\f & \id_B\otimes\f\end{array}\right)$. The slice $\id_{\E^*}\otimes\f$ defined in this way is an unbounded linear map from a suitable domain $\dom(\id_{\E^*}\otimes\f)\sbe \Ls(\E\otimes\G,B\otimes \G)$ to $\Ls(\E,B)$.

Similarly, one can construct the slices $\id_\E\otimes\La\colon \dom(\id_\E\otimes\La)\sbe \M(\E\otimes\G)\to\M(\E\otimes H)$ and $\id_{\E^*}\otimes\La\colon \dom(\id_{\E^*}\otimes\La)\sbe\Ls(\E\otimes\G,B\otimes \G)\to \Ls(\E,B\otimes H)$ of $\La$ in such way that
$\id_{L(\E)}\otimes \La=\left(\begin{array}{cc}\id_{\K(\E)}\otimes\La & \id_\E\otimes\La\\ \id_{\E^*}\otimes\La & \id_B\otimes\La\end{array}\right)$.

Defined in this way, an element $X\in \Ls(\E\otimes \G,B\otimes\G)$ belongs to the domain of $\id_{\E^*}\otimes\La$ if and only if $X(\eta\otimes 1)$ belongs to the domain of $\id_B\otimes \La$ for all $\eta\in \E$, and in this case
\begin{equation}
(\id_{\E^*}\otimes\La)(X)\eta=(\id_B\otimes\La)(X(\eta\otimes 1)).
\end{equation}
In particular, $\xi\in \M(\E)_\si$ if and only if $\co_\E(\xi)^*\in \dom(\id_{\E^*}\otimes\Lambda)$, and in this case
\begin{equation}\label{eq:BraOperatorAsSlice}
\bbra{\xi}=(\id_{\E^*}\otimes\La)\bigl(\co_\E(\xi)^*\bigr).
\end{equation}

\begin{example}[See also Example~5.16 in \cite{Buss-Meyer:Square-integrable}]\label{477} If $\G$ is a compact quantum group, that is, if the Haar weight $\f$ is bounded, then every Hilbert $B$-module $\E$ with a coaction of $\G$ is square-integrable. Given any $\xi\in \M(\E)$ the adjointable operator $\bbra{\xi}\in \Ls(\E,B\otimes H)$ can be described as follows: observe that $\La(x)=\La(x\cdot 1)=x\La(1)=x(\dtg_1)$, where $1$ is the unit of $\G$ and $\dtg_1:=\La(1)$. More generally, the map $\id_B\otimes\La$ is given by
$$(\id_B\otimes\La)(x)=x(\id_{B}\otimes\La)(1_B\otimes 1)=x(1_B\otimes\dtg_1)$$ for all $x\in \M(B\otimes\G)$, where we have identified $\M(B\otimes\G)\sbe\Ls(B\otimes H)$.
Even more generally, the map $\id_{\E^*}\otimes\La$ can also be written in the form
$$(\id_{\E^*}\otimes\La)(x)=x(1_\E\otimes\dtg_1)$$ for all $x\in \dom(\id_{\E^*}\otimes\Lambda)=\Ls(\E\otimes\G,B\otimes\G)$, where $1_\E$ denotes the identity operator on $\E$.
Thus $1_\E\otimes\dtg_1$ is an element of $\Ls(\E)\otimes H \sbe\Ls(\E,\E\otimes H)$.
Here we are identifying
$\Ls(B\otimes\G,\E\otimes H)=\M(\E\otimes\G)\sbe \M(\E\otimes \K(H))\cong\Ls(B\otimes H,\E\otimes H)$ and therefore $x$ is considered as an element of
$\Ls(\E\otimes\G,B\otimes\G)\sbe\Ls(\E\otimes H,B\otimes H)$.
In particular, we get
$$\bbra{\xi}=(\id_{\E^*}\otimes\La)\bigl(\co_\E(\xi)^*\bigr)=\co_\E(\xi)^*(1_\E\otimes\dtg_1)$$ for all $\xi\in\M(\E)$. The adjoint operator $\kket{\xi}\in \Ls(B\otimes H,\E)$ is therefore given by
$$\kket{\xi}=(1_{\E}\otimes\dtg_1^*)\co_\E(\xi),$$
where $\dtg_1^*$ denotes the element of $\Ls(H,\C)$ given by $\dtg_1^*(v)=\braket{\dtg_1}{v}$ for all $v\in H$.
\end{example}

It is useful to keep the group case in mind. As explained in \cite[Example~5.14]{Buss-Meyer:Square-integrable}, for a locally compact group $G$, if we consider the corresponding commutative quantum group $\G=\cont_0(G)$, then the theory of square-integrability specializes to the one developed in \cite{Meyer:Generalized_Fixed}. In this case, if $\E$ is a Hilbert $B,G$-module with $G$ action $\gamma$, and $\xi\in \E_\si$, the bra-ket operators are adjointable operators $\bbra{\xi}\colon \E\to L^2(G,B)\cong B\otimes L^2(G)$ and $\kket{\xi}\colon L^2(G,B)\to \E$ determined by the formulas:
\begin{equation}
\bbra{\xi}(\eta)|_t=\braket{\gamma_t(\xi)}{\eta}_B\quad \mbox{for all }\eta\in \E,\, t\in G,
\end{equation}
and
\begin{equation}
\kket{\xi}(f)=\int_G\gamma_t(\xi)f(t)\dd{t}\quad\mbox{for all }f\in \contc(G,B)\sbe L^2(G,B).
\end{equation}

The next result gives some basic properties of the bra-ket operators $\bbra{\xi}$ and $\kket{\xi}$.
Given a \cstar{}algebra $A$ with a coaction $\co_A$ of $\G$ and an element $a\in \M(A)_\ii$,
we define $E_1(a)\defeq (\id_A\otimes\f)\bigl(\co_A(a)\bigr)\in \M(A)$. By \cite[Lemma~4.10]{Buss-Meyer:Square-integrable}, $E_1(a)$ belongs to the multiplier fixed point algebra $\M_1(A)=\{a\in \M(A):\co_A(a)=a\otimes 1\}$. If $\alpha$ is an action of a locally compact group $G$ on $A$, then $E_1(a)$ can be interpreted as the strict unconditional integral $E_1(a)=\int_G^\su\alpha_t(a)\dd{t}$ (see \cite{Buss-Meyer:Continuous,Exel:Unconditional,Exel:SpectralTheory} for further details).

\begin{proposition}\label{003} Let $\E$ be a Hilbert $B$-module with a coaction $\co_\E$ of $\G$.
\begin{enumerate}
\item[\textup{(i)}] If $\xi,\eta\in \M(\E)_\si$, then $\xi\circ\eta^*\in \M\bigl(\K(\E)\bigr)_\ii$ and
$\kket{\xi}\bbra{\eta}= E_1(\xi\circ\eta^*).$

In particular, if $\xi,\eta\in \E_\si$, then $\ket{\xi}\bra{\eta}\in \K(\E)_\ii$ and
$\kket{\xi}\bbra{\eta}= E_1(\ket{\xi}\bra{\eta}).$
\item[\textup{(ii)}] If $\xi\in\M(\E)_\si$ and $b\in \M(B)$, then $\xi\cdot b\in\M(\E)_\si$ and
$\kket{\xi\cdot b}=\kket{\xi}\circ\co_B(b),$
where we have identified $\co_B(b)\in \M(B\otimes \G)\sbe\Ls(B\otimes H)$.

In particular, if $\xi\in \E_\si$ \textup(or even in $\M(\E)_\si$\textup) and $b\in B$, then $\xi\cdot b\in \E_\si$ and $\kket{\xi\cdot b}=\kket{\xi}\circ\co_B(b).$

\item[\textup{(iii)}] Let $\F$ be another Hilbert $B$-module with a coaction of $\G$.
If $\xi\in\M(\E)_\si$ and $T\in\Ls^\G(\E,\F)$, then $T\circ\xi\in\M(\F)_\si$ and
$\kket{T\circ\xi}=T\circ\kket{\xi}.$

In particular, if $\xi\in \E_\si$ and $T\in \Ls^\G(\E,\F)$, then $T(\xi)\in \F_\si$ and $$\kket{T(\xi)}=T\circ\kket{\xi}.$$

\item[\textup{(iv)}] If $T\in \M\bigl(\K(\E)\bigr)_\si$ and $\xi\in\M(\E)$, then $T\circ\xi\in \M(\E)_\si$ and
$$\kket{T\circ\xi}=\kket{T}\circ\co_\E(\xi).$$

In particular, if $T\in \M\bigl(\K(\E)\bigr)_\si$ and $\xi\in \E$, then $T(\xi)\in \E_\si$ and $$\kket{T(\xi)}=\kket{T}\circ\co_\E(\xi).$$

More generally, if $\pi:A\to \Ls(\E)$ is a $\G$-equivariant
nondegenerate $*$-homo\-morphism, where $A$ is a \cstar{}algebra with a coaction of $\G$,
then for all $a\in \M(A)_\si$ and $\xi\in \M(\E)$ we have $\pi(a)\circ\xi\in\M(\E)_\si$ and
$$\kket{\pi(a)\circ\xi}=\kket{\pi(a)}\co_\E(\xi)=(\pi\otimes\id_{H^*}) (\kket{a})\circ\co_\E(\xi).$$
\item[\textup{(v)}] If $\xi\in \M(\E)_\si$ and $\eta\in \M(\E)$, then $\xi\circ\eta^*\in \M\bigl(\K(\E)\bigr)_\si$ and
$$\kket{\xi\circ\eta^*}=\kket{\xi}\circ\co_\E(\eta)^*.$$

In particular, if $\xi\in \E_\si$ and $\eta\in\E$, then $\ket{\xi}\bra{\eta}\in \K(\E)_\si$ and
$$\kket{\ket{\xi}\bra{\eta}}=\kket{\xi}\circ\co_\E(\eta)^*.$$
\end{enumerate}
In \textup{(iv)} and \textup{(v)} we are viewing
$\M(\E\otimes\G)=\Ls(B\otimes\G,\E\otimes\G)$ as a subspace of $\Ls(B\otimes
H,\E\otimes H)$ and \textup(hence\textup) also $\Ls(\E\otimes\G,B\otimes\G)$ as a subspace $\Ls(\E\otimes H,B\otimes H)$ using the representation $\G\into \Ls(H)$. In \textup{(iv)} we also use the canonical isomorphism  $\Ls(\E,\E\otimes H)\cong\Ls(\K(\E),\K(\E)\otimes H)$.
\end{proposition}
\begin{proof} Using the definition~\eqref{eq:DefBraOperator} of the bra-operators $\bbra{\xi}$, or alternatively its description in Equation~\eqref{eq:BraOperatorAsSlice}, essentially all assertions will follow from properties of the slice map $\id_B\otimes\La$ described in \cite[Propositions~3.18,3.27,3.38]{Kustermans-Vaes:Weight}, or alternatively the corresponding properties for the slice $\id_{\E^*}\otimes\La$ (which can be derived from the properties of $\id_B\otimes\La$ using the linking algebras). So, for instance, the property:
$$(\id_B\otimes\La)(x)^*(\id_B\otimes\La)(y)=(\id_B\otimes\f)(x^*y)\quad\mbox{for all }x,y\in \dom(\id_B\otimes\La)$$
(which is proved in \cite[Proposition~3.18]{Kustermans-Vaes:Weight}) has a corresponding analogue for $\id_{\E^*}\otimes\La$:
$$(\id_{\E^*}\otimes\La)(x)^*(\id_{\E^*}\otimes\La)(y)=(\id_B\otimes\f)(x^*y)\quad\mbox{for all }x,y\in \dom(\id_{\E^*}\otimes\La).$$
This together with Equation~\eqref{eq:BraOperatorAsSlice} yields item (i). Similarly, (ii) follows from:
\begin{equation*}
(\id_{\E^*}\otimes\La)(xy)=x\cdot (\id_{\E^*}\otimes\La)(y)
\end{equation*}
for all $x\in \M(B\otimes\G)\sbe \Ls(B\otimes H)$ and $y\in \dom(\id_{\E^*}\otimes\La)$, which is also an extension of the corresponding property of $\id_B\otimes\La$ proved in \cite[Proposition~3.18]{Kustermans-Vaes:Weight}. This same result (applied to the linking algebra of $\E$) also yields:
\begin{equation*}
(\id_{\E^*}\otimes\La)(xy)=x\cdot (\id_{\K(\E)}\otimes\La)(y)
\end{equation*}
for all $x\in \Ls(\E\otimes\G,B\otimes G)\sbe \Ls(\E\otimes H,B\otimes H)$ and $y\in \dom(\id_{\K(\E)}\otimes\La)\sbe \M(\K(\E)\otimes\G)\cong\Ls(\E\otimes\G)$; and also
\begin{equation*}
(\id_{\E^*}\otimes\La)(xy)=x\cdot (\id_{\E^*}\otimes\La)(y)
\end{equation*}
for all $x\in \Ls(\E\otimes\G,B\otimes G)$ and $y\in \dom(\id_{\E^*}\otimes\La)\sbe \Ls(\E\otimes \G,B\otimes \G)$.
These properties then imply the first part of (iv) and (v). The second part in (iv) also uses:
\begin{equation*}
\kket{\pi(a)}=(\pi\otimes\id_{H^*})(\kket{a})\quad\mbox{for all }a\in \M(A)_\si.
\end{equation*}
This holds whenever $\pi\colon A\to \M(B)$ is a nondegenerate \Star{}homomorphism which is equivariant: $\co_{B}(\pi(a))=(\pi\otimes\id)(\co_A(a))$. In fact, by Proposition~3.38 in \cite{Kustermans-Vaes:Weight}, which implies that for all $X\in \dom(\id_A\otimes\La)$, one has $(\pi\otimes\id)(X)\in \dom(\id_B\otimes\La)$ and $(\id_B\otimes\La)((\pi\otimes\id)(X))=(\pi\otimes\id_H)((\id_A\otimes\La)(X))$. Hence,
\begin{align*}
\bbra{\pi(\xi)}&=(\id_B\otimes\La)\bigl(\co_B(\pi(\xi))^*\bigr)
	\\&=(\id_B\otimes\La)\bigl((\pi\otimes\id_\G)\co_A(\xi)^*\bigr)
	\\&=(\pi\otimes\id_H)\Bigl((\id_A\otimes\La)\bigl(\co_A(\xi)^*\bigr)\Bigr)
	\\&=(\pi\otimes\id_H)(\bbra{\xi}).
\end{align*}
Since $(\pi\otimes\id_H)(x)^*=(\pi\otimes\id_{H^*})(x^*)$, it also follows from this equation that
\begin{equation*}
\kket{\pi(\xi)}=(\pi\otimes\id_{H^*})(\kket{\xi}).
\end{equation*}
Finally, item (iii) follows from the
$\G$-equivariance of $T$: $\co_\F(T\xi)=(T\otimes 1)\co_\E(\xi)$, and the equality:
$$(\id_{\F^*}\otimes\La)(x(T\otimes 1))=(\id_{\E^*}\otimes \La)(x)T$$
which holds for all $x\in \dom(\id_{\E^*}\otimes\La)$ and $T\in \Ls(\F,\E)$ -- this implies $x(T\otimes 1)\in \dom(\id_{\F^*}\otimes \La)$. In fact, using linking algebras, this last property follows from the corresponding property for slices only involving \cstar{}algebras as proved in \cite[Proposition~3.27]{Kustermans-Vaes:Weight}.
\end{proof}

Let $(\E,\co_\E)_{(B,\co_B)}$ be a Hilbert $B$-module $\G$-coaction. By \cite[Lemma~5.28]{Buss-Meyer:Square-integrable}, if we equip $\M(\E)_\si$ with the so-called $\si$-norm:
$$\|\xi\|_\si:=\|\xi\|+\|\kket{\xi}\|=\|\xi\|+\|\bbra{\xi}\|=\|\<\xi|\xi\>\|^{\frac{1}{2}}+\|\bbraket{\xi}{\xi}\|^{\frac{1}{2}},$$
then $\M(\E)_\si$ is a Banach $\Ls^\G(\E),\M(B)$-bimodule, that is, $\M(\E)_\si$ is complete with respect to $\|\cdot\|_\si$ and for
all $\xi\in \E_\si$, $T\in \Ls^\G(\E)$ and $b\in B$, we have
$$\|T\xi\|_\si\leq\|T\|\|\xi\|_\si\quad\mbox{and}\quad \|\xi b\|_\si\leq\|\xi\|_\si\|b\|.$$
Moreover, $\E_\si$ is a closed submodule of $\M(\E)_\si$ and hence also complete in its own.

\begin{remark}\label{100} Suppose that $\G$ is a compact quantum group. We already know (see Example~\ref{477}) that in this case
every Hilbert $B$-module $\E$ with a coaction of $\G$ is square-integrable.
By Proposition~\ref{003}(i), we have
$$\|\kket{\xi}\|^2=\|\kket{\xi}\bbra{\xi}\|=\|(\id\otimes\f)(\co_{\K(\E)}(|\xi\>\<\xi|))\|\leq\|\f\|\|\xi\|$$ for all $\xi\in \E_\si=\E$. Thus $\|\xi\|\leq\|\xi\|_\si\leq (1+\|\f\|)\|\xi\|$.
Therefore the $\si$-norm and the norm on $\E$ are
equivalent.
\end{remark}

Consider a Hilbert $B$-module $\E$ with a coaction $\co_\E$ of $\G$.
The bra-ket operators
$\bbra{\xi}\in \Ls(\E,B\otimes H)$ and $\kket{\xi}\in \Ls(B\otimes H,\E)$
are $\G$-equivariant, for any square-integrable element $\xi$ in $\E$. In order to turn this into a precise statement, we have to define a $\G$-coaction on $B\otimes H$. The coaction that works is a kind of balanced tensor product of the coactions $\co_B$ on $B$ and a coaction $\co_{H}$ on $H$ which comes from the left regular corepresentation $W$ of $\G$ (this is a unitary multiplier in $\M(\G\otimes\dualG)$, where $\dualG\sbe \Ls(H)$ denotes the dual of $\G$, and hence may be viewed as a unitary in $\Ls(\G\otimes H)$; see \cite{Kustermans-Vaes:LCQG} for the precise definition of $W$). More precisely, this is the coaction $\co_{B\otimes H}$ of $\G$ on $B\otimes H$ defined by the formula:
\begin{equation}\label{136}
\co_{B\otimes H}(\zeta):=(1\otimes\Sigma W)(\co_B\otimes\id)(\zeta)=\Sigma_{23} W_{23}(\co_B\otimes\id)(\zeta),\quad\zeta\in B\otimes H,
\end{equation}
where $\Sigma:\G\otimes H\to H\otimes \G$ is the flip
operator. Recall that $\hat{W}\Sigma=\Sigma W^*$,
where $\hat{W}$ is the left regular corepresentation of the dual of $\G$. If we consider $B=\C$ with the trivial coaction of $\G$, then we get a coaction $\co_{H}$ of $\G$ on $H$ given by
$$\co_{H}(\eta)=\Sigma W(1\otimes\eta)=\Sigma W\Sigma^*(\eta\otimes 1)=\hat{W}^*(\eta\otimes 1),\quad \eta\in H.$$

The above coaction on $B\otimes H$ is one of the basic examples of a square-integrable coaction.
In a similar way, there is a canonical $\G$-coaction on the Hilbert $B$-module $\E\otimes H$ which is always square-integrable (for any coaction on $\E$).

The Kasparov Stabilization Theorem relates square-integrability of a given coaction with $B\otimes H$:

\begin{theorem}[{\bf Kasparov's Stabilization Theorem}, Theorem~6.1 in \cite{Buss-Meyer:Square-integrable}]\label{121}
Let $B$ be a \cstar{}algebra with a coaction $\co_B$ of $\G$ and let $\E$ be a countably generated Hilbert $B$-module with a $\co_B$-compatible coaction of $\G$. The following statements are equivalent\textup:
\begin{enumerate}
\item[\textup{(i)}] $\E$ is square-integrable;
\item[\textup{(ii)}] $\K(\E)$ is integrable;
\item[\textup{(iii)}] $\E\oplus \Hc_B\cong \Hc_B$ as Hilbert $B,\G$-modules;
\item[\textup{(iv)}] $\E$ is a $\G$-invariant direct summand of $\Hc_B$.
\end{enumerate}
\end{theorem}

\section{The $L^1$-action on square-integrable elements}\label{409}
\noindent
Let $\E$ be a Hilbert $B,\G$-module. If $\xi\in \E_\si$ and $\omega\in L^1(\G)$, then it is natural to ask whether $\omega*\xi\in \E_\si$.
However, if $\G$ is not unimodular, that is, if the modular element is not trivial, then some problems appear.
Let us analyze the group case $\G=\cont_0(G)$, where $G$ is some locally compact group.
Suppose that $\co_\E$ corresponds to an action $\gamma$ of $G$ on $\E$.
Then for a function $\omega\in L^1(G)$, the element $\omega*\xi\in \E$ is given by
$$\omega*\xi=\int_G\gamma_t(\xi)\omega(t)\dd{t}.$$
Thus, for all $f\in \cont_c(G,B)$, we have
\begin{align*}
\kket{\omega*\xi}f&=\int_G\int_G\gamma_{st}(\xi)f(s)\omega(t)\dd{t}\dd{s}
					\\&=\int_G\int_G\gamma_t(\xi)f(s)\omega(s^{-1}t)\dd{t}\dd{s}
					\\&=\kket{\xi}(f*\omega),
\end{align*}
where $(f*\omega)(t):=\int_G f(s)\omega(s^{-1}t)\dd{s}=\int_Gf(ts^{-1})\mo_G(s)^{-1}\omega(s)\dd{s}$, where $\dtg_G$ denotes the modular function of $G$.
If $\omega$ satisfies $\int_G\mo_G(t)^{-\frac{1}{2}}|\omega(t)|\dd{t}<\infty$, then the map
$\rho_\omega:=[g\mapsto g*\omega]$ defines a bounded operator on
$L^2(G)$ with $\|\rho_\omega\|\leq\int_G\mo_G(t)^{-\frac{1}{2}}|\omega(t)|\dd{t}$ (\cite[Theorem~20.13]{Hewitt-Ross:Abstract_harmonic_analysisI}).
Note that $f*\omega=(1_B\otimes\rho_\omega)(f)$. Thus, if
$\xi\in \E_\si$ and $\omega\in L^1(G)$ satisfies $\mo_G^{-\frac{1}{2}}\omega\in L^1(G)$, then $\omega*\xi\in\E_\si$ and
$$\kket{\omega*\xi}=\kket{\xi}(1_B\otimes\rho_\omega).$$
The hypothesis $\mo_G^{-\frac{1}{2}}\omega\in L^1(G)$ is
essential here in order to define the operator $\rho_\omega$. In
fact, if $G$ is not unimodular, then there are functions $\omega\in L^1(G)$ and $g\in L^2(G)$ such that $g*\omega\notin L^2(G)$ (see \cite[20.34]{Hewitt-Ross:Abstract_harmonic_analysisI}).

In order to generalize the results above for a general locally compact quantum group $\G$, we shall need the modular element. As usual the proof in the quantum setting is much more technical. Let us recall that the modular element of $\G$, denoted by $\mo$, is a strictly positive operator affiliated with $\G$ (see \cite{Lance:Hilbert_modules,Woronowicz:Unbounded_affiliated} for a precise definition) such that $\sigma_t(\mo)=\nu^t\mo$ for all $t\in \Real$ and $\psi=\f_\mo$ (see \cite{Kustermans-Vaes:LCQG}), where $\psi$ is the right invariant Haar weight, $\{\sigma_s\}_{s\in\Real}$ is the modular automorphism group of $\f$ and $\nu$ is the scaling constant of $\G$. We also recall that $\Dt(\mo)=\mo\otimes\mo$ (\cite[Proposition~7.9]{Kustermans-Vaes:LCQG}). Roughly speaking, the relation $\psi=\phi_\mo$ means that $\psi(\,\cdot\,)=\f(\mo^{\frac{1}{2}}\,\cdot\,\mo^{\frac{1}{2}})$ and one can define a GNS-construction for $\psi$ of the form $(H,\iota,\Gamma)$ from the GNS-construction $(H,\iota,\La)$ for $\f$ satisfying $\Gamma(\,\cdot\,)=\La(\,\cdot\,\mo^{\frac{1}{2}})$ (see \cite{Kustermans:KMS} for more details).

For each $n\in \Nat$, we define
\begin{equation}\label{307}
e_n:=\frac{n}{\sqrt\pi}\int_\Real\exp(-n^2t^2)\mo^{\ii t}\dd{t}.
\end{equation}
These elements behave very well with respect to the modular element. For instance, they commute with any power of $\mo$ and $\sigma_y(e_n)\mo^z=\mo^z\sigma_y(e_n)$ for all $n\in \Nat$ and $y,z\in \C$ (see \cite[Proposition~8.2]{Kustermans:KMS} for further details).

We shall need a generalization of \cite[Proposition~1.9.13]{Vaes:Thesis}.
This result says that for all $a\in \dom(\La)$, $u\in \dom(\mo^{\frac{1}{2}})$ and $v\in H$,
we have $(\id\otimes\omega_{u,v})\Dt(a)\in \dom(\La)$ (where $\omega_{u,v}$ is the vector functional defined by $\omega_{u,v}(x)\defeq \braket{u}{xv}$) and

\begin{equation}\label{410}
\La\bigl((\id\otimes\omega_{u,v})\Dt(a)\bigr)=(\id\otimes\omega_{\mo^{\frac{1}{2}}u,v})(V)\La(a),
\end{equation}
where $V$ is the right regular corepresentation of $\G$, which is determined by:
\begin{equation}\label{086}
V\bigl(\Gamma(a)\otimes 1\bigr)=(\Gamma\otimes\id)\bigl(\Dt(a)\bigr)\quad\mbox{for all }a\in \dom(\Gamma).
\end{equation}
or, equivalently,
\begin{equation}\label{085}
(\id\otimes\omega)(V)\Gamma(a)=\Gamma\bigl((\id\otimes\omega)\Dt(a)\bigr)\quad\mbox{for all }a\in \dom(\Gamma)\mbox{ and }\omega\in \Ls(H)_*.
\end{equation}

The proof in \cite[Proposition~1.9.13]{Vaes:Thesis} can be easily generalized to slices with $\f$ and yields: for $B$ a \cstar{}algebra, $x\in \dom(\id_B\otimes\Lambda)$, $u\in \dom(\mo^{\frac{1}{2}})$ and $v\in H$,
\begin{equation}\label{263}
(\id_{B}\otimes\id_\G\otimes\omega_{u,v})\bigl((\id_{B}\otimes\Dt)(x)\bigr)\in \dom(\id_{B}\otimes\Lambda)\quad \mbox{and}\quad
\end{equation}
\begin{multline}\label{263.1}
(\id_{B}\otimes\La)\Bigl((\id_{B}\otimes\id_\G\otimes\omega_{u,v}) \bigl((\id_{B}\otimes\Dt)(x)\bigr)\Bigr)
\\=\bigl(1_B\otimes(\id_{\K(H)}\otimes\omega_{\mo^{\frac{1}{2}}u,v})(V)\bigr)(\id_{B}\otimes\La)(x).
\end{multline}

We are now ready to prove the main result of this section. Define
$$L^1_{00}(\G):=\spn\{\omega_{u,v}:u\in H,v\in \dom(\mo^{\frac{1}{2}})\}$$
Note that $L^1_{00}(\G)$ is a dense subspace of $L^1(\G)=\cspn\{\omega_{u,v}:u,v\in H\}$.
Moreover, $L^1(\G)$ is the predual of the von Neumann algebra $\G''\sbe \Ls(H)$ generated by $\G$, and since this is in standard form (see (\cite[10.15]{Stratila-Zsido:von_Neumann})), we have $L^1(\G)=\{\omega_{u,v}:u,v\in H\}$ (see \cite[V.3.15]{Takesaki:Theory_1}).
Thus the essential difference between $L^1_{00}(\G)$ and $L^1(\G)$ lies in the difference between $\dom(\mo^{\frac{1}{2}})$ and $H$.
In particular, if $\G$ is unimodular, then $L^1_{00}(\G)$ is equal to $L^1(\G)$.
We also define a map
$$\rho:L^1_{00}(\G)\to \Ls(H),\quad\rho_{\omega_{u,v}}:=(\id\otimes\omega_{u,\mo^{\frac{1}{2}}v})(V^*)$$
for all $u,\in H$ and $v\in \dom(\mo^{\frac{1}{2}})$, and extend it linearly to $L^1_{00}(\G)$.
Note that if $\G$ is unimodular, then
$\rho_\omega=(\id\otimes\omega)(V^*)$ for all $\omega\in L^1(\G)$.

\begin{proposition}\label{264} Let $\E$ be a Hilbert $B,\G$-module. Then, for all $\xi\in \E_\si$ and $\omega\in L^1_{00}(\G)$, we have
$\omega*\xi\in \E_\si$ and
$$\kket{\omega*\xi}=\kket{\xi}(1_B\otimes\rho_\omega).$$
In particular,
$\|\omega*\xi\|_\si\leq\|\omega\|_{\rho}\|\xi\|_\si$, where
$\|\omega\|_{\rho}:=\max\{\|\omega\|,\|\rho_\omega\|\}$. Here $\|\omega\|$ denotes the norm of $\omega$ in $L^1(\G)$ and
$\|\rho_\omega\|$ denotes the norm of the operator $\rho_\omega\in \Ls(H)$.
\end{proposition}
\begin{proof} We may assume that $\omega=\omega_{u,v}$, for $u\in H$ and
$v\in \dom(\mo^{\frac{1}{2}})$. We have
\begin{align*}
\co_\E(\omega*\xi)&=\co_\E\bigl((\id_\E\otimes\omega)\co_\E(\xi)\bigr)
                \\&=(\id_\E\otimes\id_\G\otimes\omega)\bigl((\co_\E\otimes\id_\G)\co_\E(\xi)\bigr)
				\\&=(\id_\E\otimes\id_\G\otimes\omega)\bigr((\id_\E\otimes\Dt)\co_\E(\xi)\bigr).
\end{align*}
Hence
$\co_\E(\omega*\xi)^*=(\id_{\E^*}\otimes\id_\G\otimes\omega_{v,u})\bigl((\id_{\E^*}\otimes\Dt)\co_\E(\xi)^*\bigr)$.
Since $\xi\in \E_\si$ we have $\co_\E(\xi)^*\in
\dom(\id_{\E^*}\otimes\Lambda)$ and hence, by~\eqref{263} and~\eqref{263.1},
$\co_\E(\omega*\xi)^*\in \dom(\id_{\E^*}\otimes\Lambda)$, that
is, $\omega*\xi\in \E_\si$ and
\begin{align*}
\bbra{\omega*\xi}&=(\id_{\E^*}\otimes\La)\bigl(\co_\E(\omega*\xi)^*\bigr)
			   \\&=(\id_{\E^*}\otimes\La)\Bigl((\id_{\E^*}\otimes\id_\G\otimes\omega_{v,u})\bigl((\id_{\E^*}\otimes\Dt)\co_\E(\xi)^*\bigr)\Bigr)
			   \\&=\bigl(1_B\otimes(\id_{\K(H)}\otimes\omega_{\mo^{\frac{1}{2}}v,u})(V)\bigr)(\id_{\E^*}\otimes\La)\bigl(\co_\E(\xi)^*\bigr)
               \\&=\bigl(1_B\otimes(\id_{\K(H)}\otimes\omega_{\mo^{\frac{1}{2}}v,u})(V)\bigr)\bbra{\xi}.
\end{align*}
The formula $\kket{\omega*\xi}=\kket{\xi}(1_B\otimes\rho_\omega)$ now follows by taking adjoints.
\end{proof}

If $\G$ is unimodular, then $\E_\si$ is actually a Banach left $L^1(\G)$-module. In order to obtain a Banach left module also in the general non-unimodular case, we define following subspace of $L^1(\G)$:
$$L^1_0(\G):=\{\omega\in L^1(\G):\mo^{\frac{1}{2}}\omega\in L^1(\G)\},$$
where $(\mo^{\frac{1}{2}}\omega)(x):=\omega(x\mo^{\frac{1}{2}})$ for all left multipliers $x$ of $\mo^{\frac{1}{2}}$. The condition $\mo^{\frac{1}{2}}\omega\in L^1(\G)$ means that there is $\theta\in L^1(\G)$ such that $\theta(x)=\omega(x\mo^{\frac{1}{2}})$ for all left multipliers $x$ of $\mo^{\frac{1}{2}}$, and in this case we put $\mo^{\frac{1}{2}}\omega=\theta$.

\begin{proposition} $L^1_0(\G)$ is a subalgebra of $L^1(\G)$.
\end{proposition}
\begin{proof} Take $\omega_1,\omega_2\in L^1_0(\G)$. Then,
for every left multiplier $x$ of $\mo^{\frac{1}{2}}$, we have
\begin{align*}
(\omega_1\cdot\omega_2)(x\mo^{\frac{1}{2}})&=(\omega_1\otimes\omega_2)\bigl(\Dt(x\mo^{\frac{1}{2}})\bigr)
										\\&=(\omega_1\otimes\omega_2)\bigl(\Dt(x)(\mo^{\frac{1}{2}}\otimes\mo^{\frac{1}{2}})\bigr)
										\\&=(\mo^{\frac{1}{2}}\omega_1\otimes\mo^{\frac{1}{2}}\omega_2)\Dt(x)
										\\&=\bigl((\mo^{\frac{1}{2}}\omega_1)\cdot(\mo^{\frac{1}{2}}\omega_2)\bigr)(x).
\end{align*}
Thus $\mo^{\frac{1}{2}}(\omega_1\cdot\omega_2)\in L^1(\G)$, that is,
$\omega_1\cdot\omega_2\in L^1_0(\G)$, and
\begin{equation}\label{310}
\mo^{\frac{1}{2}}(\omega_1\cdot\omega_2)=(\mo^{\frac{1}{2}}\omega_1)\cdot(\mo^{\frac{1}{2}}\omega_2).\qedhere
\end{equation}
\end{proof}

Now define the following norm on $L^1_0(\G)$,
$$\|\omega\|_0:=\max\left\{\|\omega\|,\|\mo^{\frac{1}{2}}\omega\|\right\}.$$

\begin{proposition} The space $L^1_0(\G)$ endowed with the norm $\|\cdot\|_0$ \textup(and the product of $L^1(\G)$\textup) is a Banach algebra.
\end{proposition}
\begin{proof} By Equation~\eqref{310}, we have
$$\|\omega_1\cdot\omega_2\|_0\leq\|\omega_1\|_0\|\omega_2\|_0$$ for all $\omega_1,\omega_2\in L^1_0(\G)$. Thus all we have to prove
is that $L^1_0(\G)$ is a Banach space with the norm $\|\cdot\|_0$.
Take a Cauchy sequence $(\omega_n)$ in $L^1_0(\G)$ (with respect to
$\|\cdot\|_0$). Then, by definition of the norm $\|\cdot\|_0$,
both $(\omega_n)$ and $(\mo^{\frac{1}{2}}\omega_n)$ are Cauchy sequences in
$L^1(\G)$. Let $\omega$ and $\theta$ be the respective limits in
$L^1(\G)$. Then, for every left multiplier $x$ of $\mo^{\frac{1}{2}}$, we have
$$(\mo^{\frac{1}{2}}\omega)(x)=\omega(x\mo^{\frac{1}{2}})=\lim\limits_{n\to \infty}\omega_n(x\mo^{\frac{1}{2}})=\lim\limits_{n\to \infty}\mo^{\frac{1}{2}}\omega_n(x)=\theta(x).$$
Hence $\mo^{\frac{1}{2}}\omega=\theta\in L^1(\G)$, that is, $\omega\in
L^1_0(\G)$, and therefore $\|\omega_n-\omega\|_0\to 0$.
\end{proof}

Note that $L^1_{00}(\G)$ is contained in $L^1_0(\G)$. If fact, if $u\in L^2(\G)$ and $v\in \dom(\mo^{\frac{1}{2}})$, then
\begin{equation}\label{524}
\mo^{\frac{1}{2}}\omega_{u,v}(x)=\omega_{u,v}(x\mo^{\frac{1}{2}})=\braket{u}{x\mo^{\frac{1}{2}}v}=\omega_{u,\mo^{\frac{1}{2}}v}(x)
\end{equation}
for every left multiplier $x$ of $\mo^{\frac{1}{2}}$. This means that
$\mo^{\frac{1}{2}}\omega_{u,v}=\omega_{u,\mo^{\frac{1}{2}}v}\in L^1(\G)$.

\begin{proposition}\label{309} The subspace $L^1_{00}(\G)$ is dense in $L^1_0(\G)$ \textup(with respect to $\|\cdot\|_0$\textup).
\end{proposition}
\begin{proof}
Let $\omega\in L^1_0(\G)$. Then $\omega\in L^1(\G)\cong \G''_*$, the predual of the von Neumann algebra $\G''$, which is in standard form, so that there exist $u,v\in H$ such that $\omega=\omega_{u,v}$. Take a sequence $(v_k)\sbe
\dom(\mo^{\frac{1}{2}})$ such that $v_k\to v$, and define $v_{n,k}:=e_nv_k$ (where $e_n$ are defined as in Equation~\eqref{307}). Since $e_n$ commutes with $\mo^{\frac{1}{2}}$, it follows that $v_{n,k}\in\dom(\mo^{\frac{1}{2}})$. Observe that $\omega_{u,v_{n,k}}\in L^1_{00}(\G)$ for all $n,k\in \Nat$. Since $v_{n,k}\to v$ as $n,k\to \infty$, we have $\omega_{u,v_{n,k}}\to \omega_{u,v}$ in $L^1(\G)$ as $n,k\to \infty$. Now note that
\begin{align*}
	\|\mo^{\frac{1}{2}}\omega_{u,v_{n,k}}-\mo^{\frac{1}{2}}\omega_{u,v}\|&=\|\mo^{\frac{1}{2}}\omega_{u,e_nv_k}-\mo^{\frac{1}{2}}\omega_{u,v}\|
                        \\&=\|\mo^{\frac{1}{2}}e_n\omega_{u,v_k}-\mo^{\frac{1}{2}}\omega_{u,v}\|
						\\&\leq \|\mo^{\frac{1}{2}}e_n\omega_{u,v_k}-\mo^{\frac{1}{2}}e_n\omega_{u,v}\|+
                                         \|\mo^{\frac{1}{2}}e_n\omega_{u,v}-\mo^{\frac{1}{2}}\omega_{u,v}\|.
\end{align*}
For the second term above, we use $e_n\mo^{\frac{1}{2}}=\mo^{\frac{1}{2}}e_n$ to get
$$\|\mo^{\frac{1}{2}}e_n\omega_{u,v}-\mo^{\frac{1}{2}}\omega_{u,v}\|=\|e_n\mo^{\frac{1}{2}}\omega-\mo^{\frac{1}{2}}\omega\|\to
0,\quad\mbox{as }n\to\infty$$
For the first term, note that, for each fixed $n$, we have
$$\|\mo^{\frac{1}{2}}e_n\omega_{u,v_k}-\mo^{\frac{1}{2}}e_n\omega_{u,v}\|\to 0,\quad\mbox{as }k\to\infty.$$
Thus we can find a sequence $(k_n)$ of natural numbers such that $k_1<k_2<\ldots$ and
$$\|\mo^{\frac{1}{2}}e_n\omega_{u,v_k}-\mo^{\frac{1}{2}}e_n\omega_{u,v}\|<1/n.$$
Finally, defining $v_n:=v_{n,k_n}$, we conclude that $\omega_n:=\omega_{u,v_n}\in L^1_{00}(\G)$ and
$$\|\mo^{\frac{1}{2}}\omega_n-\mo^{\frac{1}{2}}\omega\|\leq 1/n + \|e_n\mo^{\frac{1}{2}}\omega-\mo^{\frac{1}{2}}\omega\|\to
0.$$
Therefore $\|\omega_n-\omega\|_0\leq\|\omega_n-\omega\|+\|\mo^{\frac{1}{2}}\omega_n-\mo^{\frac{1}{2}}\omega\|\to 0$.
\end{proof}

Define
$$\rho:L^1_0(\G)\to \Ls(H),\quad\rho(\omega):=(\id\otimes\mo^{\frac{1}{2}}\omega)(V^*).$$
Note that $\rho$ is, in fact, an extension of the map $\rho:L^1_{00}(\G)\to\Ls(H)$ previously defined, so that there is no problem of
notation.

\begin{proposition}\label{508} The map $\rho:L^1_0(\G)\to\Ls(H)$ is an injective, contractive, algebra anti-homo\-morphism whose image is dense in $\cdualG$.
\end{proposition}
\begin{proof} Consider the opposite $\opG$ of $\G$. The left regular corepresentation $W^\op$ of $\opG$
is equal to $\Sigma V^*\Sigma$ (see \cite[Proposition~1.14.10]{Vaes:Thesis}). It follows that
$$\rho(\omega)=(\id\otimes\mo^{\frac{1}{2}}\omega)(V^*)=(\mo^{\frac{1}{2}}\omega\otimes\id)(W^\op)=\laop{}(\mo^{\frac{1}{2}}\omega)$$
for all $\omega\in L^1_0(\G)$. Since $L^1(\opG)$ equals the opposite algebra of $L^1(\G)$, we get
\begin{align*}
\rho(\omega_1\cdot\omega_2)&=\laop{}\bigl(\mo^{\frac{1}{2}}(\omega_1\cdot\omega_2)\bigr)
						  \\&=\laop{}\bigl((\mo^{\frac{1}{2}}\omega_1)\cdot(\mo^{\frac{1}{2}}\omega_2)\bigr)
						  \\&=\laop{}(\mo^{\frac{1}{2}}\omega_2)\laop{}(\mo^{\frac{1}{2}}\omega_1)
						  \\&=\rho(\omega_2)\rho(\omega_1).
\end{align*}
Thus $\rho$ is an anti-homomorphism. Note also that
$\|\rho(\omega)\|\leq\|\mo^{\frac{1}{2}}\omega\|\leq\|\omega\|_0$. Hence
$\rho$ is contractive. If $\rho(\omega)=\laop{}(\mo^{\frac{1}{2}}\omega)=0$, then $\mo^{\frac{1}{2}}\omega=0$ because $\laop{}$ is injective.
This implies $\omega(x\mo^{\frac{1}{2}})=0$ for every left multiplier $x$ of $\mo^{\frac{1}{2}}$. Taking
$x=ye_n\mo^{-\frac{1}{2}}$ we get $\omega(ye_n)=0$ for all $n\in \Nat$ and $y\in \G$
and hence $\omega=0$ because $e_n\to 1$ strictly. Therefore $\rho:L^1_0(\G)\to \Ls(H)$ is an injective, contractive, algebra anti-homomorphism.

Finally, note that $\rho\bigl(L^1_0(\G)\bigr)=\laop{}\bigl(\mo^{\frac{1}{2}}L^1_0(\G)\bigr)\sbe\widehat{\opG}=\cdualG$.
Since $\mo^{\frac{1}{2}}L^1_0(\G)$ contains $\mo^{\frac{1}{2}}L^1_{00}(\G)$, which
contains elements of the form $\omega_{u,\mo^{\frac{1}{2}}v}$, where $u\in H$ and $v\in \dom(\mo^{\frac{1}{2}})$,
and since such elements span a dense subspace of $L^1(\G)$,
we conclude that $\rho\bigl(L^1_0(\G)\bigr)$ is dense in $\cdualG$ as well
(the image of $\laop{}$ is dense in $\widehat{\opG}=\cdualG$).
\end{proof}

The next result implies that $\E_\si$ is a Banach left $L^1_0(\G)$-module.

\begin{proposition}\label{358} Let $\E$ be a Hilbert $B,\G$-module. If $\omega\in L^1_0(\G)$ and $\xi\in \E_\si$, then $\omega\in \E_\si$ and
$$\kket{\omega*\xi}=\kket{\xi}(1_B\otimes\rho_\omega).$$
In particular, $\|\omega*\xi\|_\si\leq\|\omega\|_0\|\xi\|_\si$ for all $\xi\in \E_\si$ and $\omega\in L^1_0(\G)$.
\end{proposition}
\begin{proof} Let $(\omega_n)$ be a sequence in $L^1_{00}(\G)$ converging to $\omega$ (with respect to $\|\cdot\|_0$). In particular, $\omega_n\to\omega$ in $L^1(\G)$, and hence $\omega_n*\xi\to \omega*\xi$ in $\E$. Since $\rho_{\omega_n}\to \rho_{\omega}$, we also have
$$\kket{\omega_n*\xi}=\kket{\xi}(1_B\otimes\rho_{\omega_n})\to \kket{\xi}(1_B\otimes\rho_\omega).$$
This implies that $(\omega_n*\xi)$ is a Cauchy sequence with respect to $\|\cdot\|_\si$. By Lemma~5.28 in \cite{Buss-Meyer:Square-integrable}, this sequence converges to some $\eta\in \E_\si$. In particular, $\omega_n*\xi\to \eta$ in $\E$. It follows that $\omega*\xi=\eta\in \E_\si$. Moreover,
\begin{equation*}
    \kket{\omega*\xi}=\kket{\eta}=\lim\limits_n\kket{\omega_n*\xi}
		=\lim\limits_n\kket{\xi}(1_B\otimes\rho_{\omega_n})=\kket{\xi}(1_B\otimes\rho_\omega).\qedhere
\end{equation*}
\end{proof}

\begin{remark}\label{500} Let us return to the group case, that is, $\G=\cont_0(G)$, where $G$ is some locally compact group.
There is a small difference of
convention with respect to the modular element $\mo$ of $\G=\cont_0(G)$,
in the sense that it is not given by the modular function $\dtg_G$ of $G$, but by its inverse, that is, by
the function $t\mapsto \mo_G(t)^{-1}$ (see comments after Definition~1.9.1 in \cite{Vaes:Thesis}). It follows that $L^1_0(\G)$ corresponds to
$$L^1_0(G)=\{\omega\in L^1(G):\dtg_G^{-\frac{1}{2}}\cdot\omega\in L^1(G)\},$$
where $\cdot$ denotes pointwise multiplication.
Given $\omega\in L^1_0(G)$, the operator $\rho_\omega\in \Ls\bigl(L^2(G)\bigr)$ corresponds to the operator given by right convolution with $\omega$.
Thus, for groups, Proposition~\ref{358} says exactly what have already seen in the beginning of this section.
\end{remark}

Before finishing this section, we want to relate co-amenability of $\G$ (as defined in \cite{Bedos-Tuset:Amenability_co-amenability}) with the existence of bounded approximate units for $L^1_0(\G)$.

\begin{proposition}\label{531} The Banach algebra $L^1_0(\G)$ has a bounded approximate unit if and only if $\G$ is co-amenable.
\end{proposition}
\begin{proof}
Suppose that $\G$ is co-amenable. Then one can find an approximate unit $(\omega_i)$ for $L^1(\G)\cong M_*$ consisting of normal states, where $M:=\G''$ (see \cite[Theorem~2]{Hu-Neufang-Ruan:Multipliers}). Since $M$ is in standard form, each $\omega_i$ has the form $\omega_i=\omega_{\xi_i,\xi_i}$, where $\xi_i\in H$ are unit vectors. By the Banach--Alaoglu Theorem, we may assume (by passing to a subnet, if necessary) that $\omega_i(x)\to \epsilon(x)$ for all $x\in M$, where $\epsilon\in M^*$ is some state whose restriction
to $\G$ is (necessarily) the counit of $\G$ (see the proof of \cite[Theorem~3.1]{Bedos-Tuset:Amenability_co-amenability}). In particular,
$$\epsilon(x)=\lim\limits_i\omega_i(x)=\lim\limits_i\braket{\xi_i}{x\xi_i}\quad\mbox{ for all }x\in \M(\G).$$

Let $e\in \M(\G)$ with $\epsilon(e)=1$. We claim that $\|e\omega_i-\omega_i\|\to 0$. In fact, recall that $\epsilon$ is a $*$-homomorphism. Thus
$$\|e\xi_i-\xi_i\|^2=\braket{\xi_i}{e^*e\xi_i}-\braket{\xi_i}{e^*\xi_i}-\braket{\xi_i}{e\xi_i}+1\to\epsilon(e^*e)-\epsilon(e^*)-\epsilon(e)+1=0.$$ Hence
$$\|e\omega_i-\omega_i\|=\|\omega_{\xi_i,e\xi_i}-\omega_{\xi_i,\xi_i}\|\leq\|e\xi_i-\xi_i\|\to 0.$$
Note that this implies that $(e\omega_i)$ is also a (bounded) approximate unit for $L^1(\G)$.
Now suppose, in addition, that $e$ is a
right multiplier of $\mo^{\frac{1}{2}}$ (for instance, one can take $e=e_n$ defined by Equation~\eqref{307}, for any $n\in \Nat$).
Then, for all $\omega\in L^1(\G)$, we have $e\omega\in L^1_0(\G)$ and
$$\|e\omega\|_0\leq\max\left\{\|e\|,\|\mo^{\frac{1}{2}}e\|\right\}\|\omega\|.$$
In other words, $\omega\mapsto e\omega$ is a bounded linear map $L^1(\G)\to L^1_0(\G)$.
Note that $\epsilon(\mo^{\frac{1}{2}}e)=1$ (this follows from the relations $\Dt(\mo)=\mo\otimes \mo$ and
$(\epsilon\otimes\id)\circ\Dt=\id$). From the claim we have just proved
above (applied to $\mo^{\frac{1}{2}}e$),
it follows that $\|\mo^{\frac{1}{2}}e\omega_i-\omega_i\|\to 0$ and therefore $(\mo^{\frac{1}{2}}e\omega_i)$ is also a (bounded) approximate unit for $L^1(\G)$.
To complete the proof, we show that the (bounded) net $(e\omega_i)$ is an approximate unit for $L^1_0(\G)$.
In fact, by Equation~\eqref{310} and the fact that the nets $(e\omega_i)$ and $(\mo^{\frac{1}{2}}e\omega_i)$
are approximate units for $L^1(\G)$, we get
$$\|(e\omega_i)\cdot\omega-\omega\|_0
\leq\|(e\omega_i)\cdot\omega-\omega\|+\|(\mo^{\frac{1}{2}}e\omega_i)\cdot(\mo^{\frac{1}{2}}\omega)-\mo^{\frac{1}{2}}\omega\|\to 0$$
for any $\omega\in L^1_0(\G)$. Analogously, $\|\omega\cdot(e\omega_i)-\omega\|_0\to 0$ for all $\omega\in L^1_0(\G)$.
\end{proof}

\section{Relative continuity and generalized fixed point algebras}

Throughout this section we fix a locally compact quantum group $\G$.
We also fix a \emph{$\G$-\cstar{}algebra} $B$, that is, a \cstar{}algebra $B$ endowed with a continuous coaction $\co_B\colon B\to \M(B\otimes\G)$ of $\G$. Recall that the reduced crossed product is the \cstar{}subalgebra of $\Ls(B\otimes H)$ defined as the closed linear span of operators of the form $\co_B(b)(1\otimes x)$ with $b\in B$ and $x\in \cdualG$. It is easy to see that these operators are equivariant with respect to the coaction on $B\otimes H$ defined in Equation~\eqref{136}. In other words, $B\rtimes_\red \cdualG$ is a \cstar{}subalgebra of $\Ls^\G(B\otimes H)$, the space of all $\G$-equivariant operators on $B\otimes H$.

Now let $\E$ be a $\G$-equivariant $B$-module (henceforth also called a Hilbert $B,\G$-module), meaning a Hilbert $B$-module equipped with a $\G$-coaction compatible with $\co_B$. Recall that, given $\xi,\eta\in \E_\si$, the bra-ket operators $\bbraket{\xi}{\eta}\defeq\bbra{\xi}\circ\kket{\eta}$ (that is, the composition of the operators $\bbra{\xi}\in \Ls(\E,B\otimes H)$ and $\kket{\eta}\in \Ls(B\otimes H,\E)$) live in $\Ls(B\otimes H)$. These operators are also equivariant, that is, $\bbraket{\xi}{\eta}\in \Ls^\G(B\otimes H)$.
Therefore it is natural to ask whether $\bbraket{\xi}{\eta}\in B\rtimes_\red\cdualG$. This turns out to be a crucial question and is therefore turned into a definition:

\begin{definition}\label{574} Let $B$ be a $\G$-\cstar{}algebra,
let $\E$ be a Hilbert $B,\G$-module, and suppose that $\xi,\eta\in \E_\si$. We say that the pair $(\xi,\eta)$ is
\emph{relatively continuous}\index{relatively continuous!pair of elements},
and write $\xi\rc\eta$, if $\bbraket{\xi}{\eta}\in B\rtimes_\red\cdualG$. A subset $\R\sbe\E_\si$ is called
\emph{relatively continuous}\index{relatively continuous!subset} if $\xi\rc\eta$ for all $\xi,\eta\in \R$, that is,
$$\bbraket{\R}{\R}:=\{\bbraket{\xi}{\eta}:\xi,\eta\in \R\}\sbe B\rtimes_\red\cdualG.$$
For a relatively continuous subset $\R$ of $\E$, we define the following subspaces:
$$\F(\E,\R):=\cspn(\kket{\R}\circ B\rtimes_\red\cdualG)\sbe\Ls(B\otimes H,\E),$$
$$\I(\E,\R):=\cspn(B\rtimes_\red\cdualG \circ\bbraket{\R}{\R}\circ B\rtimes_\red\cdualG)\sbe B\rtimes_\red\cdualG,$$
and the \emph{generalized fixed point algebra}
$$\Fix(\E,\R):=\cspn(\kket{\R}\circ B\rtimes_\red\cdualG\circ \bbra{\R})\sbe\Ls(\E).$$
We say that $\R$ is \emph{saturated}\index{relatively continuous!saturated subset} if $\I(\E,\R)=B\rtimes_\red\cdualG$.
\end{definition}

\begin{remark} {\bf (1)} Relative continuity was first defined by Exel \cite{Exel:SpectralTheory} in the case of Abelian groups and
was generalized to non-Abelian groups by Meyer in \cite{Meyer:Generalized_Fixed}. Our definition generalizes Meyer's definition to quantum groups.
Exel's definition involves integrable elements instead of square-integrable elements as above, but this turns out to be essentially equivalent
(see \cite{Buss-Meyer:Continuous}).

{\bf (2)} Relative continuity is \emph{not} an equivalence relation. For instance, it is not true, in general,
that $\xi\rc\xi$. Of course, it is a symmetric relation, that is, if $\xi\rc\eta$, then $\eta\rc\xi$.
Note that it is not transitive, that is, the conditions $\xi\rc\eta$ and $\eta\rc\zeta$ do not imply that $\xi\rc\zeta$ because
we always have $\xi\rc 0$ and $0\rc\zeta$.

{\bf (3)} Observe that we are assuming continuity of the coaction of $\G$ on $B$ (because this is necessary to
define the reduced crossed product $B\rtimes_\red\cdualG$), but we do not assume that the coaction of $\G$ on $\E$ is continuous. The reason for this is that the coaction on $H$ (or more generally on $B\otimes H$) is not continuous in general, except if $\G$ is regular; and this coaction will turn out to be the most important example in order to develop the theory.

{\bf (4)} Let $\E$ be a Hilbert $B,\G$-module, and assume that $\R\sbe\E$ is a relatively continuous subset.
It is clear from the definition above that $\I(\E,\R)$ is an ideal of $A:=B\rtimes_\red\cdualG$ and $\Fix(\E,\R)$ is a \cstar{}subalgebra of $\Ls(\E)$.
Let $\F:=\F(\E,\R)$.
Since $\bbraket{\R}{\R}\sbe A$, we have
\begin{equation}\label{095}
\cspn\kket{\R}\sbe\F.
\end{equation}
In fact, let $(e_i)$ be an approximate unit for $A$. If $T$ is
an operator on $B\otimes H$ such that $T^*T\in A$, then we have
$$\|Te_i-T\|^2=\|e_iT^*Te_i-e_iT^*T-T^*Te_i+T^*T\|\to 0.$$
Note that, by definition, we have
\begin{equation}\label{507}
\I(\E,\R)=\cspn\F^*\circ\F\quad\mbox{and}\quad\Fix(\E,\R)=\cspn\F\circ\F^*.
\end{equation}
In particular, we get
\begin{equation}\label{391}
\cspn\bbraket{\R}{\R}\sbe\I(\E,\R)\quad\mbox{and}\quad\cspn\kket{\R}\bbra{\R}\sbe\Fix(\E,\R).
\end{equation}
We are going to see later that the inclusions \eqref{095} and \eqref{391} become equalities if we impose more conditions on $\R$.
\end{remark}

\begin{proposition}\label{079} Let $\E$ be a Hilbert $B,\G$-module, let $\R\sbe\E_\si$ be a relatively continuous subset
and denote $A:=B\rtimes_\red\cdualG$. Then $\F:=\F(\E,\R)\sbe \Ls^\G(B\otimes H,\E)$ is a \emph{concrete Hilbert $A$-module} \textup(as defined in \cite[Section~5]{Meyer:Generalized_Fixed} in the group case\textup), meaning that $\F\circ A\sbe\F$ and $\F^*\circ\F\sbe A$. Moreover, if $\R$ is dense in $\E$, then $\F$ is \emph{essential} in the sense that $\cspn\F(B\otimes H)=\E$.
\end{proposition}
\begin{proof} Since $A\sbe\Ls^\G(B\otimes H)$, and since the bra-ket operators are $\G$-equivariant, we have $\F\sbe \Ls^\G(B\otimes H,\E)$. From the definition of $\F$, it is obvious that $\F\circ A\sbe \F$ and $\F^*\circ\F\sbe A$, and hence $\F$ is a concrete Hilbert $A$-module. Now, suppose that $\R$ is dense in $\E$. Since $A$ is a nondegenerate \cstar{}subalgebra of $\Ls(B\otimes H)$, we have $A(B\otimes H)=B\otimes H$. And since $\R$ is dense in $\E$, Lemmas~5.17 and 6.2 in \cite{Buss-Meyer:Square-integrable} imply that $\cspn  \kket{\R}(B\otimes H)=\E$. Therefore
$$\cspn\bigl(\F(B\otimes H)\bigr)=\cspn\bigl(\kket{\R}\circ A(B\otimes H)\bigr)=\cspn\bigl(\kket{\R}(B\otimes H)\bigr)=\E.$$
Hence $\F$ is essential.
\end{proof}

Proposition~\ref{079} and Equation~\eqref{507} show that $\Fix(\E,\R)$ is contained in $\Ls^\G(\E)$. Since $\Ls^\G(\E)$ is (under the canonical identification $\Ls(\E)\cong\M(\K(\E))$) the multiplier fixed point algebra $\M_1\bigl(\K(\E)\bigr)=\{x\in\M(\K(\E)):\co_{\K(\E)}(x)=x\otimes 1\}$, we see that the elements of $\Fix(\E,\R)$ are fixed by the coaction of $\K(\E)$ and $\Fix(\E,\R)$ is a \cstar{}subalgebra of $\M_1\bigl(\K(\E)\bigr)$. Note that, by \cite[Theorem~5.2]{Meyer:Generalized_Fixed} (or rather by its obvious generalization to quantum groups), $\Fix(\E,\R)$ is a nondegenerate \cstar{}subalgebra of $\Ls(\E)$ if and only if $\F(\E,\R)$ is essential. For instance, this is the case if $\R$ is dense.

\begin{proposition}\label{111} Let $\E$ be a Hilbert $B,\G$-module, and let $\R\sbe\E_\si$ be a relatively continuous subset of $\E$.
Then $\F(\E,\R)$ is a Morita equivalence between the generalized fixed point algebra
$\Fix(\E,\R)$ and the ideal $\I(\E,\R)$ in $B\rtimes_\red\cdualG$.
\end{proposition}
\begin{proof} Theorem~5.2 in \cite{Meyer:Generalized_Fixed} yields a canonical identification $\K(\F)\cong \cspn\F\circ\F^*=\Fix(\E,\R)$, where $\F:=\F(\E,\R)$. And by definition of the $A$-valued inner product on $\F$: $\braket{x}{y}\defeq x^*\circ y$, we have $\cspn\{\<x|y\>:x,y\in\F\}=\cspn\F^*\F=\I(\E,\R)$.
\end{proof}

In the situation above, $\R$ is, by definition, saturated if and only if $\I(\E,\R)$ is the entire reduced crossed product $B\rtimes_\red\cdualG$.
Thus, in this case, $\F(\E,\R)$ is a Morita equivalence between $\Fix(\E,\R)$ and $B\rtimes_\red\cdualG$.

For a locally compact group $G$, $\cont_c(G)$ is a relatively continuous subspace of the
$G$-\cstar{}algebra $\cont_0(G)$ endowed with the translation action $\alpha_t(f)|_s=f(st)$ for all $t,s\in G$ and $f\in \cont_0(G)$. This action
is equivalent to the action obtained by considering $\cont_0(G)$ as a quantum group and letting it coact on itself by the comultiplication.

Thus, it is natural to ask what happens in the general case of a locally compact quantum group $\G$ coacting on itself by the comultiplication.
The comultiplication, viewed as a coaction of $\G$ on itself, is always integrable for locally compact quantum group (this is a special case of Proposition~4.20 in \cite{Buss-Meyer:Square-integrable}), but here is a first place where we see a difference between integrability and \emph{continuous integrability} (meaning the existence of a dense relatively continuous subspace):

\begin{proposition}\label{067} Let $\G$ be a locally compact quantum group and let $\G$ coact on itself by the comultiplication $\Dt$.
Then there is a non-zero relatively continuous subset of $\G$ if and only if $\G$ is semi-regular.
In this case, any subset $\R\sbe\G_\si$ is relatively continuous \textup(in particular, $\G_\si$ itself is relatively continuous\textup)
and
$$\F(\G,\R)=(1_\G\otimes H_0^*)W\sbe\Ls(\G\otimes H,\G),\footnote{Here $H_0^*$ denotes the set of
all $\xi^*\in \Ls(H,\C)$, with $\xi\in H_0$, where $\xi^*$ denotes the element of $\Ls(H,\C)$ given by $\xi^*(\eta)=\braket{\xi}{\eta}$ for all $\eta\in H$.}$$
where $H_0:=\cspn\bigl(\cdualG\La(\G\R^*)\bigr)\sbe H$ and $W\in \Ls(\G\otimes H)$ is the left regular corepresentation of $\G$ \textup(see \cite[Section~4]{Kustermans-Vaes:LCQG} for its definition\textup). In particular, if $\R\not=\{0\}$, then
$$\Fix(\G,\R)=\C 1_\G\cong\C\quad\mbox{and}\quad\I(\G,\R)=W^*(1\otimes\K(H_0))W\cong\K(H_0).$$
There is a saturated, relatively continuous subset of $\G$ if and only if $\G$ is regular.
\end{proposition}
Recall that a quantum group $\G$ is called \emph{semi-regular} if the \cstar{}subalgebra $C:=\cspn(\G\cdualG)\sbe \Ls(H)$ contains $\K(H)$ and it is called \emph{regular} if $C=\K(H)$.
\begin{proof} Given $x\in \G^+$, from the left invariance of $\f$, it follows that $x\in \dom(\f)$ if and only if $\Delta(x)\in \dom(\id\otimes\f)$ (see Result~2.4 and Proposition~5.15 in \cite{Kustermans-Vaes:LCQG}). In other words, $x\in \G_\ii^+$ if and only if $x\in \dom(\f)$.
Thus, for $\xi\in \G$, we have $\Delta(\xi)\in \dom(\id\otimes\La)$ if and only if $\xi\in \G_\si$ if and only if $\xi\xi^*\in \dom(\f)$, that is, $\G_\si=\dom(\La)^*$. Moreover, from \cite[Result~2.10]{Kustermans-Vaes:LCQG} (or also \cite[Proposition~1.T2.3]{Vaes:Thesis}), it follows that
$$(\id\otimes\La)(\Delta(x))=W^*(1\otimes\La(x))\quad\mbox{for all }x\in \dom(\La).$$
Therefore,
\begin{equation}\label{eq:FormulaBraOperatorA=G}
\bbra{\xi}=(\id\otimes\La)(\Delta(\xi^*))=W^*(1\otimes\La(\xi^*))\quad\mbox{for all }\xi\in \G_\si.
\end{equation}
It follows that
\begin{equation}\label{505}
\bbraket{\xi}{\eta}=W^*(1\otimes|\Lambda(\xi^*)\>\<\Lambda(\eta^*)|)W\in W^*(1\otimes\K(H))W\quad\mbox{for all }\xi,\eta\in \G_\si,
\end{equation}
where, for $u,v\in H$, $\ket{u}\bra{v}$ denotes the compact operator on $H$ defined by $\ket{u}\bra{v}(w)=u\braket{v}{w}$.

On the other hand, since $\Delta(x)=W^*(1\otimes x)W$ and $W\in \M(\G\otimes\dualG)$, it follows that
\begin{equation}\label{506}
\G\rtimes_\red\cdualG=W^*(1\otimes C)W,
\end{equation}
where $C:=\cspn(\G\cdualG)$. Thus the existence of a non-zero relatively continuous subset implies $C\cap
\K(H)\not=\{0\}$. Since $C$ is an irreducible \cstar{}subalgebra of $\Ls(H)$, the condition $C\cap\K(H)\not=\{0\}$ is equivalent to semi-regularity (see also \cite[Proposition~5.6]{Baaj-Skandalis-Vaes:Non-semi-regular}). Conversely, if $\G$ is semi-regular, then $\K(H)\sbe C$ and hence any subset
$\R\sbe\G_\si=\dom(\La)^*$ is relatively continuous by the same calculation above. Moreover, by Equations~\eqref{eq:FormulaBraOperatorA=G} and ~\eqref{506}, we have
$$\F(\G,\R)=\cspn\bigl((1\otimes\La(\R^*)^*)W(\G\rtimes_\red\cdualG)\bigr)=\cspn\bigl((1\otimes\La(\R^*)^*C)W\bigr)=(1\otimes H_0^*)W.$$

Finally, we prove that there is a saturated relatively continuous subset $\R\sbe\G$ if and only if $\G$ is regular.
In fact, if $\G$ is regular, then $\K(H)=C$ and hence for $\R=\dom(\La)^*$ we get
$$\cspn\bbraket{\R}{\R}=W^*(1\otimes (\cspn\{|\Lambda(\xi)\>\<\Lambda(\eta)|:\xi,\eta\in \dom(\La)\}))W$$
$$=W^*(1\otimes \K(H))W=\G\rtimes_\red\cdualG.$$
It follows from Equation~\eqref{391} that $\I(\G,\R)=\G\rtimes_\red\cdualG$, that is, $\R$ is saturated. Conversely, suppose that $\R\sbe
\G_\si=\dom(\La)^*$ is relatively continuous and saturated. In particular, $\R\not=\{0\}$ and hence $\G$ is semi-regular.
It follows from Equations~\eqref{505} and~\eqref{506} that
\begin{align*}
\G\rtimes_\red\cdualG&=\I(\G,\R)=\cspn\bigl(\G\rtimes_\red\cdualG(\bbraket{\R}{\R})\G\rtimes_\red\cdualG\bigr)
\\&= \cspn\bigl(W^*(1\otimes C)(\ket{\La(\R^*)}\bra{\La(\R^*)})(1\otimes C)W\bigr)
\\&\sbe W^*(1\otimes \K(H))W\sbe \G\rtimes_\red\cdualG.
\end{align*}
In the last inclusion above we have used the semi-regularity of $\G$.
We conclude that $W^*(1\otimes \K(H))W=\G\rtimes_\red\cdualG$
and this is equivalent to the regularity of $\G$.
\end{proof}

Next, we analyze the $\G$-Hilbert space $H=L^2(\G)$. Recall that the coaction on $H$ is given
by $\co_H(\xi)=\hat W^*(\xi\otimes 1)$ for all $\xi\in H$, where $\hat W$ is the left regular corepresentation of the dual $\dualG$.
We already know that $H$ is square-integrable. In fact, this will follow again from the result below where we
show that we can always find a dense, relatively continuous subspace of $H$.

Before stating the result we need some preparation. Recall that $\G$ is equal to the closure in $\Ls(H)$ of the space of
the operators $\dualla(\omega)=(\omega\otimes\id)(\hat W)$ with $\omega\in \Ls(H)_*$.
Similarly, the dual $\dualG$ of $\G$ is given by the closure of the operators $\la(\omega)=(\omega\otimes\id)(W)$ with $\omega\in \Ls(H)_*$.
By Theorem~1.11.13 in \cite{Vaes:Thesis},
the dual left Haar weight $\hat\f$ of $\dualG$ has a GNS-construction of the form $(H,\hat\iota,\hat\La)$, where
$\hat\iota$ denotes the inclusion $\dualG\hookrightarrow\Ls(H)$.

Let $\T_{\hat\f}\sbe\dualG$ be the Tomita $*$-algebra of the dual left Haar weight $\hat\f$:
\begin{multline}\label{501}
\T_{\hat\f}:=\{x\in\dualG:x\mbox{ is analytic with respect to }\hat{\sigma} \\
\mbox{ and }\hat{\sigma}_z(x)\in \dom(\hat{\La})\cap\dom(\hat{\La})^*, \mbox{ for all }z\in\C\}.
\end{multline}
We need the following result from \cite[Proposition~1.11.25]{Vaes:Thesis} (applied to the dual $\dualG$, and assuming the inner products to be linear in the second variable):

\begin{lemma}\label{269}
For every $a\in \T_{\hat\f}$ and $\eta\in H$, we have
$$\hat\lambda(\omega_{\hat\La(a),\eta})\in
\dom(\La)\quad\mbox{and}\quad\La\bigl(\dualla(\omega_{\hat\La(a),\eta})\bigr)=\dualJ \hat\sigma_{\frac{\ii}{2}}(a)\dualJ \eta,$$
where $\hat\sigma$ is the modular group of $\hat\f$ and $\dualJ$ is the modular conjugation of $\hat\f$ in the
GNS-construction $(H,\hat\iota,\hat\La)$.
\end{lemma}

\begin{proposition}\label{268} Let $\G$ be a locally compact quantum group and consider $H=L^2(\G)$ with the coaction $\co_H$
defined above.
Then $\R:=\hat\La(\T_{\hat\f})$ is a dense, relatively continuous subspace of $H$ and we have
$\kket{\xi}=\dualJ \hat\sigma_{\frac{\ii}{2}}(a)^*\dualJ$ for all $\xi=\hat\La(a)\in \R$. Moreover,
$$\F(H,\R)=\I(H,\R)=\Fix(H,\R)=\cdualG.$$
\end{proposition}
\begin{proof}
By definition, we have
$$\xi\in H_\si\Longleftrightarrow \co_H(\xi)^*(\eta\otimes 1)\in\dom(\Lambda),\,\forall\, \eta\in H\Longleftrightarrow $$
$$(\xi^*\otimes 1)\hat W(\eta\otimes 1)=(\omega_{\xi,\eta}\otimes\id)(\hat W)=\hat\lambda(\omega_{\xi,\eta})\in \dom(\Lambda),\,\forall\,\eta\in H.$$
Lemma~\ref{269} implies that $\xi:=\hat\La(a)\in H_\si$
for all $a\in \T_{\hat\f}$, and
$$\bbra{\xi}\eta=\dualJ \hat\sigma_{\frac{\ii}{2}}(a)\dualJ \eta,\quad\forall\, \eta\in H.$$
In other words, we have $\kket{\xi}=\dualJ \hat\sigma_{\frac{\ii}{2}}(a)^*\dualJ $.
Moreover, since
$\dualJ \hat\sigma_{\frac{\ii}{2}}(a)^*\dualJ \in\cdualG=\C\rtimes_\red\cdualG$,
we get that $\R=\hat\La(\T_{\hat\f})$ is a dense,
relatively continuous subspace of $H$. Since $\hat
J\hat\sigma_{\frac{\ii}{2}}(\T_{\hat\f})^*\dualJ $ is dense in $\cdualG$, we conclude that
$$\F(H,\R)=\cspn(\kket{\R}\cdualG)=\cspn(\dualJ \hat\sigma_{\frac{\ii}{2}}(\T_{\hat\f})^*\dualJ \cdualG)=\cdualG.$$
And hence $\Fix(H,\R)=\I(H,\R)=\cdualG$.
\end{proof}

Next, we consider one of the most important examples, namely, the Hilbert $B,\G$-module $B\otimes L^2(\G)$,
where $B$ is some fixed $\G$-\cstar{}algebra. Recall that we always consider $B\otimes L^2(\G)$ endowed with the coaction
defined by Equation~\eqref{136}:
\begin{equation*}
\co_{B\otimes H}(\zeta)=\Sigma_{23}
W_{23}(\co_B\otimes\id)(\zeta),\quad\zeta\in B\otimes H.
\end{equation*}

\begin{proposition}\label{270} Let $B$ be a $\G$-\cstar{}algebra.
Then $\R:=B\odot\hat\La(\T_{\hat\f})$ is a dense, relatively continuous subspace of the Hilbert $B,\G$-module
$B\otimes H$, and
$$\kket{b\otimes\xi}=\bigl(1_B\otimes\dualJ \hat\sigma_{\frac{\ii}{2}}(a)^*\dualJ\bigr)\co_B(b)
\quad\mbox{for all }b\in B\mbox{ and } \xi=\hat\La(a)\in \hat\La(\T_{\hat\f}).$$
Moreover,
$\F(B\otimes H,\R)=\I(B\otimes H,\R)=\Fix(B\otimes H)=B\rtimes_\red\cdualG.$
\end{proposition}
\begin{proof}
This is a special case of Proposition~\ref{173} below, so we omit the proof.
\end{proof}

Recall that given a Hilbert $B,\G$-module $(\E,\co_\E)$, $\E\rtimes_\red\cdualG$ denotes the closed linear span of $\co_\E(\E)(1_B\otimes\cdualG)$ in $\Ls(B\otimes H,\E\otimes H)$, where the embedding $\G\into \Ls(H)$ is used to view $\M(\E\otimes\G)$ as a subspace of $\Ls(B\otimes H,\E\otimes H)$.
If the coaction of $\G$ on $\E$ is continuous, $\E\rtimes_\red\cdualG$ is also the closed linear span of $(1_{\K(\E)}\otimes\cdualG)\co_\E(\E)$ and
we have a canonical isomorphism $\K(\E\rtimes_\red\cdualG)\cong\K(\E)\rtimes_\red\cdualG$ (for the canonical coaction on $\K(\E)$ induced by $\co_\E$).

We also consider on the Hilbert $B$-module $\E\otimes H$ the following coaction of $\G$:
$$\co_{\E\otimes H}(\zeta)=\Sigma_{23}W_{23}(\co_\E\otimes\id_H)(\zeta),\quad \zeta\in \E\otimes H,$$
where $\Sigma:\G\otimes H\to H\otimes\G$ is the flip operator. Notice that this is a generalization of the coaction on $B\otimes H$ defined by Equation~\eqref{136}. Thus, the following result generalizes Proposition~\ref{270}.

\begin{proposition}\label{173} Let $\E$ be a Hilbert $B,\G$-module and consider on $\E\otimes H$ the coaction of $\G$ defined above.
If $\xi\in \E$ and $v\in H_\si$, then
$\xi\otimes v\in (\E\otimes H)_\si$ and
$$\kket{\xi\otimes v}=(1_\E\otimes\kket{v})\co_\E(\xi),$$
where here we view $\co_\E(\xi)\in \M(\E\otimes\G)$ as an element of $\Ls(B\otimes H,\E\otimes H)$ using the representation $\G\into\Ls(H)$.
Suppose that the coaction of $\G$ on $\E$ is continuous. Then $\R:=\E\odot\hat\La(\T_{\hat\f})$
is a dense, relatively continuous subspace of $\E\otimes H$, and
\begin{multline*}
\F(\E\otimes H,\R)=\E\rtimes_\red\cdualG,\,\,\,\Fix(\E,\R)=\K(\E\rtimes_\red\cdualG)\cong\K(\E)\rtimes_\red\cdualG,\,\,\,\I(\E,\R)=I\rtimes_\red\cdualG,
\end{multline*}
where $I:=\cspn\braket{\E}{\E}_B$ (this is a $\G$-invariant ideal in $B$). In particular, if $\E$ is full, then $\R$ is saturated.
\end{proposition}
\begin{proof}
If $v\in H_\si$, then $\ket{v}\bra{v}\in \K(H)_\ii$ and hence $T\otimes \ket{v}\bra{v}\in (\K(\E)\otimes\K(H))_\ii$ for all $T\in \K(\E)$ because the canonical homomorphism $\K(\E)\to\M(\K(\E)\otimes\K(H))$ is nondegenerate and $\G$-equivariant. In particular, $\ket{\xi}\bra{\xi}\otimes\ket{v}\bra{v}\in (\K(\E)\otimes\K(H))_\ii$ for all $\xi\in \E$. It follows from \cite[Proposition~5.20]{Buss-Meyer:Square-integrable} that $\xi\otimes v\in (\E\otimes H)_\si$.

To compute $\bbra{\xi\otimes v}$, first note that
\begin{align*}
\co_{\E\otimes H}(\xi\otimes v)&=\Sigma_{23}W_{23}(\co_\E\otimes\id_H)(\xi\otimes v)
							 \\&=\Sigma_{23}W_{23}(\co_\E(\xi)\otimes v)
							 \\&=\Sigma_{23}W_{23}(1_\E\otimes 1_\G\otimes v)\co_\E(\xi)
							 \\&=(1_\E\otimes\Sigma W(1_\G\otimes v))\co_\E(\xi)
							 \\&=(1_\E\otimes\co_H(v))\co_\E(\xi).
\end{align*}							
Now, if $\eta\in \E$ and $\zeta\in H$, then
\begin{align*}
\bbra{\xi\otimes v}(\eta\otimes\zeta)&=(\id_{B}\otimes\La)\bigl(\co_{\E\otimes H}(\xi\otimes v)^*(\eta\otimes\zeta\otimes 1)\bigr)
				\\&=(\id_{B}\otimes\La)\bigl(\co_\E(\xi)^*(\eta\otimes\co_H(v)^*(\zeta\otimes 1))\bigr)
                \\&=(\id_{B}\otimes\La)\bigl(\co_\E(\xi)^*(\eta\otimes 1)(1\otimes \co_H(v)^*(\zeta\otimes 1))\bigr)
                \\&=\co_\E(\xi)^*(\eta\otimes 1)(\id_{B}\otimes\La)\bigl(1\otimes \co_H(v)^*(\zeta\otimes 1)\bigr)
				\\&=\co_\E(\xi)^*\bigl(\eta\otimes \La(\co_H(v)^*(\zeta\otimes 1))\bigr)
				\\&=\co_\E(\xi)^*(1_\E\otimes\bbra{v})(\eta\otimes\zeta).
\end{align*}
Therefore, $\bbra{\xi\otimes v}=\co_\E(\xi)^*(1_\E\otimes\bbra{v})$, that is, $\kket{\xi\otimes v}=(1_\E\otimes \kket{v})\co_\E(\xi)$.
If $v\in \hat\La(\T_{\hat\f})$, then we know from Proposition~\ref{268} that $\kket{v}\in \cdualG$.
Thus, if $\co_\E$ is continuous,
then
$$\kket{\xi\otimes v}=(1_\E\otimes \kket{v})\co_\E(\xi)\in (1\otimes\cdualG)\co_\E(\E)\sbe\E\rtimes_\red\cdualG.$$
Since $\E\rtimes_\red\cdualG$ is a Hilbert $B\rtimes_\red\cdualG$-module,
it follows that $\R=\E\odot\hat\La(\T_{\hat\f})$ is a dense relatively
continuous subspace of $\E\otimes H$ and
\begin{align*}
\F(\E\otimes H,\R)&=\cspn\bigl(\kket{\R} (B\rtimes_\red\cdualG)\bigr)
				\\&=\cspn\bigl((1_\E\otimes \kket{\hat\La(\T_{\hat\f})})\co_\E(\E)(B\rtimes_\red\cdualG)\bigr)
				\\&=\cspn\bigl((1_\E\otimes \cdualG)\co_\E(\E)(B\rtimes_\red\cdualG)\bigr)
				\\&=\cspn\bigl((\E\rtimes_\red\cdualG)(B\rtimes_\red\cdualG)\bigr)=\E\rtimes_\red\cdualG.
\end{align*}
It follows that
\begin{equation*}
\Fix(\E,\R)=\K(\E\rtimes_\red\cdualG)\cong\K(\E)\rtimes_\red\cdualG\quad\mbox{and}\quad \I(\E,\R)=I\rtimes_\red\cdualG.\qedhere
\end{equation*}
\end{proof}

\begin{remark} Let notation be as in Proposition~\ref{173}.
For each $\xi\in \E$, the operator $\co_\E(\xi)\in \M(\E\otimes\G)$, considered as an element of
$\Ls(B\otimes H,\E\otimes H)$, is $\G$-equivariant, that is, for all $\zeta\in B\otimes H$ we have
$$\co_{\E\otimes H}(\co_\E(\xi)\zeta)=(\co_\E(\xi)\otimes 1_\G)\co_{B\otimes H}(\zeta).$$
Therefore, by Proposition~\ref{003}(iii), $\co_{\E}(\xi)\zeta\in (\E\otimes H)_\si$ for all $\zeta\in (B\otimes H)_\si$, and
$$\kket{\co_\E(\xi)\zeta}=\co_{\E}(\xi)\kket{\zeta}.$$
By Proposition~\ref{270}, $\R:=B\odot\hat\La(\T_{\hat\f})$ is a relatively continuous
subspace of $B\otimes H$ and $\kket{\R}$ is dense in $B\rtimes_\red\cdualG$.
It follows that $\co_\E(\E)\R$ is a relatively
continuous subset of $\E\otimes H$ and
$$\F(\E\otimes H,\co_\E(\E)\R)=\cspn \co_\E(\E)(B\rtimes_\red\cdualG)=\E\rtimes_\red\cdualG.$$
Since the linear span of $\co_\E(\E)(B\otimes\G)$ is dense in $\E\otimes\G$, it follows that the linear span of $\co_\E(\E)\R$ is
a dense, relatively continuous subspace of $\E\otimes H$. Note that this argument does not use continuity of the coaction $\co_\E$.
\end{remark}

For a compact group $G$, every subset of a Hilbert $B,G$-module is relatively continuous.
Now we show that this is also true for compact quantum groups.

\begin{proposition}\label{398} Let $\G$ be a compact quantum group and let $\E$ be a Hilbert $B,\G$-module. Then any
subset of $\E$ is relatively continuous. In particular, $\E$ itself is relatively continuous. Moreover, we have
$$\F_\E:=\F(\E,\E)=(1_\E\otimes\dtg_1^*)\E\rtimes_\red\cdualG,$$
where $\dtg_1:=\La(1)\in H$ and $\dtg_1^*$ denotes the element of $\Ls(H,\C)$ given by $\dtg_1(\eta)=\braket{\dtg_1}{\eta}$ for all
$\eta\in H$. The generalized fixed point algebra $\Fix(\E):=\Fix(\E,\E)$
is the usual fixed point algebra
$$\Fix(\E)=(1\otimes\dtg_1^*)\K(\E\rtimes_\red\cdualG)(1\otimes\dtg_1)=\{x\in \K(\E):\co_{\K(\E)}(x)=x\otimes 1_\G\},$$
and it is Morita equivalent to the ideal $\I_\E:=\I(\E,\E)\sbe B\rtimes_\red\cdualG$ given by
$$\I_\E=\cspn(\E\rtimes_\red\cdualG)^*(1\otimes p_1)(\E\rtimes_\red\cdualG)=\cspn\co_\E(\E)^*(1_\E\otimes p_1)\co_\E(\E),$$
where $p_1:=\ket{\dtg_1}\bra{\dtg_1}\in \K(H)$.
\end{proposition}
\begin{proof}
We already know that $\E=\E_\si$. Thus we have to show that, for any
$\xi,\eta\in \E$, the element $\bbraket{\xi}{\eta}$ belongs to $B\rtimes_\red\cdualG$. Recall from Example~\ref{477} that
$$\bbra{\xi}=\co_\E(\xi)^*(1_{\K(\E)}\otimes\dtg_1)\quad\mbox{for all }\xi\in \E.$$
Thus
$$\bbraket{\xi}{\eta}=\co_\E(\xi)^*(1_{\K(\E)}\otimes p_1)\co_\E(\eta).$$
We may assume that $\f$ is a state, that is, $\f(1)=1$. Thus $\dtg_1$
is a unitary vector and hence $p_1$ is a projection. Note also that $\f=\omega_{\dtg_1,\dtg_1}\in
L^1(\G)$. We claim that $p_1=\rho(\f)$ (recall that $\rho(\omega)=(\id\otimes\omega)(V^*)$, where $V$ is the right regular corepresentation of $\G$).
In fact, by Equation~\eqref{085}, we have (using that compact quantum groups are unimodular, so that $\Gamma=\La$)
$$(\id\otimes\f)(V)\La(b)=\La\bigl((\id\otimes\f)\Dt(b)\bigr)=\La\bigl(1\f(b)\bigr)=\dtg_1\f(b)$$ for all $b\in \G$. On the other hand
$$p_1\La(b)=|\dtg_1\>\<\dtg_1|\La(b)=\dtg_1\<\La(1)|\La(b)\>=\dtg_1\f(b).$$
Thus $(\id\otimes\f)(V)=p_1$ and hence
$$\rho(\f)=(\id\otimes\f)(V^*)=(\id\otimes\f)(V)^*=p_1^*=p_1.$$
In particular, $p_1\in \rho(L^1(\G))\sbe\cdualG$. We conclude that the operator
\begin{align*}
\bbraket{\xi}{\eta}&=\co_\E(\xi)^*(1_{\K(\E)}\otimes p_1)\co_\E(\eta)
				 \\&=\bigl((1_{\K(\E)}\otimes p_1)\co_\E(\xi)\bigr)^*\bigl((1_{\K(\E)}\otimes p_1)\co_\E(\eta)\bigr)
\end{align*}
belongs to $(\E\rtimes_\red\cdualG)^*(\E\rtimes_\red\cdualG)\sbe B\rtimes_\red\cdualG$.
Here we are using that compact quantum groups are regular, so that $\co_\E$ is automatically continuous (see Proposition~5.8 in \cite{Baaj-Skandalis-Vaes:Non-semi-regular}).
Therefore any subset of $\E$ is relatively continuous.

The equation
$\kket{\xi}=(1\otimes\dtg_1^*)\co_\E(\xi)$ yields
$$\F_\E=\cspn(1\otimes\dtg_1^*)\co_\E(\E)(B\rtimes_\red\cdualG)=(1\otimes\dtg_1^*)\E\rtimes_\red\cdualG.$$
Hence
$$\Fix(\E)=\cspn(1\otimes\dtg_1^*)(\E\rtimes_\red\cdualG)(\E\rtimes_\red\cdualG)^*(1\otimes\dtg_1)
															=(1\otimes\dtg_1^*)\K(\E\rtimes_\red\cdualG)(1\otimes\dtg_1),$$
which is therefore Morita equivalent to
$$\I_\E=\cspn(\E\rtimes_\red\cdualG)^*(1\otimes p_1)(\E\rtimes_\red\cdualG).$$
Since $\co_\E$ is continuous, the linear span of $L^1(\G)*\E$ is dense in $\E$. Combining this with Propositions~\ref{508} and~\ref{358}
(and using that $\G$ is unimodular, so that $L^1_0(\G)=L^1(\G)$),
we get that
\begin{align*}
\overline{\kket{\E}}&=\cspn\bigl((1\otimes\dtg_1^*)\co_\E(\E)\bigl(1\otimes\rho(L^1(\G))\bigr)\bigr)
\\&=\cspn \bigl((1\otimes\dtg_1^*)\co_\E(\E)(1\otimes\cdualG)\bigr)
\\&=(1\otimes\dtg_1^*)(\E\rtimes_\red\cdualG)=\F_\E.
\end{align*}
In particular,
$$\I_\E=\cspn\bbraket{\E}{\E}=\cspn\co_\E(\E)^*(1\otimes p_1)\co_\E(\E),$$
and (using the equality $\omega_{\dtg_1,\dtg_1}=\f$)
\begin{align*}
\Fix(\E)&=\cspn\kket{\E}\bbra{\E}
		\\&=\cspn(1\otimes\dtg_1^*)\co_{\K(\E)}(\K(\E))(1\otimes\dtg_1)
		\\&=\cspn(\id_{\K(\E)}\otimes\f)(\co_{\K(\E)}(\K(\E)))
		\\&=\{x\in \K(\E):\co_{\K(\E)}(x)=x\otimes 1_\G\},
\end{align*}
where the last equality is proved in the following way: since
$\co_{\K(\E)}(\K(\E))$ is contained in $\tilde\M\bigl(\K(\E)\otimes\G\bigr)$ (which is equal to $\K(\E)\otimes\G$ because $\G$ is unital),
and since $\f\in
\G^*$, we have
$$(\id_{\K(\E)}\otimes\f)(\co_{\K(\E)}(\K(\E)))\sbe\{x\in\K(\E):\co_{\K(\E)}(x)=x\otimes 1_\G\}.$$
Conversely, if
$\co_{\K(\E)}(x)=x\otimes 1_\G$, then
$(\id_{\K(\E)}\otimes\f)(\co_{\K(\E)}(x))=x$, and therefore
\begin{equation*}
\Fix(\E)=\{x\in \K(\E):\co_{\K(\E)}(x)=x\otimes 1_\G\}.\qedhere
\end{equation*}
\end{proof}

\begin{remark}
In the case of a $\G$-\cstar{}algebra $A$ with $\G$ compact, the Morita
equivalence between $\Fix(A)$ and the ideal $\I_A$ in
$A\rtimes_\red\cdualG$ was obtained by Ng in \cite{Ng:MoritaEquivalences}. He also defined an interesting condition on the coaction:
$\co_A$ is called \emph{effective} if the linear span of $\co_A(A)(A\otimes 1)$ is dense in $A\otimes\G$.
This condition implies that $\R=A$ is saturated, that is, $\I_A$ is equal to $A\rtimes_\red\cdualG$ (\cite[Lemma~2.6]{Ng:MoritaEquivalences}).
Thus, in this case, $\Fix(A)$ is Morita equivalent to $A\rtimes_\red\cdualG$.
Note that comultiplications are effective and hence any dual coaction is effective.
Observe that this result applied to the comultiplication $\Dt$ of $\G$ and combined with Proposition~\ref{067} yields
a well-known result: any compact quantum group is regular.
\end{remark}

The following result provides a canonical way to associate relatively continuous subspaces of $\E$ to
relatively continuous subspaces of $\K(\E)$. It also provides a formula for the corresponding Hilbert modules over the reduced crossed product and generalized fixed point algebras.

\begin{proposition}\label{171} Let $\G$ be a locally compact quantum group and let
$\E$ be a Hilbert $B,\G$-module with a continuous coaction of $\G$.
\begin{enumerate}
\item[\textup{(i)}] Suppose that there is a left action $\pi:A\to \Ls(\E)$ of a $\G$-\cstar{}algebra $A$
turning $\E$ into a $\G$-equivariant right-Hilbert $A,B$-bimodule. This means that $\pi$ is a $\G$-equivariant nondegenerate $*$-homomorphism.
We will use the module notation for the left action\textup: $a\cdot\xi:=\pi(a)\xi$ for all $a\in A$ and $\xi\in \E$.

If $\R$ is a relatively continuous subset of $A$, then $\R\cdot\E$ is a relatively continuous subset of $\E$ and
$$\F(\E,\R\cdot\E)=\cspn\bigl(\F(A,\R)\cdot \E\rtimes_\red\cdualG\bigr)\cong \F(A,\R)\otimes_{A\rtimes_\red\cdualG}(\E\rtimes_\red\cdualG),$$
where for $x\in \F(A,\R)\sbe \Ls(A\otimes H,A)$
and $y\in \E\rtimes_\red\cdualG\sbe\Ls(B\otimes H,\E\otimes H)$ we are using the notation
$x\cdot y:=(\pi\otimes\id_{H^*})(x)y$. Observe that
$\pi\otimes\id_{H^*}:\Ls(A\otimes H,A)\to \Ls(\E\otimes H,\E)$ and therefore the
composition $(\pi\otimes\id_{H^*})(x)y$ makes sense.
Note that $\Ls(A\otimes H,A)\cong\M(A\otimes H^*)$ and $\Ls(\E\otimes H,\E)\cong\M(\K(\E)\otimes H^*)$.

In particular, if $\R$ is a relatively continuous subspace of $\K(\E)$,
then $\R(\E)$ is a relatively continuous subspace of $\E$ and
$$\F\bigl(\E,\R(\E)\bigr)=\cspn\bigl(\F(\K(\E),\R)\circ (\E\rtimes_\red\cdualG)\bigr)
\cong\F\bigl(\K(\E),\R\bigr)\otimes_{\K(\E)\rtimes_\red\cdualG}(\E\rtimes_\red\cdualG) .$$

\item[\textup{(ii)}]
If $\R$ is a relatively continuous subset of $\E$, then $|\R\>\<\E|$ is a relatively continuous subset of $\K(\E)$ and
$$\F\bigl(\K(\E),|\R\>\<\E|\bigr)=\cspn\bigl(\F(\E,\R)\circ(\E^*\rtimes_\red\cdualG)\bigr)\cong
\F(\E,\R)\otimes_{B\rtimes_\red\cdualG}(\E^*\rtimes_\red\cdualG).$$
\end{enumerate}
\end{proposition}
\begin{proof} (i) It follows from Proposition~\ref{003}(iv) that $\R\cdot \E\sbe \E_\si$ and, for all $a\in \R$ and $\xi\in \E$, we have
$$\kket{a\cdot\xi}=(\pi\otimes\id_{H^*})(\kket{a})\co_\E(\xi)=\kket{a}\cdot\co_\E(\xi).$$
Thus, for all $a,b\in \R$ and $\xi,\eta\in \E$ we get
$$\bbraket{a\cdot\xi}{b\cdot\eta}=\co_\E(\xi)^*(\pi\otimes\id_\K)(\bbraket{a}{b})\co_\E(\eta)
																=\co_\E(\xi)^*(\pi\rtimes_\red\cdualG)(\bbraket{a}{b})\co_\E(\eta),$$
where $\K:=\K(H)$. Since $\R$ is relatively continuous, we have
$\bbraket{a}{b}\in A\rtimes_\red\cdualG$. Thus to prove that
$\R\cdot\E$ is relatively continuous it is enough to prove that
$$\co_\E(\E)^*(\pi\rtimes_\red\cdualG)(A\rtimes_\red\cdualG)\co_\E(\E)\sbe B\rtimes_\red\cdualG.$$
If $c\in A$, $\hat x\in \cdualG$ and $\xi,\eta\in \E$ then
$$\co_\E(\xi)^*(\pi\rtimes_\red\cdualG)((1\otimes\hat x)\co_A(c))\co_\E(\eta)=\co_\E(\xi)^*((1\otimes\hat x)\co_{\K(\E)}(\pi(c))\co_\E(\eta)$$
$$=\co_\E(\xi)^*(1\otimes\hat x)\co_\E(\pi(c)\eta)\sbe (\E\rtimes_\red\cdualG)^*(\E\rtimes_\red\cdualG)\sbe B\rtimes_\red\cdualG.$$
Hence $\R\cdot\E$ is relatively continuous. We compute
\begin{align*}
\E\rtimes_\red\cdualG&=(A\rtimes_\red\cdualG)\cdot(\E\rtimes_\red\cdualG)
				\\&=(\pi\rtimes_\red\cdualG)(A\rtimes_\red\cdualG)(\E\rtimes_\red\cdualG)
				\\&=(\pi\otimes\id_\K)(A\rtimes_\red\cdualG)(\E\rtimes_\red\cdualG),
\end{align*}
and hence
\begin{align*}
\F(\E,\R\cdot\E)&=\cspn\kket{\R\cdot\E}(B\rtimes_\red\cdualG)
			  \\&=\cspn(\kket{\R}\cdot\co_\E(\E))(B\rtimes_\red\cdualG)
			  \\&=\cspn(\pi\otimes\id_{H^*})(\kket{\R})\co_\E(\E)(B\rtimes_\red\cdualG)
			  \\&=\cspn(\pi\otimes\id_{H^*})(\kket{\R})(\E\rtimes_\red\cdualG)
			  \\&=\cspn(\pi\otimes\id_{H^*})(\kket{\R})(A\rtimes_\red\cdualG)\cdot(\E\rtimes_\red\cdualG)
			  \\&=\cspn(\pi\otimes\id_{H^*})(\kket{\R})(\pi\otimes\id_\K)(A\rtimes_\red\cdualG)(\E\rtimes_\red\cdualG)
		      \\&=\cspn(\pi\otimes\id_{H^*})\bigl(\kket{\R}(A\rtimes_\red\cdualG)\bigr)(\E\rtimes_\red\cdualG)
			  \\&=\cspn\F(A,\R)\cdot (\E\rtimes_\red\cdualG).
\end{align*}
Finally, it is easy to see that the map
$x\otimes y\mapsto x\cdot y$, where $x\in \F(A,\R)$ and $y\in \E\rtimes_\red\cdualG$,
induces an isomorphism
\begin{equation*}
\F(A,\R)\otimes_{A\rtimes_\red\cdualG}(\E\rtimes_\red\cdualG)\cong \cspn\F(A,\R)\cdot (\E\rtimes_\red\cdualG).
\end{equation*}
(ii) By Proposition~\ref{003}(v), we have $|\R\>\<\E|\sbe \K(\E)_\si$ and, for all $\xi\in \R$, $\eta\in \E$,
$$\KKET{\ket{\xi}\bra{\eta}}=\kket{\xi}\co_\E(\eta)^*.$$
Thus, if $\xi_1,\xi_2\in \R$ and $\eta_1,\eta_2\in\E$ we get
$$\BBRAKET{\ket{\xi_1}\bra{\eta_1}}{\ket{\xi_2}\bra{\eta_2}}=\co_\E(\eta_1)\bbraket{\xi_1}{\xi_2}\co_\E(\eta_2)^*\in
(\E\rtimes_\red\cdualG)(\E\rtimes_\red\cdualG)^*\sbe \K(\E)\rtimes_\red\cdualG.$$
Thus $|\R\>\<\E|$ is relatively continuous and because
$\E^*\rtimes_\red\cdualG=(B\rtimes_\red\cdualG)(\E^*\rtimes_\red\cdualG)$
we conclude that
\begin{align*}
\F(\K(\E),|\R\>\<\E|)&=\cspn \bigl(\KKET{|\R\>\<\E|} (\K(\E)\rtimes_\red\cdualG)\bigr)
					\\&=\cspn \bigl(\kket{\R}\co_\E(\E)^*(\K(\E)\rtimes_\red\cdualG)\bigr)
                    \\&=\cspn (\kket{\R}(\E^*\rtimes_\red\cdualG))
					\\&=\cspn \bigl(\kket{\R}(B\rtimes_\red\cdualG)(\E^*\rtimes_\red\cdualG)\bigr)
					\\&=\cspn\F(\E,\R)(\E^*\rtimes_\red\cdualG).
\end{align*}
Finally, it is easy to see that the map
$z\otimes w\mapsto z\circ w$, where $z\in \F(\E,\R)$ and $w\in \E^*\rtimes_\red\cdualG$,
induces an isomorphism
\begin{equation*}
\F(\E,\R)\otimes_{B\rtimes_\red\cdualG}(\E^*\rtimes_\red\cdualG)\cong \cspn\F(\E,\R)(\E^*\rtimes_\red\cdualG).\qedhere
\end{equation*}
\end{proof}

In the group case, it is a basic observation that $A\rtimes_\red G$ appears as a generalized fixed point algebra of $A\otimes\K\bigl(L^2(G)\bigr)$,
where $G$ is a locally compact group and $A$ is a $G$-\cstar{}algebra. Using the result above we can now prove the following generalization:

\begin{proposition} Let $\G$ be a regular locally compact quantum group. Let $\E$ be a Hilbert $B,\G$-module
with an injective coaction of $\G$ and consider the $\G$-\cstar{}algebra $A\otimes\K$, where $A:=\K(\E)$ and $\K:=\K\bigl(L^2(\G)\bigr)$.
Then there is a dense, relatively continuous subspace $\R\sbe A\otimes\K$ such that
$$\F(A\otimes\K,\R)\cong (A\rtimes_\red\cdualG)\otimes L^2(\G)^*,\quad \Fix(A\otimes\K,\R)\cong A\rtimes_\red\cdualG,$$
$$\mbox{and}\quad\I(A\otimes\K,\R)\cong (A\rtimes_\red\cdualG)\otimes\K\cong (A\otimes\K)\rtimes_\red\cdualG.$$
Hence $A\rtimes_\red\cdualG$  appears as a generalized fixed point algebra of $A\otimes\K$.
\end{proposition}
\begin{proof}
Note that $\co_\E$ is injective if and only if $\co_{A}$ is injective. Thus $(A,\co_A)$ is a reduced coaction of $\G$. Since $\G$ is regular we have the duality isomorphism:
$$A\rtimes_\red\cdualG\rtimes_\red\G\cong A\otimes\K.$$
Hence $(A\otimes\K)\rtimes_\red\cdualG\cong (A\rtimes_\red\cdualG)\otimes\K$.
By Proposition~\ref{173}, there is a dense, relatively continuous
subset $\R_0\sbe \E\otimes L^2(\G)$ such that
$$\F(\E\otimes L^2(\G),\R_0)=\E\rtimes_\red\cdualG\cong\E\otimes_B (B\rtimes_\red\cdualG).$$
By Proposition~\ref{171}(ii), $\R:=\spn(|\R_0\>\<\E|)$ is a dense, relatively continuous subspace of
$\K(\E\otimes L^2(\G))\cong A\otimes\K$ and
$$\F(A\otimes\K,\R)\cong \F(\E\otimes L^2(\G),\R_0)\otimes_{B\rtimes_\red\cdualG}(\E\otimes L^2(\G))^*\rtimes_\red\cdualG.$$
Now note that
\begin{align*}
(\E\otimes L^2(\G))^*\rtimes_\red\cdualG & \cong(\E^*\otimes L^2(\G)^*)\otimes_{A\otimes\K}(A\otimes\K)\rtimes_\red\cdualG
\\&\cong(\E^*\otimes L^2(\G)^*)\otimes_{A\otimes\K}(A\rtimes_\red\cdualG)\otimes\K
\\&\cong\bigl(\E^*\otimes_{A}(A\rtimes_\red\cdualG)\bigr)\otimes \bigl(L^2(\G)^*\otimes_{\K}\K\bigr)
\\&\cong(\E^*\rtimes_\red\cdualG)\otimes L^2(\G)^*.
\end{align*}
Thus
$$\F(A\otimes \K,\R)\cong (\E\rtimes_\red\cdualG)\otimes_{B\rtimes_\red\cdualG}(\E^*\rtimes_\red\cdualG)\otimes L^2(\G)^*\cong
(A\rtimes_\red\cdualG)\otimes L^2(\G)^*.$$
Therefore,
$$\Fix(A\otimes \K,\R)\cong A\rtimes_\red\cdualG$$
and
\begin{equation*}
\I(A\otimes\K,\R)\cong(A\rtimes_\red\cdualG)\otimes\K\cong(A\otimes\K)\rtimes_\red\cdualG.\qedhere
\end{equation*}
\end{proof}

In the situation above, we have $A\otimes\K\cong A\rtimes_\red\cdualG\rtimes_\red\G$. Thus $A\otimes\K$ is a dual coaction
and therefore the following result generalizes the proposition above.

\begin{proposition} Let $\G$ be a regular locally compact quantum group and suppose that $\E$ is a Hilbert $B,\G$-module, where
$B$ is a reduced $\G$-\cstar{}algebra. Consider the dual coaction of $\cdualG$ on $\E\rtimes_\red\cdualG$.
Then there is a dense, relatively continuous subspace $\R$ of $\E\rtimes_\red\cdualG$ such that
\begin{align*}
\F(\E\rtimes_\red\cdualG,\R)\cong L^2(\G)^*\otimes\E,\quad\Fix(\E\rtimes_\red\cdualG,\R)\cong\K(\E),\\
\mbox{and}\quad\I(\E\rtimes_\red\cdualG,\R)\cong I\otimes\K\sbe B\otimes \K\cong B\rtimes_\red\cdualG\rtimes_\red\G,
\end{align*}
where $I:=\cspn\braket{\E}{\E}_B\sbe B$ and $\K:=\K\bigl(L^2(\G)\bigr)$. In particular, if $\E$ is full, then $\R$ is saturated.
\end{proposition}
\begin{proof}
Let $A:=\cdualG$, where $\cdualG$ is regarded as a $\cdualG$-\cstar{}algebra (with the comultiplication as coaction).
Since $\G$ is regular, Proposition~\ref{067} implies that
$\R_0:=A_\si$ is a dense, relatively continuous subspace of $A$ and
$$\F(A,\R_0)=(1_\G\otimes L^2(\G)^*)W\cong L^2(\G)^*.$$
The dual coaction on $\E\rtimes_\red\cdualG$ is defined in such way that the canonical nondegenerate $*$-homomorphism
$\pi:A\rightarrow \Ls(\E\rtimes_\red\cdualG)$, $x\mapsto \pi(x):=1\otimes x$, is $\cdualG$-equivariant. So, by Proposition~\ref{171}(i), $\R:=\spn\bigl(\pi(\R_0)(\E\rtimes_\red\cdualG)\bigr)$ is a
dense relatively continuous subspace of $\E\rtimes_\red\cdualG$ and
$$\F(\E\rtimes_\red\cdualG,\R)\cong \F(A,\R_0)\otimes_{A\rtimes_\red\G}(\E\rtimes_\red\cdualG\rtimes_\red\G).$$
Since the coaction on $B$ is reduced and $\G$ is regular, we have $B\rtimes_\red\cdualG\rtimes_\red\G\cong B\otimes\K$, so that
\begin{align*}
\E\rtimes_\red\cdualG\rtimes_\red\G & \cong(\E\otimes_B(B\rtimes_\red\cdualG))\otimes_{B\rtimes_\red\cdualG}(B\rtimes_\red\cdualG\rtimes_\red\G)
\\&\cong\E\otimes_B(B\otimes\K)\cong \E\otimes\K.
\end{align*}
Since $\G$ is regular, we also have $A\rtimes_\red\G=\cdualG\rtimes_\red\G\cong \K$. Thus
$$\F(\E\rtimes_\red\cdualG,\R)\cong L^2(\G)^*\otimes_{\K}(\E\otimes\K)\cong L^2(\G)^*\otimes\E.$$
And therefore
\begin{equation*}
\Fix(\E\rtimes_\red\cdualG,\R)\cong\K(\E)\quad\mbox{and}\quad
\I(\E\rtimes_\red\cdualG,\R)\cong I\otimes\K.\qedhere
\end{equation*}
\end{proof}

\section{Completions of relatively continuous subsets}
\noindent
In general, for a given Hilbert $B,\G$-module $\E$ there might be several relatively continuous subspaces $\R\sbe\E$ yielding the same Hilbert module $\F=\F(\E,\R)$. In this section we impose more conditions on $\R$ to minimize these possibilities. For this we shall use the Banach algebra $L^1_0(\G)\sbe L^1(\G)$ introduced in Section~\ref{409}.

\begin{definition}\label{274} Let $\E$ be a Hilbert $B,\G$-module.
A subspace $\R\sbe \E_\si$ is called \emph{complete}\index{complete}
if $\R$ is $\|\cdot\|_\si$-closed, $L^1_0(\G)$-invariant and also $B$-invariant, that is,
if $\omega*\xi$ and $\xi\cdot b$ belong to $\R$ for all $\xi\in \R$, $\omega\in L^1_0(\G)$ and $b\in B$. Here $*$ denotes the left action of
$L^1(\G)$ on $\E$ induced by the coaction of $\E$ (see Equation~\eqref{483}) and $\cdot$ denotes the right $B$-action.

The \emph{completion}\index{completion} of a subset $\R\sbe \E_\si$, denoted by $\R_\com$, is the smallest complete subspace of $\E_\si$ containing
$\R$.
\end{definition}

Note that $\E_\si$ is complete by Propositions~\ref{003}(ii),~\ref{264} and \cite[Lemma~5.28]{Buss-Meyer:Square-integrable}, and hence $\R$ is complete if and only if $\R$ is an $L^1_0(\G),B$-invariant closed subspace of $\E_\si$. Since the intersection of complete subspaces is clearly complete, the completion of a subset $\R\sbe\E_\si$ always exists and is the intersection of all complete subspaces of $\E_\si$ containing $\R$. If $\R\sbe\E_\si$ is an $L^1_0(\G),B$-invariant subspace, then so is the $\si$-closure $\overline{\R}^\si$ by Proposition~\ref{264} and \cite[Lemma~5.29]{Buss-Meyer:Square-integrable}, and therefore $\R_\com=\overline{\R}^\si$. In general, we can describe the completion of a subset $\R\sbe \E_\si$ as the $\si$-closure of the smallest $L^1_0(\G),B$-invariant subspace of $\E_\si$ containing $\R$.

\begin{proposition}\label{279} Let $\E$ be a Hilbert $B,\G$-module. If $\R\sbe\E$ is relatively continuous,
then so are $\R\cdot B$, $L^1_0(\G)*\R$ and $\overline{\R}^\si$,
and we have
$$\F(\E,\R)=\F(\E,\overline{\R}^\si)=\F(\E,\R\cdot B)=\F(\E,L^1_0(\G)*\R).$$
Moreover, the completion $\R_\com$ of $\R$ is also relatively continuous, and we have
$$\F(\E,\R)=\F(\E,\R_\com).$$
\end{proposition}
\begin{proof}
Let $\xi\in \R$, $b\in B$ and $\omega\in L^1_0(\G)$.
By Propositions~\ref{003}(ii) and~\ref{264}, we have the formulas
$\kket{\xi\cdot b}=\kket{\xi}\co_B(b)$ and $\kket{\omega*\xi}=\kket{\xi}(1_B\otimes\rho_\omega)$.
Since $\co_B(b)$ and $(1_B\otimes\rho_\omega)$ are multipliers of $B\rtimes_\red\cdualG$ (remember that $\rho_\omega\in \cdualG$; see
Proposition~\ref{508}), it follows that $\R\cdot B$ and
$L^1_0(\G)*\R$ are relatively continuous. And from the definition of $\|\cdot\|_\si$
it also follows that $\overline{\R}^\si$ is relatively continuous.

Let $A:=B\rtimes_\red\cdualG$. By definition of
$\|\cdot\|_\si$ and because $\R\sbe \overline{\R}^\si$, we get
$$\F(\E,\R)\sbe \F(\E,\overline{\R}^\si)=\cspn(\kket{\overline{\R}^\si}\circ A)\sbe \cspn(\kket{\R}\circ A)=\F(\E,\R).$$
Thus $\F(\E,\R)=\F(\E,\overline{\R}^\si)$. By Proposition~\ref{003}(ii) and because the linear span of $\co_B(B)A$ is dense in $A$ we get
$$\F(\E,\R\cdot B)=\cspn(\kket{\R}\co_B(B)A)=\F(\E,\R).$$
Analogously, by Proposition~\ref{264}, and because the linear span of $(1\otimes
\rho(L^1_0(\G))A$ is dense in $A$ (by Proposition~\ref{508}), we get that
$\F(\E,L^1_0(\G)*\R)=\F(\E,\R)$. It follows that $\R_\com$ is
relatively continuous as well and $\F(\E,\R_\com)=\F(\E,\R)$.
\end{proof}

Given a complete subspace $\R\sbe\E_\si$, Propositions~\ref{003}(ii) and~\ref{264} imply that
$\overline{\kket{\R}}$ is already a (concrete) $A$-module, where $A:=B\rtimes_\red\cdualG$.
In other words, we have $\overline{\kket{\R}}\circ A\sbe \overline{\kket{\R}}$.
Therefore, if $\R$ is also relatively continuous, then it follows from Equation~\eqref{095} that
$$\F(\E,\R)=\cspn(\kket{\R}\circ A)\sbe \overline{\kket{\R}}\sbe \F(\E,\R),$$
that is, $\F(\E,\R)=\overline{\kket{\R}}$ for any complete, relatively continuous subspace. Combining this with Proposition~\ref{279} we get
\begin{equation}\label{277}
\F(\E,\R)=\overline{\kket{\R_\com}},
\end{equation}
for any relatively continuous subset $\R\sbe\E_\si$.

\begin{corollary}\label{280} For any relatively continuous subset $\R$ of a Hilbert $B,\G$-module $\E$, we have
$$\Fix(\E,\R)=\cspn(\kket{\R_\com}\bbra{\R_\com})\quad
\mbox{and}
\quad\quad \I(\E,\R)=\cspn(\bbraket{\R_\com}{\R_\com}).$$
\end{corollary}

Since the bra-ket operators are $\G$-equivariant we see (again) that
$\Fix(\E,\R)$ is a \cstar{}subalgebra of $\Ls^\G(\E)=\M_1\bigl(\K(\E)\bigr)$.
Proposition~\ref{003}(i) yields the equality
\begin{equation}\label{275}
\Fix(\E,\R)=\cspn\{(\id_{\K(\E)}\otimes\f)(\co_{\K(\E)}(\ket{\xi}\bra{\eta})):\xi,\eta\in \R_\com\}.
\end{equation}

The following result gives a useful criterion to show that a subspace is complete or to calculate its completion.

\begin{proposition}\label{276} Let $\E$ be a Hilbert $B,\G$-module, let $\R$ be a subspace of $\E_\si$
and suppose that $D_0\sbe L^1_0(\G)$ and $B_0\sbe B$ are dense subsets.
\begin{enumerate}
\item[\textup{(i)}] $\R$ is complete if and only if it is $\si$-closed, $D_0*\R\sbe\R$ and $\R\cdot B_0\sbe\R$.
\item[\textup{(ii)}] If $D_0* \R\sbe \overline{\R}^\si$ and $\R\cdot B_0\sbe \overline{\R}^\si$, then
the completion of $\R$ is equal to $\overline{\R}^\si$.
\end{enumerate}
\end{proposition}
\begin{proof} By Proposition~\ref{264} and \cite[Lemma~5.29]{Buss-Meyer:Square-integrable}, the left $L^1_0(\G)$-action and
the right $B$-action on $\E_\si$ are continuous with respect to $\|\cdot\|_\si$. The assertions now follow easily.
\end{proof}

At this point, the following question naturally appears. Let $\R,\R'\sbe\E$ be complete, relatively continuous subspaces and suppose that
$\F(\E,\R)=\F(\E,\R')$. Does it follow that $\R=\R'$? For locally compact groups, that is,
for $\G=\cont_0(G)$, this is in fact true (\cite[Theorem~6.1]{Meyer:Generalized_Fixed}). Unfortunately, this is
not the case for general locally compact quantum groups.
Problems appear for non-co-amenable locally compact quantum groups $\G$.
In these cases, coactions are not necessarily injective. Take any non-injective coaction
$(\E,\co_\E)$ of a locally compact quantum group $\G$. Note that any $\xi\in \ker(\co_\E)$ is square-integrable with $\kket{\xi}=0$.
Thus $\R:=\{0\}$ and $\R':=\ker(\co_\E)$ are different complete, relatively continuous subspaces with
$\F(\E,\R)=\F(\E,\R')=\{0\}$. In order to circumvent this problem we need an extra condition.

\begin{definition}\label{281} Let $\E$ be a Hilbert $B,\G$-module. We say that a complete subspace $\R\sbe \E_\si$
is \emph{slice-complete}\index{slice-complete}, or shortly,
\emph{\slc-complete}\index{\slc-complete} if for all $\xi\in \E_\si$, with $\bbraket{\xi}{\xi}\in B\rtimes_\red\cdualG$, one has
$$\omega*\xi\in \R\mbox{ for all }\omega\in L^1_0(\G)\Longrightarrow \xi\in \R.$$
The \emph{\slc-completion}\index{\slc-completion} of a subset $\R\sbe\E_\si$, denoted by
$\R_\sco$, is the smallest \slc-complete subspace of $\E_\si$
containing $\R$.
\end{definition}

Note that, by definition, $\E_\si$ is
\slc-complete, and intersections of \slc-complete subspaces are again
\slc-complete. Thus the \slc-completion of a subset $\R\sbe
\E_\si$ always exists: it is the intersection of all
\slc-complete subspaces of $\E_\si$ containing $\R$.

Note also that any \slc-complete subspace contains $\ker(\co_\E)$ because $\omega*\xi=0$ for all $\omega\in L^1_0(\G)$
and $\xi\in \ker(\co_\E)$. Thus, if $\co_\E$ is not injective, the trivial subspace $\R=\{0\}$ is complete (and relatively continuous),
but not \slc-complete. The converse is also true:

\begin{proposition}\label{282} Let $\E$ be a Hilbert $B,\G$-module. Then the s-completion of $\{0\}$ is $\ker(\co_\E)$. In particular, $\{0\}$ is
\slc-complete
if and only if $\co_\E$ is injective.
\end{proposition}
\begin{proof} It suffices to show that $\R_0:=\ker(\co_\E)$ is \slc-complete. Of course, $\R_0$ is complete. Now suppose that
$\xi\in \E$ and $\omega*\xi\in\R_0$ for all $\omega\in L^1_0(\G)$, that is,
$$0=\co_\E(\omega*\xi)=\co_\E((\id_\E\otimes\omega)(\co_\E(\xi)))=(\id_\E\otimes\id_\G\otimes\omega)(\co_\E\otimes\id_\G)(\co_\E(\xi)).$$
Since $\omega\in L^1_0(\G)$ is arbitrary, it follows that
$$0=(\co_\E\otimes\id_\G)\bigl(\co_\E(\xi)\bigr)=(\id_\E\otimes\Dt)\bigl(\co_\E(\xi)\bigr).$$
And finally, because $\Dt$ is injective, we get $\co_\E(\xi)=0$,
that is, $\xi\in \ker(\co_\E)=\R_0$. Therefore $\R_0$ is \slc-complete.
\end{proof}

If one restricts to injective coactions, that is, reduced coactions, then it is not clear whether there exist
examples of complete subspaces that are not \slc-complete.

\begin{remark}\label{283} Note that every complete subspace $\R\sbe\E$ satisfies the following property:
for all $\xi\in \E_\si$ with $\bbraket{\xi}{\xi}\in B\rtimes_\red\cdualG$, if $\xi\cdot b\in \R$ for all $b\in B$, then $\xi\in \R$.
In fact, let $(e_i)$ be an approximate unit for $B$. Then $\xi\cdot e_i\to\xi$ and
$\co_B(e_i)\to 1$ strictly in $\M(B\rtimes_\red\cdualG)$.
Now the condition $\bbraket{\xi}{\xi}\in B\rtimes_\red\cdualG$ means that $\R':=\{\xi\}$ is relatively continuous.
Thus $\F:=\F(\E,\R')$ is a (concrete) Hilbert $A$-module, where
$A:=B\rtimes_\red\cdualG$. Thus, by Cohen's Factorization Theorem, for
any $x\in\F$, the map $\M(A)\ni a\mapsto x\cdot a\in \F$ is
continuous for the strict topology on $\M(A)$ and the norm topology
on $\F$. Equation~\eqref{095} says that $\kket{\xi}\in \F$. Thus
$$\kket{\xi\cdot e_i}=\kket{\xi}\circ\co_B(e_i)\to \kket{\xi}.$$
Hence $\xi\cdot e_i\to \xi$ in the $\si$-norm and therefore $\xi\in \R$.
\end{remark}

Note that one important point above was the use of a (bounded) approximate unit for $B$. In order to follow the same idea above and
try to prove the same property for the left $L^1_0(\G)$-action, that is, to
prove that every complete subspace is automatically \slc-complete, one needs a bounded approximate unit for $L^1_0(\G)$, that
is, one needs co-amenability of $\G$. This is the content of the next result.

\begin{proposition}\label{293} Let $\E$ be a Hilbert $B,\G$-module and suppose that $\G$ is co-amenable.
Then every complete subspace $\R\sbe \E_\si$ is automatically \slc-complete.
\end{proposition}
\begin{proof}
Let $(\omega_i)$ be a bounded
approximate unit for $L^1_0(\G)$ (Proposition~\ref{531}). Then $\rho_{\omega_i}\to 1$
strictly in $\M(B\rtimes_\red\cdualG)$. Using Proposition~\ref{358}, one can now follow the same idea as in Remark~\ref{283}.
\end{proof}

Proposition~\ref{293} applies to actions of locally compact groups, that is, to coactions of
$\G=\cont_0(G)$, where $G$ is a locally compact group, because
$\cont_0(G)$ is always co-amenable as a quantum group.
On the other hand, it does not apply to coactions of groups, that is, to the dual $C_\red^*(G)$, unless $G$ is amenable.
Indeed, the quantum group $C_\red^*(G)$ is co-amenable if and only if $G$ is amenable.

\begin{proposition}\label{285} Let $\E$ be a Hilbert $B,\G$-module, and let $\R$ be a complete, relatively continuous subspace of $\E_\si$.
Equipped with the $\si$-norm, $\R$ is a nondegenerate Banach right $B$-module, that is, $\R\cdot B=\R$.
Moreover, if $\G$ is co-amenable, then $\R$ is also a nondegenerate Banach left $L^1_0(\G)$-module, that is, $L^1_0(\G)*\R=\R$.
\end{proposition}
\begin{proof} We already know that $\R$ is a Banach left $L^1_0(\G)$-module and also a Banach right $B$-module. We only
have to prove the nondegeneracy of the actions. Now, if $(e_j)$ and $(\omega_i)$ are bounded approximate units for $B$ and $L^1_0(\G)$,
respectively, then, as we saw in Remark~\ref{283} and Proposition~\ref{293}, we have $\xi\cdot e_j\to \xi$ and
$\omega_i*\xi\to \xi$ with respect to the $\si$-norm, for all $\xi\in\R$.
Therefore, by Cohen's Factorization Theorem, $\R\cdot B=\R$ and $L^1_0(\G)*\R=\R$.
\end{proof}

If $\G$ is not co-amenable, then the conclusion of Proposition~\ref{285}
does not hold in general. A trivial example can be found in the case of non-injective
coactions. In fact, if $(\E,\co_\E)$ is a non-injective coaction,
then $\R:=\ker(\co_\E)$ is relatively continuous and \slc-complete (Proposition~\ref{282}), but
$L^1(\G)*\R=\{0\}$. If one restricts to injective coactions, then it is not clear whether there exist counterexamples.

\begin{proposition}\label{284} Let $\E$ be a Hilbert $B,\G$-module and let $\F\sbe\Ls^{\G}(B\otimes L^2(\G),\E)$ be a concrete Hilbert $A$-module,
where $A:=B\rtimes_\red\cdualG$. Define
$$\R_{\F}:=\{\xi\in\E_\si: \kket{\xi}\in \F\},$$
$$\R_{\F}^0:=\{x(K): x\in \F,\, K\in B\odot \hat\La(\T_{\hat\f})\}.$$
Then $\R_\F^0\sbe\R_\F$, both $\R_\F^0$ and $\R_\F$ are relatively
continuous, $\R_\F$ is complete, and $\kket{\R_\F^0}$ and
$\kket{\R_\F}$ are dense in $\F$. In particular, we have
$$\F(\E,\R_\F^0)=\F(\E,\R_\F)=\F.$$
\end{proposition}
\begin{proof} Let $\R_0:=B\odot \hat\La(\T_{\hat\f})\sbe B\otimes L^2(\G)$. By Proposition~\ref{270}, $\R_0$ is a relatively continuous subset
of $B\otimes L^2(\G)$ and $\kket{\R_0}$ is a dense subspace of $A$. Proposition~\ref{003}(iii) implies that $\R_\F^0\sbe\E_\si$ and
$$\kket{\R_\F^0}=\F\circ\kket{\R_0}\sbe\F\circ A\sbe\F.$$
Thus $\R_\F^0\sbe\R_\F$. This implies that $\kket{\R_\F^0}\sbe\kket{\R_\F}\sbe\F$.
The equation above also shows that $\kket{\R_\F^0}$ is dense in $\F\circ A$
which by Cohen's Factorization Theorem is equal to $\F$. Since $\F^*\circ\F\sbe A$, $\R_\F$
(and therefore $\R_\F^0$) is relatively continuous. Since $\F$ is a
concrete $A$-module and $\E_\si$ is $L^1_0(\G),B$-invariant, it follows from Propositions~\ref{003}(ii) and~\ref{264} that
$\R_\F$ is $L^1_0(\G),B$-invariant as well. From the
definition of $\|\cdot\|_\si$ and because $\E_\si$ is
$\si$-closed, it follows that $\R_\F$ is
$\si$-closed. Thus $\R_\F$ is complete.
\end{proof}

\begin{remark} Proposition~\ref{284} is a quantum version of Meyer's result \cite[Proposition~6.1]{Meyer:Generalized_Fixed} for classical groups.
There is a small difference between our version and Meyer's version in \cite{Meyer:Generalized_Fixed}, namely, the choice of $\R_\F^0$.
For groups, that is, for $\G=\cont_0(G)$, where $G$ is some locally compact group, one can replace $\R_\F^0$ by $\tilde\R_\F^0:=\{x(K): x\in \F,\, K\in \cont_c(G,B)\}$. This set satisfies the same properties of $\R_\F^0$ defined above (this is exactly \cite[Proposition~6.1]{Meyer:Generalized_Fixed}). The point here is that $\cont_c(G,B)$ is also a relatively continuous subset of $L^2(G,B)$ and $\F\bigl(L^2(G,B),\cont_c(G,B)\bigr)=B\rtimes_\red G$
(this is proved in \cite{Meyer:Generalized_Fixed}). For an arbitrary locally compact quantum group $\G$ we may take any relatively continuous
subset $\R_0\sbe B\otimes L^2(\G)$ satisfying $\F\bigl(B\otimes L^2(\G),\R_0\bigr)=B\rtimes_\red\cdualG$ and define $\tilde\R_\F^0:=\{x(K):x\in \F,\, K\in \R_0\}$. An argument analogous to that in the proof of Proposition~\ref{284} shows $\tilde\R_\F^0\sbe\R_\F$ (so that $\tilde\R_\F^0$ is relatively continuous) and $\F(\E,\tilde\R_\F^0)=\F(\E,\R_\F)=\F$. We are going to see later that if $\tilde\R_\F^0$ is chosen in this way, then
the \slc-completion of $\tilde\R_\F^0$ is equal to $\R_\F$. In this sense, all such choices are equivalent.
\end{remark}

\begin{proposition}\label{288} Let $\F\sbe \Ls^\G(B\otimes L^2(\G),\E)$ be a concrete Hilbert $B\rtimes_\red\cdualG$-module, where $\E$ is
some Hilbert $B,\G$-module. Then $\R_\F$ is \slc-complete.
\end{proposition}
\begin{proof} By Proposition~\ref{284}, $\R_\F$ is complete. Suppose that $\xi\in
\E_\si$ is such that $\bbraket{\xi}{\xi}\in B\rtimes_\red\cdualG$ and
$\omega*\xi\in \R_\F$ for all $\omega\in L^1_0(\G)$. By Proposition~\ref{264}, this means that
$$\kket{\omega*\xi}=\kket{\xi}(1_B\otimes \rho_{\omega})\in \F$$ for all $\omega\in L^1_0(\G)$. Since $\rho(L^1_0(\G))$ is
dense in $\cdualG$, there is a bounded approximate unit $(e_i)$ for
$\cdualG$ of the form $e_i=\rho(\omega_i)$ with $\omega_i\in L^1_0(\G)$ for all $i$.
It follows that $(1_B\otimes e_i)$
converges strictly to $1$ in $\M(B\rtimes_\red\cdualG)$. Since
$\bbraket{\xi}{\xi}\in B\rtimes_\red\cdualG$, we get
$$\kket{\omega_i*\xi}=\kket{\xi}\bigl(1_B\otimes \rho(\omega_i)\bigr)\to \kket{\xi}.$$
Therefore $\kket{\xi}\in \F$, that is, $\xi\in \R_\F$ and hence
$\R_\F$ is \slc-complete.
\end{proof}

\section{Continuous square-integrability}
\noindent
Throughout this section, $\G$ denotes a locally compact quantum group and $B$ denotes a $\G$-\cstar{}algebra.
We are ready to give one of the main definitions of this paper.

\begin{definition}\label{288.1}
A \emph{continuously square-integrable Hilbert $B,\G$-module}\index{continuously square-integrable Hilbert $B,\G$-module}
is a pair $(\E,\R)$ consisting of a Hilbert $B,\G$-module $\E$ and a dense, complete, relatively continuous subspace $\R\sbe\E_\si$.
If, in addition, $\R$ is \slc-complete, then we say that $(\E,\R)$ is an
\emph{\slc-continuously square-integrable Hilbert $B,\G$-module}.
\index{\slc-continuously square-integrable Hilbert $B,\G$-module}

The \emph{generalized fixed point algebra}\index{generalized fixed point algebra}
associated to a continuously square-integrable Hilbert $B,\G$-module $(\E,\R)$ is the \cstar{}algebra
$\Fix(\E,\R)=\cspn\kket{\R}\bbra{\R}$.
\end{definition}

By Equation~\eqref{275}, the generalized fixed point algebra $\Fix(\E,\R)$ is
the closed linear span of the ``averages'' $(\id_{\K(\E)}\otimes\f)(x)$ where
$x=\ket{\xi}\bra{\eta}$ with $\xi,\eta\in \R$. Note that in the group case $\G=\cont_0(G)$, the average $(\id_{\K(\E)}\otimes\f)(x)$ is the
strict-unconditional integral $\int_G^\su\alpha_t(x)\dd{t}$ (as defined in \cite{Exel:Unconditional}), where $\alpha$ is the corresponding action of $G$ on $\K(\E)$. In particular, $\Fix(\E,\R)$ is contained in the multiplier fixed point algebra $\M_1\bigl(\K(\E)\bigr)=\{x\in \M\bigl(\K(\E)\bigr):\co_{\K(\E)}(x)=x\otimes 1\}$ and thus elements in $\Fix(\E,\R)$ are fixed by the coaction of $\K(\E)$. Proposition~\ref{111} tell us that $\Fix(\E,\R)$ is Morita equivalent to the ideal $\I(\E,\R)=\cspn\bbraket{\R}{\R}\sbe B\rtimes_\red\cdualG$, where $\F(\E,\R)$ is the canonical candidate for the imprimitivity Hilbert module.

In what follows, we are going to generalize \cite[Theorem~6.1]{Meyer:Generalized_Fixed} to the setting of locally compact quantum groups. This result describes relatively continuous subspaces via concrete Hilbert modules. First we need a preliminary result. Recall that $\sigma$ denotes the modular automorphism group and $\T_\f$ denotes the Tomita \Star{}algebra of the left Haar weight $\f$ (see Equation~\eqref{501}).

\begin{lemma}\label{287} Let $\E$ be a Hilbert $B,\G$-module. Let $b\in \T_\f$, $\xi\in \E_\si$ and suppose that $a\in \dom(\La)$
is such that $\La(a)\in \dom(\mo^{\frac{1}{2}})$. Define $\omega:=\omega_{\La(b),\La(a)}=a\f b^*\in L^1(\G)$ and $x_\omega:=a\sigma_{-\ii}(b^*)\in
\dom(\La)$.
Then $\omega*\xi\in \E_\si$ and
$$\|\omega*\xi\|_\si\leq c_\omega\|\kket{\xi}\|,$$
where $c_\omega:=\max\{\|\Lambda(x_\omega)\|,\|\rho(\omega)\|\}$.
\end{lemma}
\begin{proof}
Lemma~5.17 in \cite{Buss-Meyer:Square-integrable} implies $\omega*\xi=\kket{\xi}\bigl(1_B\otimes \Lambda(x_\omega)\bigr)$,
so that $\|\omega*\xi\|\leq \|\La(x_\omega)\|\|\kket{\xi}\|$. Proposition~\ref{264} says that $\omega*\xi\in \E_{\si}$ and
$$\kket{\omega*\xi}=\kket{\xi}(1_B\otimes\rho_\omega).$$
Hence $\|\kket{\omega*\xi}\|\leq \|\rho_\omega\|\|\kket{\xi}\|$. The desired result now follows.
\end{proof}

We are ready to proof one of our main results.

\begin{theorem}\label{289} Let $\G$ be a locally compact quantum group, and let $\E$ be a Hilbert $B,\G$-module.
Then the map $\R\mapsto \F(\E,\R)$ is a bijection between \slc-complete, relatively continuous subspaces of $\E$ and concrete Hilbert
$B\rtimes_\red\cdualG$-modules $\F\sbe \Ls^\G(B\otimes L^2(\G),\E)$. Its inverse is the map $\F\mapsto \R_\F$.
A concrete Hilbert module $\F$ is essential if and only if $\R_\F$ is dense in $\E$.
\end{theorem}
\begin{proof} By Proposition~\ref{288}, $\R_\F$ is relatively continuous and \slc-complete,
so that the map $\F\mapsto \R_\F$ is well-defined.
By Proposition~\ref{284} we have $\F(\E,\R_\F)=\F$. It remains to show that
$\R=\R_{\F(\E,\R)}$ for every \slc-complete, relatively continuous
subspace $\R$ of $\E$. By Equation~\eqref{095}, we have $\R\sbe
\R_{\F(\E,\R)}$. Let $\xi\in \R_{\F(\E,\R)}$. Then, by definition of $\R_{\F(\E,\R)}$, we have $\kket{\xi}\in \F(\E,\R)=\overline{\kket{\R}}$
(for the last equality we have used Equation~\eqref{277} and the assumption that $\R$ is complete).
Thus there is $\xi_n\in \R$ such that $\kket{\xi_n}\to \kket{\xi}$. Take any $a\in \dom(\La)$ and $b\in \T_\f$
such that $\La(a)\in \dom(\mo^{\frac{1}{2}})$ and define
$\omega:=a\f b^*=\omega_{u,v}\in
L^1_0(\G)$, where $u:=\La(b)$ and $v:=\La(a)$. By Lemma~\ref{287}, we have
$\|\omega*\eta\|_\si\leq c_\omega\|\kket{\eta}\|$ for all
$\eta\in \E_\si$, where $c_\omega$ is a constant depending only on
$\omega$. In particular,
$$\|\omega*\xi-\omega*\xi_n\|_\si\leq c_\omega\|\kket{\xi}-\kket{\xi_n}\|\to 0.$$
Since $\R$ is complete, we get that $\omega*\xi\in \R$.
Thus $\omega_{\La(b),\La(a)}*\xi\in \R$ for all $b\in \T_\f$ and $a\in \dom(\La)$ with $\La(a)\in
\dom(\mo^{\frac{1}{2}})$. Now take any $u\in H$ and $v\in \dom(\mo^{\frac{1}{2}})$ and observe that
$\dom(\mo^{\frac{1}{2}})\cap\ran(\La)$ is a core for $\mo^{\frac{1}{2}}$ (see the proof of \cite[Proposition~1.9.5]{Vaes:Thesis}).
This means that there is a sequence $(a_n)\sbe \dom(\La)$ with $\La(a_n)\in
\dom(\mo^{\frac{1}{2}})$ such that $\La(a_n)\to v$ and
$\mo^{\frac{1}{2}}(\La(a_n))\to \mo^{\frac{1}{2}}v$. By Lemma~5.13 in \cite{Kustermans:KMS}, $\La(\T_\f)$ is dense in $H$, so there is a sequence $(b_n)\sbe\T_\f$ such that $\La(b_n)\to u$. It follows that
$\omega_{\La(b_n),\La(a_n)}\to \omega_{u,v}$ in $L^1(\G)$
and
$$\rho(\omega_{\La(b_n),\La(a_n)})=(\id\otimes\omega_{\La(b_n),\mo^{\frac{1}{2}}\La(a_n)})(V^*)
\to (\id\otimes\omega_{u,\mo^{\frac{1}{2}}v})(V^*)=\rho(\omega_{u,v}).$$
Proposition~\ref{264} implies that $\omega_{\La(b_n),\La(a_n)}*\xi\to
\omega_{u,v}* \xi$ in the $\si$-norm. Thus $\omega*\xi\in
\R$ for all $\omega\in L^1_{00}(\G)$ and hence also for all $\omega\in L^1_0(\G)$ because $L^1_{00}(\G)$ is dense in $L^1_0(\G)$ and $\R$ is a Banach left $L^1_0(\G)$-module. Since $\R$ is \slc-complete we conclude that $\xi\in \R$. Therefore $\R=\R_{\F(\E,\R)}$.

If $\F$ is essential, then by the definition of $\R_\F^0$ (see Proposition~\ref{284}), the linear span of $\R_\F^0$ is dense in $\E$.
Thus $\R_\F\supseteq\R_\F^0$ is dense in $\E$ as well.
Conversely, if $\R_\F$ is dense, then $\F$ is essential by Proposition~\ref{079}.
\end{proof}

\begin{corollary}\label{293.1} Suppose $\G$ is a compact quantum group and $\E$ is a Hilbert $B,\G$-module.
Then the map $\R\mapsto\F(\E,\R)$ is a bijection between $L^1(\G),B$-invariant closed subspaces of $\E$ satisfying
\begin{equation}\label{499}
\xi\in\E\mbox{ and }\omega*\xi\in \R,\,\forall\,\omega\in L^1(\G)\Longrightarrow \xi\in \R,
\end{equation}
and concrete Hilbert $B\rtimes_\red\cdualG$-modules $\F\sbe\Ls^\G(B\otimes L^2(\G),\E)$. The inverse map is
given by $\F\to\R_\F$.
\end{corollary}
\begin{proof} Since $\G$ is compact, any subset of $\E$ is relatively continuous and the $\si$-norm is equivalent to the norm of $\E$.
Thus $\R\sbe\E$ is complete if and only if it is an $L^1(\G),B$-invariant closed subspace of $\E$ (here we are using that $\G$ is unimodular
so that $L^1(\G)=L^1_0(\G)$). Such a subspace is \slc-complete if and only if it satisfies~\eqref{499}. Thus the assertion is a special
case of Theorem~\ref{289}.
\end{proof}

\begin{corollary}\label{291} Let $\E$ be a Hilbert $B,\G$-module, where $\G$ is a locally compact quantum group, and suppose that $\R$ is a
relatively continuous subset of $\E$. Then the \slc-completion of $\R$ is equal to $\R_{\F(\E,\R)}$. In particular, the \slc-completion
of a relatively continuous subset is also relatively continuous.
\end{corollary}
\begin{proof} Let $\R_\sco$ be the \slc-completion of $\R$. By Proposition~\ref{288},
$\R_{\F(\E,\R)}$ is relatively continuous and \slc-complete and we have $\R\sbe
\R_{\F(\E,\R)}$. Thus $\R_\sco\sbe \R_{\F(\E,\R)}$. In particular,
$\R_\sco$ is relatively continuous. Now it is easy to see that the maps
$\R\mapsto \F(\E,\R)$ and $\F\mapsto \R_\F$ preserve inclusion.
Thus $\R\sbe \R_\sco$ implies $\F(\E,\R)\sbe \F(\E,\R_\sco)$. Since $\R_\sco$ is relatively continuous and
\slc-complete, Theorem~\ref{289} yields
$\R_{\F(\E,\R)}\sbe\R_{\F(\E,\R_\sco)}=\R_\sco$. Therefore $\R_{\F(\E,\R)}=\R_\sco$.
\end{proof}

\begin{corollary}\label{295} Let $\E$ be a Hilbert $B,\G$-module and suppose that $\G$ is co-amenable.
Let $\R\sbe \E_\si$ be some relatively continuous subset.
Then $\R_{\F(\E,\R)}$ is the completion of $\R$. In particular, $\R$ is complete if and only if $\R$ is equal to $\R_{\F(\E,\R)}$.
\end{corollary}
\begin{proof} This follows from Proposition~\ref{293} and Corollary~\ref{291}.
\end{proof}

The result above implies, in particular, that our definition of completeness
of relatively continuous subsets is equivalent to
\cite[Definition~6.2]{Meyer:Generalized_Fixed} in the case of groups (see \cite[Proposition~6.3]{Meyer:Generalized_Fixed}).

The following result gives a description of the \slc-completion of a relatively continuous subset.

\begin{proposition}\label{292} Let $\E$ be a Hilbert $B,\G$-module, where $\G$ is a locally compact quantum group and let $\R$ be a
relatively continuous subset of $\E$. Then the \slc-completion of $\R$ is given by
$$\R_\sco=\{\xi\in\E_\si:\bbraket{\xi}{\xi}\in B\rtimes_\red\cdualG\mbox{ and } \omega*\xi\in \R_\com\mbox{ for all }\omega\in L^1_0(\G)\},$$
where $\R_\com$ denotes the completion of $\R$.
\end{proposition}
\begin{proof} Suppose that
$\xi\in \R_\sco=\R_{\F(\E,\R)}$. By Equation~\eqref{277}, we have
$\F(\E,\R)=\overline{\kket{\R_\com}}$ and hence there is a
sequence $(\xi_n)$ in $\R_\com$ such that $\kket{\xi_n}\to
\kket{\xi}$. As in the proof of Theorem~\ref{289}, this implies that
$\omega*\xi\in \R_\com$ for all $\omega\in L^1_0(\G)$. Thus
$$\R_\sco\sbe\{\xi\in\E_\si:\bbraket{\xi}{\xi}\in B\rtimes_\red\cdualG\mbox{ and } \omega*\xi\in \R_\com\mbox{ for all }\omega\in L^1_0(\G)\}.$$
And the other inclusion is trivial because $\R_\sco$ is
\slc-complete and $\R_\com\sbe \R_\sco$.
\end{proof}

\begin{corollary}\label{258} Let $\E$ be a Hilbert $B,\G$-module and let $\R\sbe\E$ be a relatively continuous subset. Then
$$\F(\E,\R)=\F(\E,\R_\sco).$$
\end{corollary}
\begin{proof} The inclusion $\R\sbe \R_\sco$ implies $\F(\E,\R)\sbe\F(\E,\R_\sco)$. On
the other hand, if $\xi\in \R_\sco$, then by the description of
$\R_\sco$ in Proposition~\ref{292}, we have $\omega*\xi\in \R_\com$ for all
$\omega\in L^1_0(\G)$. Thus
$$\kket{\omega*\xi}=\kket{\xi}(1_B\otimes \rho_{\omega})\in \kket{\R_\com}\sbe \F(\E,\R_\com)=\F(\E,\R)$$ for all $\omega\in L^1_{00}(\G)$. Now taking
a bounded approximate unit
$(e_i)$ for $\cdualG$ of the form $e_i=\rho({\omega_i})$, where
$\omega_i\in L^1_0(\G)$ for all $i$, it follows that $\kket{\xi}\in
\F(\E,\R)$. Therefore
\begin{equation*}
\F(\E,\R_\sco)=\overline{\kket{\R_\sco}}\sbe \F(\E,\R).\qedhere
\end{equation*}
\end{proof}

\section{Functoriality and naturality}
\noindent
Throughout this section we fix a locally compact quantum group $\G$ and \cstar{}algebra $B$ with a continuous coaction of $\G$, that is, a
$\G$-\cstar{}algebra.

\begin{definition}\label{297} Let $(\E_1,\R_1)$ and  $(\E_2,\R_2)$ be continuously square-integrable Hilbert $B,\G$-modules.
An operator $T\in \Ls^\G(\E_1,\E_2)$ is called \emph{$\R$-continuous}\index{$\R$-continuous operator}
if $T(\R_1)\sbe \R_2$ and $T^*(\R_2)\sbe \R_1$.
\end{definition}

Given a locally compact quantum group $\G$ and a $\G$-\cstar{}algebra $B$, the continuously square-integrable
Hilbert $B,\G$-modules form a category with $\R$-continuous adjointable operators as morphisms.
The s-continuously square-integrable Hilbert $B,\G$-modules form a full subcategory.
By Proposition~\ref{293}, these categories are identical if $\G$ is co-amenable.

In what follows, we analyze the functoriality of our constructions.

\begin{proposition}\label{497} Let $\G$ be a locally compact quantum group and let $B$ be a $\G$-\cstar{}algebra.
The construction $(\E,\R)\mapsto \F(\E,\R)$ is a functor from the category of continuously square-integrable Hilbert $B,\G$-modules to
the category of Hilbert $B\rtimes_\red\cdualG$-modules.
\end{proposition}
\begin{proof}
Given an $\R$-continuous $\G$-equivariant operator $T:(\E_1,\R_1)\to (\E_2,\R_2)$,
the associated adjointable operator $\tilde T:\F(\E_1,\R_1)\to \F(\E_2,\R_2)$ is given by
$\tilde T(x)=T\circ x$ for all $x\in \F(\E_1,\R_1)\sbe\Ls^\G(B\otimes H,\E_1)$.
Here one uses that $\kket{T(\xi)}=T\circ\kket{\xi}$ for all $\xi\in \E_\si$ and the fact that $\F(\E_k,\R_k)$ is the closure of $\kket{\R_k}$,
for $k=1,2$. Since $T(\R_1)\sbe\R_2$, this ensures that $\tilde T$ is a map $\F(\E_1,\R_1)\to\F(\E_2,\R_2)$.
In the same way, since $T^*(\R_2)\sbe\R_1$, the operator $\tilde T$ is, in fact, adjointable and its adjoint is given by $\tilde T^*(y)=T^*\circ y$ for
all $y\in \F(\E_2,\R_2)$.
\end{proof}

Given an abstract Hilbert $B\rtimes_\red\cdualG$-module $\F$, we can, by \cite[Theorem~5.1]{Meyer:Generalized_Fixed},
identify it with the concrete Hilbert $B\rtimes_\red\cdualG$-module
$T(\F)\sbe\Ls^\G(B\otimes H,\E_\F)$, where $\E_\F:=\F\otimes_A(B\otimes H)$, $A:=B\rtimes_\red\cdualG$,
and $T:\F\to\Ls^\G(B\otimes H,\E_\F)$ is the canonical representation defined in \cite[Equation~(32)]{Meyer:Generalized_Fixed}. Recall that $T(x)f=x\otimes_A f$ for all $x\in\F$ and
$f\in B\otimes H$. In this way we get an \slc-complete, relatively continuous subspace $\R_\F\sbe\E_\F$ as in Proposition~\ref{284} by
$$\R_\F:=\{\xi\in \E_\F:\xi \mbox{ is square-integrable and }\kket{\xi}\in T(\F)\}.$$
Since $\F$ is essential, $\R_\F$ is dense $\E_\F$. In fact, note that by Theorem~\ref{289} and Proposition~\ref{284},
$\R_\F$ is the \slc-completion of the linear span of $T(\F)\bigl(B\odot\hat\La(\T_{\hat\f})\bigr)$
and this linear span equals $\F\odot_A\bigl(B\odot\hat\La(\T_{\hat\f})\bigr)$.
Thus the pair $(\E_\F,\R_\F)$ is a \slc\nobreakdash-continuously square-integrable
Hilbert $B,\G$-module.

\begin{lemma}\label{510} With notation as above we have $\F(\E_\F,\R_\F)=T(\F)$.
\end{lemma}
\begin{proof}
Since $T(x)\in \Ls^\G(B\otimes H,\E)$ we have
$x\otimes_A\zeta=T(x)\zeta\in \E_\si$ for all $x\in \F$ and $\zeta\in (B\otimes H)_\si$ and
$\kket{x\otimes_A\zeta}=\kket{T(x)\zeta}=T(x)\circ\kket{\zeta}$. It follows that
\begin{align*}
\F(\E_\F,\R_\F)&=\cspn(\kket{\R}\circ A)
	\\&=\cspn(\kket{\F\odot_A\R_0}\circ A)
	\\&=\cspn(T(\F)\circ\kket{\R_0}\circ A)
	\\&=\cspn(T(\F)\circ A)=T(\F).\qedhere
\end{align*}
\end{proof}

\begin{proposition}\label{498} The construction $\F\mapsto (\E_\F,\R_\F)$ is a functor from the category of Hilbert $A$-modules to
the category of \slc-continuously square-integrable Hilbert $B,\G$-modules.
\end{proposition}
\begin{proof} First, observe that the map $\F\mapsto\E_\F$ is functorial: to an adjointable operator $S:\F_1\to\F_2$
we associate the $\G$-equivariant adjointable operator $S\otimes_A\id:\E_1\to\E_2$, where $\E_k:=\F_k\otimes_A(B\otimes H)$, $k=1,2$.
It remains to show that $S\otimes_A\id$ is $\R$-continuous, that is,
$(S\otimes_A\id)(\R_1)\sbe\R_2$ and $(S\otimes_A\id)^*(\R_2)\sbe\R_1$, where $\R_k:=\R_{\F_k}$, $k=1,2$.
Since $(S\otimes_A\id)^*=S^*\otimes_A\id$, it is enough to show that $(S\otimes_A\id)(\R_1)\sbe\R_2$. Let $T_k:\F_k\to\Ls^\G(B\otimes H,\E_k)$ be the
canonical representation of
$\F_k$, that is, $T_k(x)f=x\otimes_A f$ for all $x\in \E_k$ and $f\in B\otimes H$. Note that $(S\otimes_A\id)\circ T_1(x)=T_2\bigl(S(x)\bigr)$
for all $x\in \F_1$. Thus $(S\otimes_A\id)\circ T_1(\F_1)\sbe\ T_2(\F_2)$. Combining this with the relation
$\kket{(S\otimes_A\id)\xi}=(S\otimes_A\id)\circ\kket{\xi}$, for every square-integrable element $\xi\in \E_1$ (see Proposition~\ref{003}(iii)),
the desired result follows.
\end{proof}

\begin{corollary}\label{495} Let $\G$ be a locally compact quantum group and let $B$ be a $\G$-\cstar{}algebra.
Isomorphism classes of Hilbert modules over $A:=B\rtimes_\red\cdualG$ correspond bijectively to
isomorphism classes of \slc-continuously square-integrable Hilbert $B,\G$-modules via the functors
\begin{equation}\label{496}
(\E,\R)\mapsto \F(\E,\R)\quad\mbox{and}\quad \F\mapsto (\E_\F,\R_\F),
\end{equation}
where $\E_\F:=\F\otimes_A(B\otimes H)$ and $\R_\F$ is the \slc-completion of $\F\odot_A\bigl(B\odot\hat\La(\T_{\hat\f})\bigr)$.
\end{corollary}
\begin{proof}
The maps in~\eqref{496} are considered between isomorphism classes
and are well-defined by Propositions~\ref{497} and~\ref{498}.
To prove that they are inverse to each other, let $(\E,\R)$ be an \slc-continuously square-integrable Hilbert $B,\G$-module and define
$\F:=\F(\E,\R)$. We have to prove that $(\E_\F,\R_\F)\cong(\E,\R)$. Define $U:\E_\F\to\E$ by $U(x\otimes_A\zeta):=x(\zeta)$ for all
$x\in \F\sbe\Ls^\G(B\otimes H,\E)$ and $\zeta\in B\otimes H$. The unitary $U$ appears in Theorem~\ref{289} and is $\G$-equivariant. It
remains to show that $U(\R_\F)=\R$. Since $\R_\F,\R\sbe\E$ are relatively continuous and \slc-complete, it is enough to show that
$\F\bigl(\E,U(\R_\F)\bigr)=\F(\E,\R)=\F$. Note that $U(\R_\F)$ is the \slc-completion of
$U(\F\odot_A\R_0)=\spn\F(\R_0)$, where $\R_0:=B\odot\hat\La(\T_{\hat\f})$.
Since $\F\sbe\Ls^\G(B\otimes H,\E)$ and $\overline{\kket{\R_0}}=A$, we have $\F(\E,\R_\F)=\cspn(\F\circ A)=\F$.
Therefore $(\E,\R)\cong(\E_\F,\R_\F)$. Now assume that $\F$ is a Hilbert $A$-module and define $(\E,\R):=(\E_\F,\R_\F)$. We have to show that
$\F\cong\F(\E,\R)$. Let $T:\F\to\Ls^\G(B\otimes H,\E_\F)$ be the canonical representation of $\F$ as defined in \cite[Equation~(32)]{Meyer:Generalized_Fixed}, that is, $T(x)\zeta:=x\otimes_A\zeta$ for all $x\in \F$ and $\zeta\in B\otimes H$.
By \cite[Theorem~5.1]{Meyer:Generalized_Fixed}, $\F\cong T(\F)$ as Hilbert $A$-modules, and by Lemma~\ref{510}, $T(\F)=\F(\E,\R)$, so that $\F\cong T(\F)=\F(\E,\R)$.
\end{proof}

Finally, we prove that our constructions are natural and yield an equivalence
between the respective categories.

\begin{theorem}\label{298} Let $\G$ be a locally compact quantum group, and let $B$ be a $\G$-\cstar{}algebra.
Let $(\E,\R)$ be an \slc-continuously square-integrable Hilbert $B,\G$-module, and let $\F:=\F(\E,\R)$.
Then there is a canonical, injective, strictly continuous $*$-homomorphism $\phi:\Ls(\F)\to \Ls^{\G}(\E)$,
whose range is the space of $\R$-continuous operators. It maps $\K(\F)$ isometrically onto $\Fix(\E,\R)$.

The categories of \slc-continuously square-integrable Hilbert $B,\G$-modules and Hilbert modules over $B\rtimes_\red\cdualG$
are equivalent via the functors $(\E,\R)\mapsto \F(\E,\R)$ and $\F\mapsto (\E_\F,\R_\F)$.

The generalized fixed point algebra $\Fix(\E,\R)$ is the closed linear span of the operators
$(\id_{\K(\E)}\otimes\f)(\ket{\xi}\bra{\eta})\in \Ls^\G(\E)$, $\xi,\eta\in \R$, and it is Morita equivalent to the ideal
$\I(\E,\R)=\cspn\{\bbraket{\xi}{\eta}:\xi,\eta\in \R\}$ of $B\rtimes_\red\cdualG$.
\end{theorem}
\begin{proof} Since $\R$ is \slc-complete, we have $\R=\R_\F$ and $\F=\overline{\kket{\R}}$ (by Corollary~\ref{291} and Equation~\eqref{277}).
These facts together with Proposition~\ref{003}(iii) imply that the set $M$ in \cite[Theorem~5.2]{Meyer:Generalized_Fixed} equals the set
of $\R$-continuous operators. Therefore, the same $\phi$ of \cite[Theorem~5.2]{Meyer:Generalized_Fixed} yields the first statement. Combining this with Theorems~\ref{289} and~\cite[Theorem~5.1]{Meyer:Generalized_Fixed}, we get the second statement. The last statement follows from Equation~\eqref{275} and Proposition~\ref{111}.
\end{proof}

\begin{corollary} Let $\G$ be a compact quantum group and suppose that $B$ is a $\G$-\cstar{}algebra. Then the functor
$$\F\mapsto \F\rot{B\rtimes_\red\cdualG}\bigl(B\otimes L^2(\G)\bigr)$$
is an equivalence between the categories of Hilbert $B\rtimes_\red\cdualG$-modules and Hilbert $B,\G$-modules. In other words, any Hilbert
$B,\G$-module appears in this way for a unique Hilbert module $\F$ over $B\rtimes_\red\cdualG$ and the map $\Ls(\F)\to\Ls^\G(\E)$ is
an isomorphism.

Given a Hilbert $B,\G$-module $\E$, the generalized fixed point algebra associated to $\E$ is the usual fixed point algebra:
$$\Fix(\E)=\{x\in \K(\E):\co(x)=x\otimes 1\}\cong\K(\F_\E),$$
where $\F_\E=\F(\E,\E)$ is the Hilbert $B\rtimes_\red\cdualG$-module associated to $\E$.
\end{corollary}
\begin{proof} If $\G$ is compact, then any Hilbert $B,\G$-module is continuously square-integra\-ble and $\R=\E$ is the unique dense, complete,
relatively continuous subspace. Therefore there is no difference between the categories of continuously (and hence also \slc-continuously)
square-integrable Hilbert $B,\G$-modules and arbitrary Hilbert $B,\G$-modules. The assertions now follow from Theorem~\ref{298}.
\end{proof}

In particular, for compact quantum groups every Hilbert $B,\G$-module is ``proper" in the following sense:

\begin{definition}
We say that a Hilbert $B,\G$-module $\E$ is \emph{$\R$-proper}\index{$\R$-proper}
if there is a unique dense, \slc-complete, relatively continuous subspace of $\E$.
\end{definition}

By Theorem~\ref{289}, $\E$ is $\R$-proper if and only if there is a unique concrete, essential Hilbert $B\rtimes_\red\cdualG$-module
$\F\sbe\Ls^\G(B\otimes L^2(\G),\E)$.

Recall that in the group case $\G=\cont_0(G)$, a $G$-\cstar{}algebra $A$ is called \emph{spectrally proper}\index{spectrally proper}
if the canonical induced action of $G$ on the primitive ideal space $\Prim(A)$ is proper (see \cite[Definition~9.2]{Meyer:Generalized_Fixed}). This class includes all the proper $G$-\cstar{}algebras in the sense of Kasparov~\cite{Kasparov:Novikov}. By Theorem~9.1 in \cite{Meyer:Generalized_Fixed}, every Hilbert module over a spectrally proper $G$-\cstar{}algebra is $\R$-proper. In particular, a commutative
$G$-\cstar{}algebra $\cont_0(X)$ (where $X$ is a locally compact $G$-space) is $\R$-proper if $X$ is a proper $G$-space.
Conversely, if $\cont_0(X)$ is $\R$-proper, then it is, in particular, integrable and therefore, by Rieffel's Theorem~4.7 in \cite{Rieffel:Integrable_proper}, $X$ is a proper $G$-space.

In the general quantum setting, unless $\G$ is compact, it is not easy to find non-trivial examples of
$\R$-proper Hilbert modules. In this direction, we have the following result:

\begin{proposition}\label{537} Let $\G$ be a locally compact quantum group and let $\G$ coact on itself by the comultiplication.
Then $\G$ is an $\R$-proper $\G$-\cstar{}algebra
if and only if $\G$ is semi-regular, that is, $\K\bigl(L^2(\G)\bigr)$ is contained in $C:=\cspn(\G\cdualG)\cong\G\rtimes_\red\cdualG$.
In this case, $\R=\G_\si$ is the unique dense, \slc-complete, relatively continuous subspace of $\G$.
The Hilbert $\G\rtimes_\red\cdualG$-module $\F(\G,\R)$ is isomorphic to
the dual $L^2(\G)^*$ of $L^2(\G)$ considered as a Hilbert $\C$-module in the canonical way.
In particular, $\Fix(\G,\R)\cong\C$ and $\I(\G,\R)\cong\K\bigl(L^2(\G)\bigr)$.
The quantum group $\G$ is regular if and only if $\R$ is saturated.
\end{proposition}
\begin{proof} By Equation~\eqref{eq:FormulaBraOperatorA=G}, we have
$$\kket{\xi}=\bigl(1_\G\otimes\La(\xi^*)^*\bigr)W,\quad \bbra{\xi}=W^*\bigl(1_\G\otimes\La(\xi^*)\bigr),\quad\mbox{for all }\xi\in \G_\si=\dom(\La)^*.$$
Since $\G$ is semi-regular, $\K(H)\sbe C\defeq \cspn \G\cdualG$ and hence $W^*(1\otimes\K(H))W\sbe W^*(1\otimes C)W=\G\rtimes\cdualG$ (see Equation~\eqref{506}). Thus, if $\G$ is semi-regular, any subset of $\G_\si$ is relatively continuous.
Now note that if $\R_0\sbe\G_\si$ is a complete subspace, then
$$\F(\G,\R_0)=\overline{\kket{\R_0}}=\overline{\bigl(1_\G\otimes\La(\R_0^*)^*\bigr)W}=(1_\G\otimes H_0^*)W,$$
where $H_0:=\overline{\La(\R_0^*)}$ (which is a closed subspace of $H=L^2(\G)$). Equivalently,
$$\F(\G,\R_0)^*=\overline{\bbra{\R_0}}=W^*(1_\G\otimes H_0).$$
In particular, $\F(\G,\R)^*=W^*(1_\G\otimes H)$.
Define the following linear map
$$T:\ran(\La)\sbe L^2(\G)\to \bbra{\R}\sbe \F(\G,\R)^*,\quad T\bigl(\La(\xi)\bigr):=W^*\bigl(1_\G\otimes\La(\xi)\bigr).$$
Then (identifying $\C\cong\C1_\G\sbe\M(\G)$)
\begin{align*}
\bigl\<T(\La(\xi))|T(\La(\eta))\bigr\>&=\bigl(W^*(1_\G\otimes\La(\xi))\bigr)^*W^*\bigl(1_\G\otimes\La(\eta)\bigr)
\\&=\bigl(1_\G\otimes\La(\xi)^*\bigr)WW^*\bigl(1_\G\otimes\La(\eta)\bigr)
\\&=1_\G\f(\xi^*\eta)=\<\La(\xi)|\La(\eta)\>_{L^2(\G)}.
\end{align*}
It follows that $T$ extends to an isomorphism
$L^2(\G)\cong\F(\G,\R)^*$ (as Hilbert spaces). Thus $\F(\G,\R)\cong L^2(\G)^*$ as Hilbert modules over
$\K\bigl(L^2(\G)\bigr)$ and hence also as Hilbert modules over $C$.

Finally, suppose that $\R_0\sbe\G_\si$ is dense and \slc-complete. Then
$$\I(\G,\R_0)=\cspn\bigl\{W^*(1_\G\otimes\ket{\xi}\bra{\eta})W:\xi,\eta\in\La(\R_0^*)\bigr\}=W^*\bigl(1_\G\otimes\K(H_0)\bigr)W.$$
The subset $\I(\G,\R_0)\sbe
W^*\bigl(1_\G\otimes\K(H)\bigr)W\sbe \G\rtimes_\red\cdualG$ is an ideal of
$\G\rtimes_\red\cdualG$ and hence also of $W^*\bigl(1_\G\otimes\K(H)\bigr)W$. It
follows that $\K(H_0)$ is an ideal of $\K(H)$. Since $\K(H)$ is simple, we get $\K(H_0)=\K(H)$ ($H_0$ is not
zero because $\R_0$ is dense in $\G$). Cohen's Factorization Theorem yields $H=\K(H)H=\K(H_0)H=H_0$. Hence
$$\F(\G,\R_0)=(1_\G\otimes H_0^*)W=(1_\G\otimes H^*)W=\F(\G,\R).$$
Therefore, $\R_0=\R$ because both $\R_0$ and $\R$ are \slc-complete. The last assertion was already proved in Proposition~\ref{067}.
\end{proof}

\begin{remark} Examples of non-semi-regular quantum groups have been constructed in \cite{Baaj-Skandalis-Vaes:Non-semi-regular}. It has been observed there that for such examples the coaction of $\G$ on itself via the comultiplication is in some sense not ``proper". We can now give this statement a precise meaning if we agree that ``proper" means $\R$-proper.

Moreover, if we agree that a proper (that is, $\R$-proper) coaction is ``free''
if the corresponding dense, \slc-complete, relatively continuous subspace is, in addition, saturated,
then we can also say that the comultiplication of a locally compact quantum group is proper and free if and only if it is regular.
\end{remark}

\vskip 0,5pc

\end{document}